%% file: infinitecayley_arxiv.tex
\documentclass[11pt]{article}
\usepackage{etex}
\usepackage{cite}
\include{psfig}
\usepackage{tabularx}
\usepackage{longtable}
\usepackage{pictex, latexsym, graphicx,amsmath,amssymb,amsbsy,amsfonts,amsthm,verbatim}
\usepackage[pdfpagemode=UseOutlines,colorlinks=true,pdfnewwindow=true,citecolor=blue,urlcolor=blue,linkcolor=blue]{hyperref}
\usepackage{color,comment}
\usepackage{algorithm,algorithmic}
\usepackage[font=small,format=plain,labelfont=bf,up]{caption}
\usepackage{lpic}
\usepackage{subfigure}
\usepackage{multirow}
\usepackage{lscape}
\usepackage{adjustbox}
\usepackage{graphics}
\usepackage{setspace}
\usepackage{xypic}

\usepackage{algorithm}

\newcommand{\union}{\bigcup} 
\newcommand{\wo}{\setminus}  
\newcommand{\set}[1]{\left\{{#1}\right\}}
\newcommand{\setof}[2]{\left\{{#1}\,:\,{#2}\right\}}
\newcommand{\abs}[1]{\left|{#1}\right|}

\newcommand{\mycomment}[1]{}

\def\cB{\mathcal{B}}
\def\cF{\mathcal{F}}
\def\Z{{\mathbb Z}}

\DeclareMathOperator{\Ind}{Ind}

\newtheorem{lemma}{Lemma}
\newtheorem{theorem}[lemma]{Theorem}

\newtheorem{corollary}[lemma]{Corollary}

\newtheorem{conjecture}[lemma]{Conjecture}
\newtheorem{claim}{Claim}[lemma]

\theoremstyle{definition}

\newtheorem{observation}[lemma]{Observation}

\newtheorem{definitiontheorem}{Definition}[lemma]
\newtheorem{observationtheorem}{Observation}[lemma]

\input colordvi
\usepackage[margin=1in]{geometry}

 \newenvironment{cem}
{
    \begin{enumerate}
        \setlength{\topsep}{0pt}
        \setlength{\parskip}{0pt}
        \setlength{\partopsep}{0pt}
        \setlength{\parsep}{0pt}
        \setlength{\itemsep}{0pt}
}
{
    \end{enumerate}
}
\usepackage{fancyhdr}

\newcounter{casenum}
\newcounter{subcasenum}
\numberwithin{subcasenum}{casenum}
\newcounter{subsubcasenum}
\numberwithin{subsubcasenum}{subcasenum}

\renewcommand{\thecasenum}{\arabic{casenum}}
\renewcommand{\thesubcasenum}{\thecasenum.\roman{subcasenum}}

\newcounter{stagenum}

\newenvironment{mycases}
{
  \list{}{%
    \leftmargin0.5cm   
    \rightmargin0cm
  }
  \item\relax
	\setcounter{casenum}{0}
}
{
	\endlist
}
\newenvironment{subcases}
{
  \list{}{%
    \leftmargin0.5cm   
    \rightmargin0cm
  }
  \item\relax
}
{
	\endlist
}

\newenvironment{casefig}
{
\vspace{0em}
	\begin{figure}[H]
		\centering
}
{
	\end{figure}
\vspace{-1em}
}

\def\casefigratio{1}

\newcommand{\case}[1]{
	\vspace{0.5em}

	\refstepcounter{casenum}
	\noindent\hspace{-0.5cm}\textit{Case \thecasenum: #1}
}

\newcommand{\subcase}[1]{
	\vspace{0.25em}

	\refstepcounter{subcasenum}
	\noindent\hspace{-0.5cm}\textit{Case \thesubcasenum: #1}
}

\def\dalpha{\overline{\alpha}}

\newenvironment{dischargingrule}
{

  \list{}{%
    \leftmargin0.5in   
    \rightmargin0cm 
    \setlength{\parindent}{-0.25in}
    \setlength{\itemindent}{-0.25in}
  }
  \item\relax
}
{
\endlist

}

\newenvironment{casecomment}
{
  \list{}{%
    \leftmargin-0.5cm   
    \rightmargin0cm
  }
  \item\relax
}
{
	\endlist
}

\begin{document}

\title{On the independence ratio of distance graphs}
\author{
James M. Carraher$^1$ \and David Galvin$^2$ \and Stephen G. Hartke$^1$ \and A. J. Radcliffe$^1$ \and Derrick Stolee$^3$}
\date{\today}

\pagestyle{plain}

\maketitle

\begin{abstract}
A distance graph is an undirected graph on the integers where two integers are adjacent if their difference is in a prescribed distance set.
The \emph{independence ratio} of a distance graph $G$ is the maximum density of an independent set in $G$.
Lih, Liu, and Zhu \cite{LLZ99} showed that the independence ratio is equal to the inverse of the fractional chromatic number, thus relating the concept to the well studied question of finding the chromatic number of distance graphs.

We prove that the independence ratio of a distance graph is achieved by a periodic set, and we present a framework for discharging arguments to demonstrate upper bounds on the independence ratio.
With these tools, we determine the exact independence ratio for several infinite families of distance sets of size three, determine asymptotic values for others, and present several conjectures.
\end{abstract}
\footnotetext[1]{Department of Mathematics, University of Nebraska--Lincoln, \texttt{\{s-jcarrah1,hartke,aradcliffe1\}@math.unl.edu}.  Hartke was supported in part by National Science Foundation grant DMS-0914815.}
\footnotetext[2]{Department of Mathematics, Notre Dame University, \texttt{dgalvin1@nd.edu}.}
\footnotetext[3]{Department of Mathematics, Department of Computer Science, Iowa State University, \texttt{dstolee@iastate.edu}.}


\section{Introduction}

For a set $S$ of positive integers, the \emph{distance graph} $G(S)$ is the infinite graph with vertex set $\Z$ where two integers $i$ and $j$ are adjacent if and only if $|i-j| \in S$.
For an independent set $A$ in $G(S)$ the density $\delta(A)$ is equal to $\limsup_{N\to\infty} \frac{|A\cap [-N,N]|}{2N+1}$.
The \emph{independence ratio} $\dalpha(S)$ is the supremum of $\delta(A)$ over all independent sets $A$ in $G(S)$.  In this paper we develop several fundamental techniques for determining the independence ratio, which we then apply to several infinite families of distance sets of size three.  These distance sets are the smallest whose independence ratios were previously unknown.

A basic tool for our study of the independence ratio is the following theorem.

\begin{theorem}\label{thm:periodic}
Let $S$ be a finite set of positive integers and let $s= \max S$.
There exists a periodic independent set $A$ in $G(S)$ with period at most $s2^s$ where $\delta(A) = \dalpha(S)$.
\end{theorem}

In particular we can give upper bounds on $\dalpha(S)$ by studying only periodic independent sets.
Our proof of Theorem~\ref{thm:periodic} uses a lemma (Lemma~\ref{lma:cycles}) about extremal walks in finite digraphs that we also apply to show there exist periodic extremal sets for dominating sets and identifying codes in finitely-generated distance graphs.
This lemma may be also applicable in other situations.

To prove upper bounds on $\dalpha(S)$, we develop a new discharging method.
The resulting \emph{Local Discharging Lemma} is then used extensively to give exact values of $\dalpha(S)$ for several infinite families of distance sets.
We witness several common themes among these proofs, and these themes may be evidence that discharging arguments of this type could be used to determine almost all values of $\dalpha(S)$.

Intense study of distance graphs began when Eggleton, Erd\H{o}s, and Skilton~\cite{EES85,EES90} defined them as a modified version of the Hadwiger-Nelson problem of coloring the graph of the plane with an edge between every pair of points at unit distance.
The chromatic number of distance graphs has since been widely studied~\cite{BS08,CHZ98,CLZ99,CCH97,Collins,DZ97,GZ96,EES85,EES90,EHL12,EHT13,Heuberger,Ilic,IB10,Katznelson, KK98, LL05, LLZ99, Liu08, LS13, LZ99, LZ04, RTV02, Togni13, Zhu02}.

A particularly effective tool for finding lower bounds on the chromatic number is to determine the \emph{fractional chromatic number}, $\chi_f(S) = \chi_f( G(S) )$.
A \emph{fractional coloring} of a graph $G$ is a function $c$ from the independent sets $I$ of $G$ to nonnegative real numbers such that for every vertex 
$v$, the sum $\sum_{I \owns v} c(I) \geq 1$, and the \emph{value} of the coloring is the sum $\sum_{I} c(I)$ taken over all independent sets $I$.
The fractional chromatic number $\chi_f(G)$ is the minimum value of a fractional coloring, and provides a lower bound on the chromatic number.
Lih, Liu and Zhu showed that determining the fractional chromatic number $\chi_f(S)$ of a distance graph $G(S)$ is equivalent to determining its independence ratio.


\begin{theorem}[{Lih, Liu, and Zhu \cite{LLZ99}}]\label{thm:fractionalchromatic}
Let $S$ be a finite set of positive integers.
Then $\chi_f(S) = \dalpha(S)^{-1}$.
\end{theorem}

Thus, the previous results computing the fractional chromatic number of distance graphs~\cite{CHZ98,CLZ99,CCH97,LL05,Zhu02} apply to the independence ratio.
We summarize these results in Section~\ref{sec:previous}.

For an integer $n$, the \emph{circulant graph} $G(n,S)$ is the graph whose vertices are the integers modulo $n$ where two integers $i$ and $j$ are adjacent if and only if $|i-j| \equiv k \pmod n$, for some $k \in S$.
The distance graph $G(S)$ can be considered to be the limit structure of the sequence of circulant graphs $G(n,S)$.
Thus, extremal questions over the circulant graphs lead to extremal questions on the distance graph.
For instance, the equality $\dalpha(S) = \limsup_{n\to \infty} \alpha(G(n,S)) / n$ is a consequence of Theorem \ref{thm:fractionalchromatic}.

Since the complement of a circulant graph is also a circulant graph, and the independence number of a graph is the clique number of its complement, studying the independence ratio of distance graphs is strongly related to determining the independence number and clique number of circulant graphs.
In particular, simultaneously bounding  the independence number and clique number of circulant graphs has shown lower bounds on Ramsey numbers~\cite{LS08}.
So far, these parameters have been studied for circulant graphs $G(n,S)$ when limited to special classes of sets $S$, whether algebraically defined~\cite{BEHWPaleySquares, BlokhuisPaleySquares, BDRPaley, CohenPaley, GreenCayley, KSUnitaryCayley, Paley} or with $S$ finite and $n$ varying~\cite{BICliqueCirculant, BHFractionalRamsey, HoshinoCirculant}.

A recent development is the discovery that certain circulant graphs $G(n,S)$ are \emph{uniquely $K_r$-saturated}, including three infinite families~\cite{UniqueSaturation}.
The first step in proving this property is showing that the clique number of $G(n,S)$ is equal to $r-1$.
In the three infinite families, the generating set $S$ uses a growing number of elements, but the complement of the graph uses a finite number of elements.
The complement $\overline{G(n,S)}$ is equal to another circulant graph $G(n,S')$ where $S'$ has a finite number of elements.
In this complement, the independence number is of particular interest.
Our discharging method is an adaptation of the discharging method used in~\cite{UniqueSaturation} to determine the independence number in circulant graphs.

We start in Section~\ref{sec:periodic} by proving Theorem~\ref{thm:periodic}, that the independence ratio in a distance graph is achievable by a periodic independent set. 
We take the opportunity there to show two quick applications of the proof technique to related density problems for other types of subset of $G(S)$. 
In Section~\ref{sec:previous} we summarize previous results on the independence ratio. 
The next section collects some introductory results concerning $\dalpha(S)$, and then in Section~\ref{sec:discharging} we define our discharging process and its connection to the independence ratio, proving the Local Discharging Lemma. 
We use the Local Discharging Lemma to prove exact values of $\dalpha(S)$ for several families of sets $S$ in Section~\ref{sec:threegens}. 
We determine the independence ratio for a range of graphs with generator sets of size $3$. In particular we determine $\dalpha(S)$ for $S=\set{1,3,2i}$ for $i\ge 2$, $\set{1,5,2i}$ for $i\ge 5$, and conjecture the value of $\dalpha{\set{1,\ell,2i}}$ for all odd $\ell$. We also analyze the case of $S=\set{1,2k,2k+2\ell}$ where $k,\ell\ge 1$. 
This extends work of Zhu \cite{Zhu02}, who considered the fractional chromatic number of $G(S)$ in distance graphs with generator sets of size $3$. 
Finally in Section~\ref{sec:computation} we discuss the algorithm we used to compute values of $\dalpha(S)$ for specific finite sets $S$.
The computed values of $\dalpha(\{1,1+k, 1+k+i\})$ are given as a table in Appendix A, while more values are given in data available online\footnote{See \url{http://www.math.iastate.edu/dstolee/r/distance.htm} for all data files.}.
The statements of several theorems and conjectures were discovered by this data, and we present the computed values along with the stated functions in figures alongside these statements\footnote{For example, see Theorem~\ref{thm:1-4-k} and Figure~\ref{fig:1-4-k}.}.

Our notation is standard.
Throughout the paper we consider $S$ to be a finite, nonempty set of positive integers.
For a positive integer $n$, we write $[n] = \{1,\dots,n\}$, and similarly $[a,b] = \{a, a+1, \dots, b-1, b\}$.
When $d \geq 1$, we let $d \cdot S = \{ d\cdot s : s \in S\}$ and $S + d = \{ s + d : s \in S\}$.

\section{Periodic Sets of Extremal Value}\label{sec:periodic}

In this section, we prove Theorem~\ref{thm:periodic}.
As a consequence, we give an alternative proof of Theorem~\ref{thm:fractionalchromatic}.
We start by proving a lemma that has independent interest. Then, after the proof of Theorem~\ref{thm:periodic}, we use the lemma to show
that the density of other subsets of $G(S)$, such as dominating sets and $r$-identifying codes, are maximized by appropriate periodic examples.

Consider a finite directed graph $G$ where every vertex $v$ is given a weight $w(v)$.
Let $W = (v_i)_{i\in \Z}$ be a doubly infinite walk on $G$.
Then the \emph{upper average weight} $\overline{w}(W)$ of $W$ is defined as $\limsup_{N\to \infty} \frac{\sum_{i=-N}^N w(v_i)}{2N+1}$, and the \emph{lower average weight} $\underline{w}(W)$ of $W$ is defined as $\liminf_{N\to \infty} \frac{\sum_{i=-N}^N w(v_i)}{2N+1}$.
Given a simple cycle $C$ in $G$, define the infinite walk $W_C$ by infinitely repeating $C$.
Observe that $\overline{w}(W_C) = \underline{w}(W_C) = \frac{\sum_{v\in V(C)} w(v)}{|C|}$.

\begin{lemma}\label{lma:cycles}
Let $G$ be a finite, vertex-weighted digraph. 
The supremum of upper average weights (or infimum of lower average weights) of infinite walks on $G$ is equal to the maximum upper average weight (or minimum lower average weight, respectively) of some infinite walk $W_C$ where $C$ is a simple cycle.
\end{lemma}

\begin{proof}
We prove that every infinite walk $W$ has $\overline{w}(W)$ bounded above by the maximum upper average weight of a simple cycle.
In the case of minimizing $\underline{w}(W)$, we can simply negate the weights of the vertices and apply the maximization case.

Fix $W = (v_i)_{i\in \Z}$.
Let $n = |V(G)|$ and let $N$ be large.
We will approximate the fraction $\frac{\sum_{i=-N}^N w(v_i)}{2N+1}$ by a convex combination of upper average weights of simple cycles of $G$.
This approximation improves as $N$ grows, so we find the limit definition of $\overline{w}(W)$ is bounded above by the maximum upper average weight of a simple cycle.

Let $X=(x_i)_{i\in [a,b]}$ be any finite walk.  If $x_i,x_{i+1},\dots, x_{i+\ell}$ are distinct vertices where $x_{i+\ell+1} = x_i$, then we say $C =  (x_i, x_{i+1},\dots, x_{i+\ell})$ is a simple cycle within $X$ (note that a simple cycle may have one or two vertices).
Define $X'$ to be the \emph{contraction of $C$ in $X$} by removing the subwalk $(x_i, x_{i+1},\dots, x_{i+\ell})$ from $X$. 

Let $W_N = ( v_i )_{i\in [-N,N]}$.
Starting with $W^{(1)} = W_N$, we iteratively construct a list of walks $W^{(1)}, \dots, W^{(t+1)}$ and a list of simple cycles $C_1,\dots, C_t$ such that $C_i$ is a simple cycle in $W^{(i)}$ and $W^{(i+1)}$ is the contraction of $C_i$ in $W^{(i)}$.
This iterative construction stops when $W^{(t+1)}$ does not contain a simple cycle.
Such a walk does not repeat any vertices, so $|W^{(t+1)}| \leq n \ll N$.
Observe that
\begin{align*}
	\sum_{i=-N}^{N} w(v_i) &= \sum_{j=1}^t \left[\sum_{v_i \in C_j} w(v_i)\right] + \sum_{v_i \in W^{(t+1)}} w(v_i)\\
		&= \sum_{j=1}^t |C_j| \overline{w}(C_j) + \sum_{v_i \in W^{(t+1)}} w(v_i).
\end{align*}
Also note that $\sum_{j=1}^t |C_j| = |W_N|-|W^{(t+1)}|$.
Thus we have
\begin{align*}
	\overline{w}(W_N)=\frac{\sum_{i=-N}^N w(v_i)}{|W_N|} &= \frac{\sum_{j=1}^t |C_j| \overline{w}(C_j) + \sum_{v_i \in W^{(t+1)}} w(v_i)}{(|W_N|-|W^{(t+1)}|) + |W^{(t+1)}|}\\
	&\leq \frac{\sum_{j=1}^t |C_j| \overline{w}(C_j)}{|W_N|-|W^{(t+1)}|}+\frac{\sum_{v_i\in W^{(t+1)}} w(v_i)}{|W_N|-|W^{(t+1)}|} \\
	& \leq \sum_{j=1}^t \frac{|C_j|}{|W_N|-|W^{(t+1)}|} \overline{w}(C_j)+\frac{n \max_{v_i\in V(G)} w(v_i)}{2N+1-n}.
\end{align*}
Since there are a finite number of cycles in $G$, the terms of the sum in the last expression can be grouped by cycles with the same vertices in $G$.
This sum is thus a convex combination of values $\overline{w}(C)$ over cycles $C$ in $G$, and hence is at most $\max\{ \overline{w}(C) : C \text{ a cycle in $G$}\}$.

By taking the limsup as $N\to \infty$ of both sides of the inequality, we obtain

\vspace{1em}
\hfill\qquad\qquad$\overline{w}(W)=\limsup \overline{w}(W_N)\leq \max\{\overline{w}(C): C \text{ a cycle in $G$}\}$.\hfill \qedhere
\end{proof}

We now proceed to use Lemma~\ref{lma:cycles} to prove that there is a periodic independent set that attains the independence ratio.  Our technique is a general approach that shows, for any type of object that is a subset of vertices in distance graphs, the extremal density is attained by periodic objects.

Let $X\subseteq \Z$ be an object satisfying a property $\mathcal{P}$ in a distance graph $G(S)$.  We consider the intersection of $X$ with disjoint consecutive intervals of fixed length $\ell$:
\[\dots, X\cap [-\ell+1,0], \ X\cap[1,\ell], \ X\cap[\ell+1,2\ell], \dots\]
There are only a finite number of possibilities for $X\cap [k \ell+1, (k+1)\ell]$, considered up to translation by multiples of $\ell$.  We call such patterns ``states'', and encode in a ``state graph'' which states can follow a given state.  Since there are a finite number of states, we can apply Lemma~\ref{lma:cycles} to the state graph to obtain an extremal periodic object.

More formally, a \emph{state} is a set $T \subseteq [\ell]$.  An object $X$ satisfying property $\mathcal{P}$ has state $T$ on the interval $[k\ell+1,(k+1)\ell]$, where $k\in\Z$, if $X\cap[k\ell+1,(k+1)\ell]=T+k\ell$.
An \emph{admissible state} is a state $T$ such that there exists an object $X$ with property $\mathcal{P}$ such that $X$ has state $T$ on some interval.
A \emph{transition} occurs between two admissible states $T$ and $T'$ if there exists an object $X$ with property $\mathcal{P}$ and an integer $k$ such that $X$ has state $T$ on interval $[k\ell+1,(k+1)\ell]$ and state $T'$ on interval $[(k+1)\ell+1,(k+2)\ell]$.
Let the \emph{state graph} be the directed graph where the vertices are the admissible states, and the directed edges are the transitions.
The weight of a state is determined by the type of objects, but is usually the density $|T|/\ell$.

Give an object $X$ with property $\mathcal{P}$, the states $T_i=X\cap[i\ell+1,(i+1)\ell]$ give rise to an infinite walk in the state graph by definition.
Every doubly infinite walk $W=(T_i)_{i\in\Z}$ in the state graph gives rise to a subset $Y$ of $\Z$ by $Y=\union_{i\in \Z}(T_i+i\ell)$ for $i\in\Z$.
Suppose that every subset $Y$ created from an infinite walk in this way also has property $\mathcal{P}$.  For the properties that we are interested in, this fact can usually be guaranteed by choosing $\ell$ large enough.
By applying Lemma~\ref{lma:cycles} to the state graph, there is a cycle $C$ whose weight is extremal.  Let $W$ be the walk formed by infinitely repeating $C$, and let $Y$ be the corresponding object with property $\mathcal{P}$.  Note that $Y$ is periodic with period $\ell$ times the length of $C$.  Since there are at most $2^\ell$ states, the period of $Y$ is at most $\ell 2^{\ell}$.


\begin{proof}[Proof of Theorem~\ref{thm:periodic}.]
We consider independent sets in $G(S)$.  Let $s=\max S$, and choose interval lengths of $\ell=s$.  The states are subsets of $[s]$, the admissible states are subsets of $[s]$ that are independent in $G(S)$, and the weight of a state $T$ is $|T|/s$.  State $T$ can transition to state $T'$ if $T\cup(T'+s)$ is an independent set.
Given any independent set $X$ in $G(S)$, the states $T_i=X\cap[is+1,(i+1)s]$ give rise to an infinite walk $W$ in the state graph since the intersection of $X$ with two consecutive intervals remains independent.
Note that the density $\delta(X)$ is equal to $\overline{w}(W)$.

Next we show that every subset $Y=\union_{i\in \Z}(T_i+is)$ created from an infinite walk on the state graph is an independent set in $G(S)$.
If not, there are $j_1, j_2\in Y$ such that $|j_1-j_2|\in S$.  Since $\ell=s$ both $j_1$ and $j_2$ either belong to the same $T_i+is$ or two consecutive states, $T_i+is$ and $T_{i+1}+(i+1)s$.
Both $j_1$ and $j_2$ can not belong to the same $T_i+is$, since $T_i$ is admissible and hence independent in $G(S)$.
If $j_1\in T_i+is$ and $j_2\in T_{i+1}+(i+1)s$, then $(T_i+is)\union (T_{i+1}+(i+1)s)$ is not independent, which contradicts that there is a transition between those states.

By Lemma~\ref{lma:cycles}, the weight of an infinite walk in the state graph is maximized by a simple cycle $C$.
The independent set $Y$ created from the infinite walk repeating $C$ has density $\dalpha(S)$ and is periodic with a period of length $s |C|$, which is at most $s2^s$.
\end{proof}

To demonstrate the versatility of Lemma~\ref{lma:cycles}, we also prove that other problems on distance graphs admit periodic extremal sets.

A set of vertices $D$ is \emph{dominating} if every vertex in the graph is either in $D$ or adjacent to a vertex in $D$.

\begin{theorem}
Let $S$ be a finite set of positive integers and set $s = \max S$.
The minimum density of a dominating set in $G(S)$ is achieved by a periodic set with period at most $(2s)2^{2s}$.
\end{theorem}
\begin{proof}
We consider dominating sets in $G(S)$.  Let $s=\max S$, and choose interval lengths of $\ell=2s$.  The states are subsets of $[2s]$, and all states are admissible.
The weight of a state $T$ is $|T|/(2s)$.  State $T$ can transition to state $T'$ if every vertex in $[s+1,3s]$ either is in $T\cup(T'+2s)$ or has a neighbor in $T\cup (T'+2s)$.
Given any dominating set $X$ in $G(S)$, the states $T_i=X\cap[i(2s)+1,(i+1)2s]$ give rise to an infinite walk $W$ in the state graph since every integer in the interval $[i(2s)+s+1,i(2s)+3s]$ is either in, or has a neighbor in, the set $\left(T_i+i(2s)\right)\cup \left(T_{i+1}+(i+1)(2s)\right)$.
Note that the density $\delta(X)$ is equal to $\overline{w}(W)$.

Next we show that every subset $Y=\union_{i\in \Z}(T_i+i(2s))$ created from an infinite walk on the state graph is a dominating set in $G(S)$.
If not, then there exists $j\notin Y$ that has no neighbor in $Y$.
The integer $j$ is in some interval $[i(2s)+1,(i+1)2s]$.
If $j\in [i(2s)+1,(i)2s+s]$, then since there is a transition from $T_{i-1}$ to $T_i$, $j$ must have a neighbor in $Y\cap [(i-1)(2s)+1,(i+1)(2s)]$.
If $j\in [i(2s)+s+1,(i+1)2s]$, then since there is a transition from $T_{i}$ to $T_{i+1}$, $j$ must have a neighbor in $Y\cap [(i)(2s)+1,(i+2)(2s)]$.

By Lemma~\ref{lma:cycles}, the weight of an infinite walk in the state graph is minimized by a simple cycle $C$.
The dominating set $Y$ created from the infinite walk repeating $C$ has minimum periodic with a period of length $2s |C|$, which is at most $(2s)2^{2s}$.
\end{proof}


Define the \emph{ball $B_r(u)$ of radius $r$ centered at $u$} to be the set of all vertices in $G(S)$ that are distance at most $r$ from $u$.
A set $A$ of vertices is an \emph{$r$-identifying code} if for every pair of distinct vertices $u$ and $v$ in $G(S)$, the sets $A \cap B_r(u)$ and $A\cap B_r(v)$ are nonempty and distinct.

\begin{theorem}
Let $S$ be a finite set of positive integers and set $s = \max S$.
The minimum density of a $1$-identifying code in $G(S)$ is achieved by a periodic set with period at most $(6s)2^{6s}$.
\end{theorem}

\begin{proof}
Let $s=\max S$, and choose intervals of length $\ell=6s$.  The states are subsets of $[6s]$, and all states are admissible. The weight of a state $T$ is $T/(6s)$.  State $T$ can transition to state $T'$ if in $T\cup (T'+6s)$ every distinct vertices $u\in [3s+1,9s]$ and $v\in [s+1,11s]$ have the property $N[u]\cap (T\cap T'+6s)\neq N[v]\cap (T\cap T'+6s)$.

Next we show that every subset $Y=\union_{i\in \Z}(T_i+i(6s))$ created from an infinite walk on the state graph is a 1-identifying code in $G(S)$.
If not, then there exist distinct $u,v\in \Z$ where $N[u]\cap Y=N[v]\cap Y$.
The integer $u$ is in some interval $[i6s+3s+1,(i+1)6s+3s]$ for some $i$.
Since $N[u]\cap N[v]\neq \emptyset$, then $u$ and $v$ must be within distance $2s$, and hence $v\in [i6s+s+1,(i+1)6s+5s]$.
Thus, $u$ and $v$ are both in $(T_i+i6s)\cup (T_{i+1}+(i+1)6s)$, where there is a transition from $T_i$ to $T_{i+1}$.

By Lemma~\ref{lma:cycles}, the weight of an infinite walk in the state graph is minimized by a simple cycle $C$.
The 1-identifying code $Y$ created from the infinite walk repeating $C$ has minimum density and is periodic with a period of length $6s |C|$, which is at most $(6s)2^{6s}$.
\end{proof}


Observe that an $r$-identifying code in a graph $G$ corresponds to a 1-identifying code in $G^r$, where $G^r$ is the graph with the same vertex set as $G$ but two vertices $u,v$ are adjacent in $G^r$ if they have distance at most $r$ in $G$.
Further, $G(S)^r$ is a distance graph with distance set $S' =\union_{t=1}^r \left\{\sum_{i=1}^t a_i : |a_i| \in S\right\}$, and $\max S' = r\max S$.
Thus, we have the following corollary.

\begin{corollary}
Let $S$ be a finite set of positive integers and set $s = \max S$.
The minimum density of an $r$-identifying code in $G(S)$ is achieved by a periodic set with period at most $(6sr)2^{6sr}$.
\end{corollary}

For completeness, we demonstrate a similar proof for chromatic number, using an unweighted statement analogous to Lemma~\ref{lma:cycles}.
This result was previously shown by Eggleton, Erd\H{os}, and Skilton~\cite{EES90}.

\begin{theorem}
Let $S$ be a finite set of positive integers and set $s = \max S$.
For $k = \chi(G(S))$, there exists a periodic proper $k$-coloring $c$ with minimum period at most $sk^s$.
\end{theorem}

\begin{proof}
A state is a coloring $c : [s] \to \{1,\dots,k\}$.
A state $c$ is admissible if the partial coloring induced on $G(S)$ is proper.
State $c$ can transition to state $c'$ if the partial coloring
$$c''(i) = \begin{cases} c(i) & \text{if $i \in [s]$}\\ c'(i-s) & \text{if $i \in [s+1,2s]$}\end{cases}$$
is proper in $G(S)$.
The state graph has admissible states as vertices and transitions as edges.
As in the case of independent sets, there is a correspondence between proper $k$-colorings of $G(S)$ and infinite walks in the state graph.
Since there is a proper $k$-coloring of $G(S)$, there is some infinite walk in the state graph, and hence the state graph contains at least one cycle.
Infinitely repeating this cycle corresponds to a periodic proper $k$-coloring with period at most $sk^s$.
\end{proof}

To finish this section, we present an alternate proof of one inequality of Theorem~\ref{thm:fractionalchromatic} using an extremal periodic independent set.
The original proof by Lih, Liu, and Zhu~\cite{LLZ99} first showed equality between $\chi_f( G(n,S) )$ and $\alpha( G(n,S) )/n$ and then took the limit.

\vspace{0.5em}
\noindent\textbf{Theorem~\ref{thm:fractionalchromatic}} ({Lih, Liu, and Zhu \cite{LLZ99}})\textbf{.}
Let $S$ be a finite set of positive integers.
Then \[\chi_f(S) = \dalpha(S)^{-1}.\]

\begin{proof}
First we show $\chi_f(S)\leq \dalpha(S)^{-1}$.
Let $\Ind(S)$ be the family of independent sets in $G(S)$.
By Theorem~\ref{thm:periodic}, there exists a periodic independent set $A$ with period $p$ and $\delta(A) = \dalpha(S)$.
Form the fractional chromatic coloring $c$ by assigning $c(A+i) = \frac{1}{p\delta(A)}$ for all $i \in \{0,\dots,p-1\}$ and $c(I) = 0$ for all other independent sets $I$.
Since $A$ is periodic and contains $p\delta(A)$ elements within any interval of length $[p]$, a vertex $x$ appears in exactly $p\delta(A)$ sets $A+i$.
Thus, $\sum_{\Ind(S): v\in I} c(I) = 1$, so $c$ is a fractional coloring.
The value of this fractional coloring is $p\cdot \frac{1}{p\delta(A)} = \dalpha(S)^{-1}$.

Next we show $\chi_f(S)\geq \dalpha(S)^{-1}$, which is the analogue for infinite graphs of the well known fact that $\chi_f(G)\geq \frac{1}{\alpha(G)}$ for a finite graph $G$.  Let $G(S)[-n,n]$ be the finite subgraph of $G(S)$ induced by the vertices $[-n,n]$.  We know
\[\dfrac{2n+1}{\alpha(G(S)[-n,n])}\leq  \chi_f(G(S)[-n,n])\leq \chi_f(S)\]
for all $n\in \Z$.
We show $\dalpha(S)\geq \limsup_{n\to\infty} \frac{\alpha(G(S)[-n,n])}{2n+1}$, which with the above inequality implies $\chi_f(S)\geq \dalpha(S)^{-1}$.

Let $L:=\limsup_{n\to\infty}\frac{\alpha(G(S)[-n,n])}{2n+1}$, and let $A_n$ be a maximum independent set in $G(S)[-n,n]$.
For $n\geq s$, $X_n:=\union_{z\in \Z}(A_n\cap [-n,n-s]+2nz)$ is an infinite independent set in $G(S)$ of density at least $\frac{\alpha(G(S)[-n,n])-s}{2n}$, which approaches $L$ as $n\to\infty$.
Thus there exists a sequence $(X_n)_{n\geq s}$ of independent sets in $G(S)$ with densities that approach $L$, and so $\dalpha(S)\geq L$.
\end{proof}

\section{Previous Results}
\label{sec:previous}

We outline here the known results concerning $\dalpha(S)$. The results were all phrased in terms of $\chi_f(S)$, but by 
Theorem~\ref{thm:fractionalchromatic} we know the following versions are equivalent.

\begin{theorem}[Gao and Zhu~\cite{GZ96}]
Let $k$ and $k'$ be positive integers such that $k \leq k'$.
\begin{enumerate}
\setlength{\itemsep}{0em}
\item $\dalpha([1,k']) = \dalpha([k']) = \frac{1}{k'+1}$.
\item If $k' \geq (5/4)k$, then $\dalpha([k,k']) = \frac{k}{k+k'}$.
\end{enumerate}
\vspace{-1.5em}  \hfill $\square$
\end{theorem}

\begin{theorem}[Chang, Liu, and Zhu~\cite{CLZ99}; Liu and Zhu~\cite{LZ99}]
For positive integers $m, k, s$ such that $sk \leq m$, let $D_{m,k,s} = [m] \setminus k[s]$.
\begin{enumerate}
\setlength{\itemsep}{0em}
\item If $2k > m$, then $\dalpha(D_{m,k,1}) = \frac{1}{k}$.
\item If $2k \leq m$, then $\dalpha(D_{m,k,1}) = \frac{2}{m+k+1}$.
\item If $m \geq (s+1)k$, then $\dalpha(D_{m,k,s}) = \frac{s+1}{m+sk+1}$.
\end{enumerate}
\vspace{-1.5em}  \hfill $\square$
\end{theorem}

\begin{theorem}[Lam and Lin~\cite{LL05}]
For positive integers $m \geq k' \geq k\geq 1$, let $D_{m,[k,k']} = [m] \setminus [k,k']$.
\begin{enumerate}
\setlength{\itemsep}{0em}
\item If $m < 2k$, then $\dalpha(D_{m,[k,k+i]}) = \frac{1}{k}$.
\item If $2k\leq m < 2k+2i$, and $1 \leq i \leq k-1$, then $\dalpha(D_{m,[k,k+i]}) = \frac{2}{m+1}$.
\item If $m \geq 2k+2i$ and $1 \leq i \leq k-1$, then $\dalpha(D_{m,[k,k+i]}) = \frac{2}{m+k+1}$.
\item If $m < (s+1)k$, then $\dalpha(D_{m,[k,sk+i]}) = \frac{1}{k}$.
\item If $1 \leq i \leq k-1$ and $(s+1)k \leq m < (s+1)k+i$, then $\dalpha(D_{m,[k,sk+i]}) = \frac{s+1}{m+1}$.  In particular, if $k > 1$ then $$\dalpha([k,m]) = \begin{cases} \frac{1}{k} & \text{if $m \not\equiv 0\pmod k$,}\\ \frac{s+1}{k+1} & \text{if $m = (s+1)k$.}\end{cases}$$
\end{enumerate}
\vspace{-1.5em}  \hfill $\square$
\end{theorem}

\begin{theorem}[Liu, Zhu~\cite{LZ04}]\label{thm:lz04}
Let $0< a< b$, $m \geq 2$, and $\gcd(a,b) = 1$.
\begin{enumerate}
\setlength{\itemsep}{0em}
\item $\dalpha(\{a, 2a, \dots, (m-1)a, b) = \begin{cases} \frac{k}{km+1} & \text{if $b = km$ for some $k$}\\ \frac{1}{m} & \text{otherwise}\end{cases}$.
\item $\dalpha(\{a, b, a+b\}) = \begin{cases}\frac{1}{3} & \text{if $b-a=3k$}\\ \frac{a+k}{3a+3k+1} & \text{if $b -a = 3k+1$}\\ \frac{a+2k+1}{3a+6k+4} & \text{if $b-a = 3k+2$}\end{cases}$.
\item if $a \not\equiv b \pmod 2$, then $\dalpha(\{a, b, b-a, a+b\}) = \frac{1}{4}$.
\item $\dalpha(\{1,2m,2m+1,2m+2\}) = \frac{m}{4m+1}$.
\end{enumerate}
\end{theorem}

Chang, Huang, and Zhu~\cite{CHZ98} and Collins~\cite{Collins} determined the exact values of $\chi_f(S)$ for all sets $S$ of size two.

\begin{theorem}[Chang, Huang, and Zhu~\cite{CHZ98}; Collins~\cite{Collins}]\label{thm:CHZ98C}
Let $S = \{a,b\}$ with $1 \leq a < b$ and $\gcd(a,b) = 1$.
\begin{enumerate}
\setlength{\itemsep}{0em}
\item If $a$ and $b$ are both odd, then  $\dalpha(S) = \frac{1}{2}$.
\item If at least one of $a$ and $b$ is even, $\dalpha(S) = \frac{a+b-1}{2a+2b}$.  In particular, $\dalpha(\{1,2k\})= \frac{k}{2k+1}$.
\end{enumerate}
\vspace{-1.5em}  \hfill $\square$
\end{theorem}

Zhu~\cite{Zhu02} investigated the case where $|S| = 3$ and determined bounds on $\chi_f(S)$ and the circular chromatic number for $G(S)$.  
These bounds are sufficient to determine $\chi(S)$ exactly.

\begin{theorem}[Zhu \cite{Zhu02}]\label{thm:3gensfractional}
Let $S = \{a,b,c\}$ where $1 \leq a < b < c$.
\begin{enumerate}
\setlength{\itemsep}{0em}
\item If $a$, $b$, and $c$ are odd, then $\dalpha(S) = \frac{1}{2}$.
\item If $S = \{1,2,3k\}$ where $k \geq 1$, then $\dalpha(S) = \frac{k}{3k+1}$.
\item If $b = a+3k$ and $c = 2a+3k$ for $k \geq 1$, then $\dalpha(S) = \frac{1}{3}$.
\item If $b = a + 3k + 1$ and $c = 2a+3k+1$ for $k \geq 1$, then $\frac{a+k}{3(a+k)+1}\leq \dalpha(S) \leq \frac{a+2k}{3(a+2k)+1}$.
\item If $b = a+3k+2$ and $c = 2a+3k+2$ for $k \geq 1$, then $\frac{a+2k+1}{3(a+2k+2)+1} \leq \dalpha(S) \leq \frac{a+2k+2}{3(a+2k+2)+1}$.
\item If $a, b, c,$ are not all odd, $c \neq a+b$, and $(a,b,c) \neq (1,2,3k)$ for any $k \geq 1$, then $\frac{1}{3} \leq \dalpha(S) < \frac{1}{2}$.
\item If $a,b,c,$ are not all odd, $c \neq 2b$, $b \neq 2a$, $c \neq 2a$ and $c \neq a+b$, then $\frac{3}{8} \leq \dalpha(S) < \frac{1}{2}$ with a finite number of exceptional triples $(a,b,c)$.
\end{enumerate}
\vspace{-1.5em}  \hfill $\square$
\end{theorem}

While Theorem~\ref{thm:3gensfractional} determines the exact value of $\dalpha(S)$ for several classes of sets of size three, it leaves many triples undetermined.


\section{Relations Between Generating Sets}\label{sec:preliminary}


The following observations are very easy to prove.

\begin{observation}
If $S$ contains only odd numbers, then $\dalpha(S) = \frac{1}{2}$.
\end{observation}

\begin{observation}\label{obs:subset}
	If $S \subseteq T$, then $\dalpha(S) \geq \dalpha(T)$.
\end{observation}

\begin{lemma}\label{lem:consecutive}
$\dalpha(\{1,2,\dots, \ell\})=\frac{1}{\ell+1}$
\end{lemma}
\begin{proof}
The set $X=(\ell+1)\cdot \Z$ is independent in $G(\{1,2,\dots, \ell\})$ with density $\frac{1}{\ell+1}$.  Let $A$ be any independent set of $G(S)$.  For each element $a\in A$ the integers $a+1,\dots, a+\ell$ cannot be in $A$.  Therefore $\delta(A)\leq \frac{1}{\ell+1}$ for every independent set $A$ of $G(S)$.
\end{proof}

Next we prove several lemmas that are useful in determining the independence density.

\begin{lemma}\label{lem:multi}
The density of an independent set $A$ of $G(S)$ equals $\limsup_{n\to\infty} \frac{|A\cap [-nd,nd]|}{2nd+1}$ for any fixed positive integer $d$.
\end{lemma}
\begin{proof}
Recall $\delta(A)=\limsup_{m\to\infty}\frac{|A\cap [-m,m]|}{2m+1}$, where we can write $m=nd+\ell$ for $\ell \in [0,d-1]$.
Taking the limsup as $m\to\infty$ (which implies $n\to\infty$) of the following bounds
\begin{align*}
\left(\frac{2nd+1}{2m+1}\right)\frac{|A\cap [-nd,nd]|}{2nd+1}\leq \frac{|A\cap [-m,m]|}{2m+1}\leq \left(\frac{2(n+1)d+1}{2m+1}\right)\frac{|A\cap [-(n+1)d,(n+1)d]|}{2(n+1)d+1}
\end{align*}
gives $\delta(A)=\limsup_{n\to\infty}\frac{|A\cap [-nd,nd]|}{2nd+1}$.
\end{proof}

\begin{lemma}\label{lem:periodic}
Let $A$ be a periodic independent set in $G(S)$ with period $p$, and set $q=|A\cap [0,p-1]|$.  Then $\delta(A)=q/p$.
\end{lemma}
\begin{proof}
Note that for every integer $z\in \Z$, $|A\cap [z,z+p-1]|=q$.  We have the following bounds
\[\frac{2nq}{2np+1}\leq \frac{|A\cap[-np,np]|}{2np+1}\leq \frac{2nq+1}{2np+1}.\]
Taking the limsup as $n\to\infty$, we obtain $\frac{q}{p}\leq \delta(A)\leq \frac{q}{p}$ by Lemma~\ref{lem:multi}.
\end{proof}

\begin{lemma}\label{prop:multiple}
For $d \geq 1$, $\dalpha(S) = \dalpha(d\cdot S)$.
\end{lemma}
\begin{proof}
Let $A$ be an independent set in $G(d\cdot S)$ with maximum density.
Define $A_\ell:=A\cap (d\cdot Z+\ell)$, where $\ell\in [0,d-1]$, and note that the disjoint union of the $A_\ell$ sets is $A$.
Since $A_\ell$ is independent in $G(d\cdot S)$, the set $X_\ell:=(A_\ell-\ell)/d$ is independent in $G(S)$.

If $dz+\ell\in A_\ell\cap [-nd,nd]$, where $z\in \Z$, then $z\in [-n,n]$.
Hence, $|A_\ell\cap [-nd,nd]|\leq |X_\ell\cap [-n,n]|$
and we have the following
\begin{align*}
\dfrac{|A\cap [-nd,nd]|}{2nd+1}
&=\dfrac{\sum_{\ell=0}^{d-1}|A_\ell\cap[-nd,nd]|}{2nd+1}\\
&\leq \dfrac{\sum_{\ell=0}^{d-1}|X_\ell\cap[-n,n]|}{2nd+1}\frac{(2nd+1)+(d-1)}{2nd+d}\\
&=\left(1+\frac{d-1}{2nd+1}\right)\frac{1}{d}\sum_{\ell=0}^{d-1}\left(\frac{|X_\ell\cap[-n,n]|}{2n+1}\right).
\end{align*}
Taking the limsup as $n\to\infty$, we have $\dalpha(d\cdot S)=\delta(A)\leq \frac{1}{d}\sum_{\ell=0}^{d-1}\delta(X_\ell)\leq \dalpha(S)$.


Next we show that $\dalpha(d\cdot S)\geq \dalpha(S)$.
Let $X$ be an independent set in $G(S)$ with maximum density.
We claim that the set $A=\union_{\ell=0}^{d-1}(d\cdot X+\ell)$ is independent in $G(d\cdot S)$.
If not, then there exists integers $i,j$ such that $i-j=ds$ for some $s\in S$ and $i=dk_1+\ell_1$ and $j=dk_2+\ell_2$, where $k_1,k_2\in X$ and $\ell_1,\ell_2\in [0,d-1]$.
Therefore, $d(k_1-k_2)+(\ell_1-\ell_2)=ds$ which implies $\ell_1=\ell_2$ and $k_1-k_2=s$, which contradicts the assumption that $X$ is independent.

We have
$$|A\cap [-nd,nd+d-1]|=d|d\cdot X\cap [-nd,nd]|.$$
Thus we have the following bounds:
\begin{align*}
\frac{|A\cap [-nd,nd]|}{2nd+1}
&\geq \dfrac{|A\cap [-nd, nd+d-1]|}{2dn+1}
-\frac{d-1}{2nd+1}\\
&\geq \dfrac{d|d\cdot X\cap[-nd,nd]|}{2dn+d}-\frac{d-1}{2nd+1}\\
&= \dfrac{|X\cap [-n,n]|}{2n+1}-\frac{d-1}{2nd+1}.
\end{align*}
Taking the limsup as $n\to\infty$ of both sides, we have $\dalpha(d\cdot S)\geq \delta(A)\geq\delta(X)=\dalpha(S)$.
\end{proof}

By Lemma~\ref{prop:multiple}, if the greatest common divisor of $S$ is not 1, we can factor out the greatest common divisor.
The following corollary is similar to, but not implied by Theorem~\ref{thm:lz04}~\cite{LZ04}.

\begin{corollary}\label{cor:oneandmultiple}
Let $k \geq 2$ and $\ell\geq 2$.
Then $\dalpha(\{1, k, 2k, \dots, \ell k\}) = \frac{1}{\ell+1}$.
\end{corollary}

\begin{proof}
By Lemma~\ref{lem:consecutive} and Lemma~\ref{prop:multiple} we know that 
\[
    \dalpha(\{1,k,2k,\dots, \ell k\})\leq \dalpha(\{k,2k,\dots, \ell k\})
        =\dalpha(\{1,2,\dots, \ell\})=\frac{1}{\ell+1}.
\]

Let $X = (\ell+1)\cdot \Z$, and for $0\leq j\leq k-1$ let $X_j$ denote $k\cdot X+2j$.
If $k$ is odd, let $X' = \union_{j=0}^{k-1} X_j$.
Note that $X_{j_1}$ and $X_{j_2}$ are disjoint for $j_1\neq j_2$.
Let $a=k(\ell+1)x_1+2j_1$ and $b=k(\ell+1)x_2+2j_2$ be distinct elements of $X'$.
If $|a-b|\equiv 0 \pmod k$, then $2|j_1-j_2|\equiv 0 \pmod k$.
Since $k$ is odd, 2 is invertible modulo $k$, and so $j_1=j_2$.
Since $a$ and $b$ are both in $X_j$ for some $j$, $|a-b|\geq k(\ell+1)$.
Thus $|a-b|\notin \{k,2k,\dots, \ell k\}$.
Moreover, since each element from $X$ is translated by at most $2(k-1)$ positions, and $2(k-1) < k(\ell+1)-1$ we have that no two elements of $X'$ have distance 1.
Therefore, $X'$ is independent and has density $\frac{1}{\ell+1}$.

If $k$ is even, let $h =  k/2$.
Let $X' = \union_{j=0}^{h-1} X_j \cup \union_{j=h}^{k-1} (X_j+1)$.
Note that $X_j$ contains only even numbers for $0\leq j\leq h-1$ and $X_j+1$ contains only odd numbers for $h\leq j\leq k-1$.
So $X_{j_1}$ and $X_{j_2}$ are disjoint for $j_1\neq j_2$.
Let $a=k(\ell+1)x_1+2j_1+e_1$ and $b=k(\ell+1)x_2+2j_2+e_2$ be distinct elements of $X'$, where $e_1,e_2\in \{0,1\}$.
If $|a-b|\equiv 0 \pmod k$, then $e_1=e_2$ and $2|j_1-j_2|\equiv 0 \pmod k$.
Thus either $j_1,j_2<h$ and so $j_1=j_2$, or $h\leq j_1, j_2\leq k-1$ and so $j_1=j_2$.
Thus $a$ and $b$ are both in $X_j$ for some $j$, and hence $|a-b|\geq k(\ell+1)$.
Thus $|a-b|\notin \{k,2k,\dots, \ell k\}$.
Moreover, since each element from $X$ is translated by at most $2(k-1)+1$ positions, and $2(k-1)+1 < k(\ell+1)-1$ we have that no two elements of $X'$ have distance 1.
Therefore, $X'$ is independent and has density $\frac{1}{\ell+1}$.
\end{proof}


\section{The Local Discharging Lemma}
\label{sec:discharging}

In this section, we define our process of using discharging to show upper bounds on $\dalpha(S)$.
We begin by defining objects, called blocks and frames, that are crucial to our process.

Fix a distance set $S$ and an independent set $X \subseteq \Z$.
Index the elements of $X$ by $\Z$, so $X = \{ \dots, x_{-2}, x_{-1}, x_0, x_1, x_2, \dots \}$.
The $i$th \emph{block} $B_i$ is the set $B_i = \{ x_i, x_i + 1, \dots, x_{i+1}-1\}$.
Hence, each block contains exactly one element of $X$, and the blocks partition $\Z$. 
Observe that $|B_i|=x_{i+1}-x_i$, and since $X$ is independent, $|B_i| \notin S$.

Let $t$ be a positive integer.
A \emph{frame of length $t$} is a set of $t$ consecutive blocks.
For $j \in \Z$, let $F_j = \{ B_j, \dots, B_{j+t-1}\}$ be the $j$th frame of length $t$.
For a set $F$ of consecutive blocks, let $\sigma(F) = \sum_{B \in F} |B|$.
Observe that $\sigma(F)$ is the distance from the first element of $X$ in $F$ to the first element of $X$ following $F$.
Since $X$ is an independent set, $\sigma(F)\notin S$.
Throughout the rest of the paper we refer to blocks of length $i$ as $i$-blocks and frames of length $t$ as $t$-frames. 

We can describe the structure of a frame $F$ by listing the sizes of its blocks in order.
We denote such a list of sizes using \emph{block notation}, which is defined recursively.
First, any list of integers $b_1\ b_2\ \cdots\ b_k$ corresponds to $k$ consecutive blocks with sizes $b_1,\dots,b_k$.
For block structure $\pi$, the block structure $\pi^e$ corresponds to $e$ consecutive sets of blocks matching block structure $\pi$.
Finally, for two block structures $\pi_1$ and $\pi_2$, the block structure $\pi_1\, \pi_2$ corresponds to consecutive sets of blocks first matching block structure $\pi_1$ then matching the block structure $\pi_2$.
For example, the block structure $(2\ 3)^5\ 7\ (3\ 4)^2$ corresponds to 15 consecutive blocks, first with ten blocks alternating between 2- and 3-blocks, then a 7-block, then four blocks alternating between 3- and 4-blocks.
Every block structure also defines an infinite, periodic set given by repeating the block structure infinitely in both directions.

We are now prepared to discuss our discharging method.
Generally, discharging is a technique that interfaces between local structure and global averages.
We use our knowledge of local structure (the distance set) to demonstrate an upper bound on the global average (the density of an independent set).

Let $\cB$ be the set of blocks and $\cF$ be the set of frames.
Fix a \textit{charge function} $\mu : \cB \to \Z$, which is any assignment of integers to the blocks of $X$.\footnote{Note that we could relax the definition of a charge function to be fractional as in other contexts, such as coloring of planar graphs.  However, we only use integral charges in our proofs as the parameters $a$, $b$, and $t$ of the Local Discharging Lemma allow us to avoid fractional values.}
A \textit{discharging rule} is a function $d : \cB \times \cB \to \Z$ such that $d(B_i, B_j) = -d(B_j, B_i)$.
We say that a discharging rule $d$ is \emph{$S$-local} if there is a number $m = m(S)$ depending only on $S$, called the \emph{locality} of $d$, such that if $|x_i - x_j| > m$, then $d(B_i,B_j) = 0$, if $|x_i-x_j| \leq m$, then $d(B_i,B_j)$ depends only on the block structure of the blocks $B_{i-m},\dots, B_{i+m}$.
The discharging rule $d$ defines a new charge function $\mu^* : \cB \to \Z$ given by
\[
	\mu^*(B_i) = \mu(B_i) + \sum_{B_j \in \cB} d(B_j, B_i).
\]
That is, positive values of $d(B_j, B_i)$ are considered to be charge sent from $B_j$ to $B_i$ and negative values of $d(B_j, B_i)$ are considered to be charge received by $B_j$ from $B_i$.
In the second stage, we discharge on frames.
Define $\nu^*(F_j) = \sum_{i=j}^{j+t-1} \mu^*(B_i)$.
A \emph{second-stage discharging rule} is a function $d' : \cF \times \cF \to \Z$ such that $d(F_i, F_j) = -d(F_j, F_i)$.
We similarly define $d'$ to be \emph{$S$-local} there exists a \emph{locality} $m'= m'(S)$ such  that if $|i - j| > m$, then $d'(F_i,F_j) = 0$, and if $|i-j| \leq m$, then $d'(F_i,F_j)$ depends only on the block structure of the frames $F_{i-m'},\dots, F_{i+m'}$.
Finally, perform the second-stage discharging rule by defining the charge function $\nu'$ on the frames as
\[
	\nu'(F_i) = \nu^*(F_i) + \sum_{F_j \in \cF} d'(F_j, F_i).
\]

	\begin{figure}[t]
		\centering
		$\xymatrix{
			\text{\textbf{Stage 1:} \textit{Blocks}} & \mu(B_j) \ar[rr]^{\text{discharge}} &\ & \mu^*(B_j) \ar[d]^{\text{defines}} && \\
			\text{\textbf{Stage 2:} \textit{Frames}} & && \nu^*(F_j) \ar[rr]^{\text{discharge}} &\ & \nu'(F_j)
		}$
		\caption{\label{fig:dischargingmethod}The two-stage discharging method.}
	\end{figure}
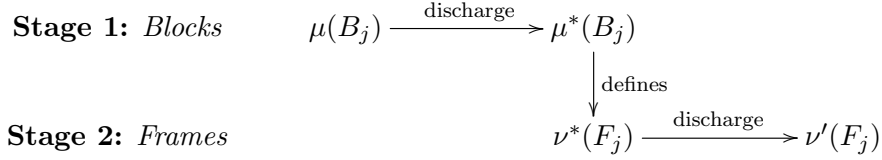

See Figure~\ref{fig:dischargingmethod} for a visualization of this discharging process.
The following lemma allows us to relate discharging functions and densities of independent sets.

\begin{lemma}[Local Discharging Lemma]\label{lem:localdischarging}
	Let $S$ be a finite, nonempty set of positive integers.
    Fix integers $a, b, t \geq 1$ and $c \geq 0$.
Let $X$ be a periodic independent set of $G(S)$.
	Initialize the charge function $\mu$ to be $\mu(B_i) = a|B_i| - b$.
	Let there be an $S$-local discharging rule $d : \cB \times \cB \to\Z$ that defines $\mu^*(B_i)$.
	On the family $\cF$ of all $t$-frames, define $\nu^*(F_j) = \sum_{B_i \in F_j} \mu^*(B_j)$, where $F_j$ is a $t$-frame, and let $d'$ be an $S$-local discharging rule $d' : \cF \times \cF \to \Z$ that defines $\nu'(F_j)$.
	If $\nu'(F_j) \geq c$ for all $j$, then $\delta(X) \leq \frac{at}{bt+c}$.
\end{lemma}

We present two proofs of the Local Discharging Lemma, and both require the set $X$ to be periodic.
In the first proof, we reduce to a finite circulant graph where the discharging rules are equivalent to the periodic set.
In the second proof, we use the limit definition to observe that the local nature causes only a finite amount of perturbation about the boundary during the density calculation.

\begin{proof}[Proof 1.]
Let $p$ be a period of the independent set $X$.
Let $D$ be the maximum locality of the two discharging rules $d$ and $d'$.
Set $q = 2Dpt$ and observe that $q$ is a period for $X$.
Thus, $X_q = X \cap [q]$ is an independent set in the circulant graph $G(q, S)$.
Let $r = |X_q|$ and let $x_1 < \cdots <x_r$ be the elements of $X_q$.
Observe that for all $i,i' \in \Z$, we have $x_i \equiv x_{i'} \pmod q$ if and only if $i \equiv i' \pmod r$.
Thus when $i \equiv i'\pmod r$ the block structure surrounding $B_i$ is equivalent to the block structure surrounding $B_{i'}$.
Further, when $i \equiv i' \pmod r$ and $j \equiv j' \pmod r$, the discharging rules $d$ and $d'$ satisfy $d(B_i, B_j) = d(B_{i'}, B_{j'})$ and $d'(F_i, F_j) = d'(F_{i'}, F_{j'})$.
Therefore, when $i \equiv i'\pmod r$, the charge functions all satisfy 
\[
\mu(B_i) = \mu(B_{i'}), \quad \mu^*(B_i) = \mu^*(B_{i'}), \quad \nu^*(F_i) = \nu^*(F_{i'}), \quad\text{and}\quad \nu'(F_i) = \nu'(F_{i'}).
\]
Observe also that
\[
t(aq-br) = t\sum_{j=1}^r \mu(B_j) = t\sum_{j=1}^r \mu^*(B_j) = \sum_{j=1}^r \nu^*(F_j) = \sum_{j=1}^r \nu'(F_j) \geq cr.
\]
From this, we have the inequality $taq \geq (tb+c)r$ and hence $\frac{ta}{tb+c} \geq \frac{r}{q} = \frac{|X_q|}{q} = \delta(X)$.
\end{proof}

\newcommand{\m}{N}
\begin{proof}[Proof 2.]
We assume that $X$ is a maximal independent set, which implies that the maximum length of a block is at most $2 \max S$.
By the locality of the first stage discharging rule $d$, there are a finite number of combinations of $2m+1$ consecutive blocks and thus a finite number of values to $d(B_i,B_j)$.
Thus, there exists a number $v$ such that $|d(B_i,B_j)| \leq v$ for all pairs of blocks $B_i, B_j$.

Since both discharging rules $d$ and $d'$ are $S$-local, the absolute differences
	\[\left|\sum_{i=1-\m}^\m \mu(B_i) - \sum_{i=1-\m}^\m \mu^*(B_i)\right|\text{ and }\left|\sum_{i=1-\m}^\m \nu^*(F_i) - \sum_{i=1-\m}^\m \nu'(F_i)\right|\] are bounded by a constant $C_1$, where $N$ is sufficiently large.
Also, since $t$ is a fixed constant, the absolute difference
\[\left|\sum_{i=1-\m}^\m \nu^*(F_i) - \sum_{i=1-\m}^\m t\mu^*(B_i)\right|\] is bounded by a constant $C_2$.
Let $C=\max\{C_1,C_2\}$.
\begin{align*}
	c \leq \frac{\sum_{i=1-\m}^\m \nu'(F_i)}{2\m}
		\leq \frac{\sum_{i=1-\m}^\m \nu^*(F_i)}{2\m}+\frac{C}{2\m}
		&\leq \frac{t\sum_{i=1-\m}^\m \mu^*(B_i)}{2\m}+\frac{2C}{2\m}\\
		&\leq \frac{t\sum_{i=1-\m}^\m \mu(B_i)}{2\m}+\frac{3C}{2\m}\\
		&= \frac{ta\sum_{i=1-\m}^\m|B_i| - 2\m tb}{2\m}+\frac{3C}{2\m}.
\end{align*}

Recall that $X$ is a periodic independent set with maximum density.
Let $p$ be the period of $X$.
Let $q=|X\cap [0,p-1]|$ and let $\m=kq+r$, where $r\in [0,q-1]$.
Note that $\sum_{i=a}^{a+q-1}|B_i|=p$ for any integer $a\in \Z$.
Therefore, we have the inequalities
\begin{align*}
	\frac{c + tb}{ta} \leq \frac{\sum_{i=1-\m}^\m|B_i|}{2\m}+\frac{3C}{2\m ta}
    &\leq \frac{\sum_{i=-kq}^{kq-1} |B_i|}{2kq}+\frac{2q}{2\m}+\frac{3C}{2\m ta}\\
    &= \frac{2kp}{2kq}+\frac{2q}{2\m}+\frac{3C}{2\m ta}\\
    &= \delta(X)^{-1}+\frac{2q}{2\m}+\frac{3C}{2\m ta}.
\end{align*}
Taking the limit as $\m\to\infty$, we have $\frac{c+tb}{ta}\leq \delta(X)^{-1}$.
\end{proof}

We use the Local Discharging as part of our method for determining independence ratios.
Suppose that we want to determine the independence ratio for some family of generator sets parameterized by $k$.
We compute values for $\dalpha(S)$ for some explicit values of $k$ using the computational techniques outlined in 
Section~\ref{sec:computation}. This leads us to a conjectured value, $\tau$ say, for $\dalpha(S)$. (Obviously $\tau$ 
will usually also depend on $k$, though we don't make it explicit in our notation here.)
To prove that $\dalpha(S)=\tau$ for each set $S$ in the family, we use the approach outlined below.
\begin{enumerate}
\item Construct a periodic independent set with density $\tau$.  This proves that $\dalpha(S)\geq \tau$.
\item Determine parameters $a$, $b$, $c$, and $t$ such that $\frac{at}{bt+c}=\tau$.
\item Let $X$ be an independent set in $G(S)$ with maximum density.
By Theorem~\ref{thm:periodic} we may assume that $X$ is periodic.
\item Construct Stage~1 discharging rules on such that $\mu^*(B)\geq 0$ for all blocks $B$.  Having all blocks have nonnegative $\mu^*$ charge makes creation of Stage~2 discharging rules easier.
\item Construct Stage~2 discharging rules such that $\nu'(F)\geq c$ for every $t$-frame $F$.
\item Deduce from the Local Discharging Lemma that $\delta(X)\leq \frac{at}{bt+c}$, i.e., that $\dalpha(S)\leq \tau$.
\end{enumerate}

The following theorem is a short example using the Local Discharging Lemma.
It provides an alternative proof of (a generalization of) part 2 of Theorem \ref{thm:3gensfractional}.

%
%
%
%

\begin{theorem}\label{thm:intervalandk}
	Let $\ell \geq 2$ and $k > \ell$. Then
	\[
		\dalpha(\{1,\dots,\ell-1, k\}) = \begin{cases}
			\frac{1}{\ell} & k \not\equiv 0 \pmod \ell\\
			\frac{k}{\ell(k+1)} & k \equiv 0 \pmod \ell
			\end{cases}
	\]
\end{theorem}

\begin{proof}
	Since $\{1,\dots,\ell-1\}$ is a subset of the generators, every block has size at least $\ell$.
	So, $\dalpha(\{1,\dots,\ell-1,k\}) \leq \frac{1}{\ell}$.
	If $k \not\equiv 0 \pmod \ell$, then the periodic set of all $\ell$-blocks is independent and equality holds.

	Otherwise, $k = \ell t$ for some integer $t$.  The periodic set (given in block notation) $\ell^{t-1}\ (\ell+1)$ is independent with density $\frac{k/\ell}{\ell t+1}=\frac{k}{\ell(k+1)}$
    To prove the upper bound we use Lemma~\ref{lem:localdischarging}, where $a = 1$, $b = \ell$, $t=k/\ell$, and $c = 1$.
	Thus, $\frac{at}{bt+c} = \frac{k/\ell}{k+1} = \frac{k}{\ell (k+1)}$.

    There are no discharging rules in this case.
    So $\mu(B_i)=\mu^*(B_i)$, and $\nu^*(F_j)=\sum_{B_i\in F_j}\mu(B_i)=\nu'(F_j)$.
    Since every block is at least an $\ell$-block, and receives initial charge $a|B_i|-b$, all blocks have non-negative charge.
    Any block of size at least $\ell+1$ has charge at least $1$.
    Since no frame has $\sigma(F) = \ell t = k$, every frame contains a block of size at least $\ell+1$.
    Hence each frame has at least $c=1$ unit of charge.
\end{proof}

\section{Discharging Arguments}
\label{sec:threegens}
 In this section we use the Local Discharging Lemma to prove exact values of $\dalpha(S)$ for several families of sets $S$ of size $3$ where $1 \in S$.
Recall that if $k, \ell$ are both odd integers, then $\dalpha(\{1,k,\ell\}) = \frac{1}{2}$.

We begin by determining the asymptotic behavior of $\dalpha( \{1, k, k+i\})$ and $\dalpha(\{1, i, k\})$ for constants $i$ and growing $k$.
We then determine the exact values for these infinite families when $i$ is a small constant.
Finally, we list some conjectures for values of the next few values of $i$.

\begin{casefig}
			\scalebox{\casefigratio}{\begin{lpic}[]{"figures/FigureKey-1-K-Kp3"(,20mm)}
			   \lbl[b]{23,18;\small Frame}
			   \lbl[b]{41,18;\small Element}
			   \lbl[b]{56,18;\small Block}
			   \lbl[br]{90,16.5; \footnotesize Possible Element}
			   \lbl[bl]{95,16.5;\footnotesize Forbidden Element}
			   \lbl[tl]{112,5;\small $\Z_n$}
			   \lbl[tr]{138,8;\footnotesize Increasing}
			   \lbl[tl]{2,8;\footnotesize Decreasing}
			   \lbl[r]{48.5,2;\footnotesize $S-\{1\}$}
			\end{lpic}}
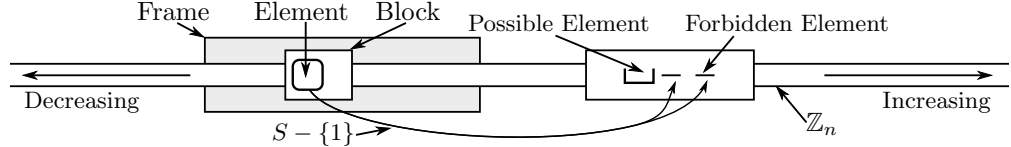
	\caption{\label{fig:blockframekey}The figure above gives a schematic of later figures we use to describe our discharging methods.}
\end{casefig}

\subsection{Asymptotic Results}

Consider sets $S = S(k)$ determined by $S(k) = \{1, f(k), g(k)\}$.
We determine the limit of $\dalpha(S(k))$ for certain functions $f(k)$ and $g(k)$.

\begin{theorem}
	For $i \geq 1$,	$\lim_{k\to\infty} \dalpha(\{1,2i+1, 2k\}) = \frac{1}{2}$.
\end{theorem}

\begin{proof}
Since $\dalpha(\{1,2i+1, 2k\}) \leq \dalpha(\{1\}) = \frac{1}{2}$, we have the upper bound immediately.

Observe that the periodic set with block structure $2^{k-1}\ (2i+3)$ is independent in $G(\{1,2i+1,2k\})$ and has density $\frac{k}{2k+2i+1}$, which tends to $\frac{1}{2}$ as $k$ grows.
\end{proof}

\begin{theorem}
	For $i \geq 1$,	$\lim_{k\to\infty} \dalpha(\{1,2i, k\}) = \frac{i}{2i+1}$.
\end{theorem}

\begin{proof}
By Theorem~\ref{thm:CHZ98C} and Observation~\ref{obs:subset} we know $\dalpha(\{1,2i, k\}) \leq \dalpha(\{1,2i\}) = \frac{i}{2i+1}$.

Let $k = (2i+1)q + r$, where $1\leq r< 2i+1$ and $q\geq 1$.
Observe that the periodic set with block structure $(2^{i-1}\ 3)^{q-1}\ (2i+2+r)$ is independent in $G(\{1,2i, k\})$, with density $\frac{iq-i+1}{(2i+1)q+r+1}$, which tends to $\frac{i}{2i+1}$ as $k$ grows ($q$ grows).
\end{proof}

\begin{theorem}
	For $i\geq 0$, we have
	$\lim_{k\to\infty} \dalpha(\{1,k, k+2i+1\}) = \frac{i+1}{2i+3}$. 
\end{theorem}

\begin{proof}
Let $k + 2i + 2 = (2i+3)q + r$, where $0 \leq r < 2i+3$.
The lower bound is given by the periodic set with block structure
$(2^i\ 3)^{q-1} (2i+3+r)$.
The density is given by $\frac{(i+1)(q-1)+1}{(2i+3)q+r} \geq \frac{(i+1)(q-1)}{(2i+3)(q+1)}=\frac{i+1}{2i+3}\frac{q-1}{q+1}$.
Note that as $k$ goes to infinity the lower bound density approaches the value $\frac{i+1}{2i+3}$.

For the upper bound, let $X$ be a periodic independent set in $G(S)$ with maximum density, and let $a=1$, $b=2$, $t=i+1$, and $c=1$.
Every block $B_j$ has nonnegative charge $\mu(B_j)$, so no Stage 1 discharging is required.
For a $t$-frame $F_j$, if $\nu^*(F_j)  =0$, then $F_j$ consists entirely of $2$-blocks and $\sigma(F_j) = 2i+2$.
Then, the elements $x_j,x_{j+1},\dots,x_{j+i},x_{j+i+1}$ have consecutive pair distances of 2, and hence the generators $k$ and $k+2i+1$ have
\[
	x_j + k+2i+1 = (x_{j+i}+k)+1,
	\quad
	x_{j+1} + k +2i+1 = (x_{j+i+1}+k)+1
\]
Thus, the consecutive elements $x_{j+i}+k, x_j+(k+2i+1), x_{j+i+1}+k, x_{j+1}+(k+2i+1)$ are not in $X$ and so are contained in a single block $B_{j'}$.
The block $B_{j'}$ has size at least five, and hence $\mu^*(B_{j'}) \geq 3$.
Let $\varphi_2$ be the function from frames $F_j$ with $\sigma(F_j) = 2i+2$ to the block $B_{j'}$ that contains $x_j + k+2i+1$.
Observe that if $\varphi_2^{-1}(B_{j'}) \neq \varnothing$, then $|B_{j'}| \geq 2|\varphi_2^{-1}(B_{j'})|+3$.
Thus, our Stage 2 discharging rule is as follows:
\begin{dischargingrule}
Stage 2: Every frame $F_j$ with $\sigma(F_j) = 2i+2$ pulls 1 unit of charge from the frame $F_{j'}$ where $B_{j'} = \varphi_2(F_j)$.
\end{dischargingrule}

If a frame $F_j$ has $\sigma(F_j) = 2i+2$, then $\nu'(F_j) = 1$.
Otherwise, $\sigma(F_j) > 2i+2$ and $F_j$ contains at least one block of size at least three, so $\nu^*(F_j) \geq 1$.
If $F_j$ loses charge in Stage 2, then $\varphi_2^{-1}(B_j) \neq \varnothing$ and $|B_j| \geq 2|\varphi_2^{-1}(B_j)|+3$, and so $\mu^*(B_j) \geq 2|\varphi_2^{-1}(B_j)|+1$.
Since $F_j$ loses at most one unit of charge for each frame in $\varphi_2^{-1}(B_j)$, $F_j$ retains at least $|\varphi_2^{-1}(B_j)|+1$ units of charge, giving $\nu'(F_j) \geq 1$.
By the Local Discharging Lemma, $\dalpha(\{1,k,k+2i+1\}) \leq \frac{at}{bt+c} = \frac{i+1}{2i+3}$.
\end{proof}

\subsection{$S = \{ 1, b, 2i\}$}

Theorem~\ref{thm:CHZ98C} states that for an odd number $b$, $\dalpha(\{1,b\}) = \frac{1}{2}$ and the maximum independent set consists entirely of 2-blocks.
For $i \geq 1$, $\dalpha(\{1,2i\}) = \frac{i}{2i+1}$ and the maximum independent set has block structure $2^{i-1}\ 3$.
We consider the union of these generators.

\begin{conjecture}
Fix $\ell \geq 3$ where $\ell$ is odd and let $2i \geq 3\ell$.
Then $\dalpha(\{1, \ell, 2i\}) = \frac{i}{2i+\ell}$.
\end{conjecture}

This conjecture is sharp, since the periodic set with block structure $2^{i-1}\ (\ell+2)$ matches this density and is independent in $G(\{1,\ell,2i\})$.
The bound $2i\geq 3\ell$ may not be sharp in all cases, but it is required for this set to be extremal, since there are independent sets of higher density even for the case $\ell = 5$ when $i$ is small.

We prove the first few cases of this conjecture.
We do not perform any discharging, so our technique is really a \emph{charging} method.
Essentially the proofs boil down to determining that $\sigma(F) \geq 2i+\ell$ for all frames of length $i$, but it is helpful to use the discharging perspective to instead show that $\nu^*(F) \geq \ell$.

\begin{theorem}
Let $i \geq 2$.
Then $\dalpha(\{1,3,2i\}) = \frac{i}{2i+3}$.
\end{theorem}

\begin{proof}
The lower bound is achieved by the periodic set with block structure $2^{i-1}\ 5$.

Let $X$ be a periodic independent set in $G(S)$ with maximum density.
Observe that since $3 \in S$ there are no 3-blocks in $X$.
Let $a = 1$, $b = 2$, $t = i$, and $c = 3$ and perform no discharging.
Consider a frame $F$ of length $t$.
Since $\sigma(F) \neq 2i$, not all blocks in $F$ are 2-blocks.
Thus, there is a block in $F$ of size at least 4.
If there is a block of size at least 5 in $F$, then $\nu^*(F) \geq 3$.
If there are two 4-blocks in $F$, then $\nu^*(F) \geq 4$.
Thus, if $\nu^*(F) < 3$ there must be $i-1$ 2-blocks in $F$ and exactly one $4$-block.
In this case, either the first block in $F$ or the last block in $F$ is a 2-block.
Removing this 2-block results in a set of $i-1$ consecutive blocks spanning $2i$ elements, a contradiction.

Thus $\nu'(F)\geq 3$ for all frames, and by the Local Discharging Lemma $\dalpha(\{1,3,2i\})\leq \frac{i}{2i+3}$.
\end{proof}

\begin{theorem}
Let $i \geq 5$.
Then $\dalpha(\{1,5,2i\}) = \frac{i}{2i+5}$.
\end{theorem}
\begin{proof}
The lower bound is achieved by the periodic set with block structure $2^{i-1}\ 7$.

Let $X$ be a perioidc independent set in $G(S)$ with maximum density.
Observe that since $5 \in S$ there are no 5-blocks in $X$.
Also, if there is a 3-block $B$ in $X$, then the blocks preceding and following $B$ are not 2-blocks.

Let $a = 1$, $b = 2$, $t = i$, and $c = 5$ and perform no discharging.
Consider a frame $F$ of length $t$.
Since $\sigma(F) \neq 2i$, not all blocks in $F$ are 2-blocks.
Thus, there is a block in $F$ of size at least 3.
If there is a block of size at least 7 in $F$, then $\nu^*(F) \geq 5$.
We now assume there are no blocks of size at least 7 in $F$.

Suppose there is no 3-block in $F$.
If there are two 6-blocks in $F$ or three 4-blocks in $F$, then $\nu^*(F) \geq 6$.
If there is a 4-block and a 6-block in $F$, then $\nu^*(F) \geq 5$.
If there are exactly two 4-blocks and $i-2$ 2-blocks in $F$, then $\nu^*(F) = 4$ and $\sigma(F) = 2i+4$.
However, if either the first or the last block $B$ of $F$ is a 4-block, then $\sigma(F-B) = 2i$, a contradiction.
Thus the first and last blocks of $F$ are 2-blocks, but removing these two blocks leaves a set of $i-2$ consecutive blocks covering $2i$ elements, a contradiction.

Therefore, there is a 3-block in $F$.
If there are at least $5$ 3-blocks in $F$, then $\nu^*(F) \geq 5$.
If there is at least one 3-block and one 6-block in $F$, then $\nu^*(F) \geq 5$.
If there are at least one 3-block and two 4-blocks in $F$, then $\nu^*(F) \geq 5$.

Hence if $\nu^*(F) < 5$, then $F$ consists of 2-blocks, at most four 3-blocks, and at most one 4-block.
Note that a 3-block and 2-block can not be consecutive, so a 4-block must be between the 3-blocks and the 2-blocks.
Thus, there exists exactly one 4-block in $F$ that separates the 2-blocks from the 3-blocks.
The cases are symmetric whether the 2-blocks or 3-blocks come before the 4-block, so we assume that $F$ has block structure $3^q\ 4\ 2^{i-q-1}$, where $1 \leq q \leq i-2$.
If $q\geq 3$, then $\nu^*(F) \geq 5$.
If $q = 1$, then the $i-1$ blocks starting at the 4-block cover exactly $2i$ elements, a contradiction.
If $q = 2$, then there are $i-3$ 2-blocks and $\sigma(F) = 2i+4$.
Since $i-3 \geq 2$, removing the last two 2-blocks from $F$ results in $i-2$ consecutive blocks covering exactly $2i$ elements, a contradiction.

Therefore, we have $\nu^*(F) \geq 5$ for all frames $F$.
\end{proof}

It is not difficult to also prove that $\dalpha(\{1,7,2i\}) = \frac{i}{2i+7}$ using similar techniques to the proofs above.
However, the case analysis becomes long and tedious, and so we do not include the proof.

\subsection{$S = \{1, 2k, 2k+2\ell\}$}

Let $S$ be a set of three generators including $1$ where the difference between the other two generators is even.
Since $\dalpha(S) = \frac{1}{2}$ when $S$ contains no even numbers, we assume the generators other than $1$ are even.
For a fixed even difference between these generators, we form the following conjecture.

\begin{conjecture}\label{thm:oddpluseven}
	Let $k \geq 1$ and $\ell \geq 1$.
	Then, $\dalpha(\{1,2k,2k+2\ell\}) = \frac{2k}{4k+2\ell}$.
\end{conjecture}

The following lemma shows that this conjecture is sharp for all possible values of $k \geq 2$ and $\ell \geq 1$.

\begin{lemma}
Let $1\leq \ell\leq k$.  Then $\dalpha(\{1,2k,2k+2\ell\}) \geq \frac{2k}{4k+2\ell}$.
\end{lemma}

\begin{proof}
When $k=\ell=1$, $\dalpha(\{1,2,4\})=\frac{2}{4+2}=\frac{1}{3}$ by Corollary~\ref{cor:oneandmultiple}.

For $k\geq 2$, consider the periodic set $X$ with block structure $2^{k-1}\ 3\ 2^{k-1}\ (2\ell+1)$.
Clearly there are no two elements in $X$ distance 1 apart.
Let $X_1$ denote the elements in $X$ described by the first $k$ blocks in the description, i.e. arising from $2^{k-1}\ 3$, and let $X_2$ denote $X \wo X_2$, the elements in $X$ arising from $2^{k-1}\ (2\ell+1)$.
Notice that all the integers in $X_1$ have the same parity, and all the integers in $X_2$ have the opposite parity to those in $X_1$.
Thus the distance between an element of $X_1$ and and element of $X_2$ is odd.

Thus if the set $X$ is not independent, then there must be two elements from $X_1$, or two from $X_2$ that are distance $2k$ or $2k+2\ell$ apart.
No pair of elements in $2^{k-1}\ 3$ are distance $2k$ or $2k+2\ell$ apart, and $2^{k-1}(2\ell+1)$ has length $2k+2\ell+1$.
No pair of elements in $2^{k-1}\ (2\ell+1)$ are distance $2k$ or $2k+2\ell$ apart, and $(2\ell+1)\ 2^{k-1}\ 3$ has length $2k+2\ell+2$.
\end{proof}

We prove a matching upper bound for small values of $\ell$.
Our proofs use the charging method, as no discharging rules are employed.

\begin{theorem}
	For $1 \leq \ell \leq 3$ and $k \geq \ell$, we have $\dalpha(\{1, 2k, 2k+2\ell\}) = \frac{2k}{4k+2\ell}$.
\end{theorem}

\begin{proof}
Let $X$ be a maximal independent set in $G(\{1,k, k+2\ell\})$, that we may assume to be periodic.
We set $a = 1$, $b = 2$, $t = 2k$, and $c = 2\ell$.
Hence $\frac{at}{bt+c} = \frac{2k}{4k+2\ell}$.

Since $1$ is a generator, all blocks have size at least two.
Every $(2+i)$-block receives charge $i \geq 0$.
We do not perform any Stage 1 discharging or Stage 2 discharging.

Observe that if a set of $k$ consecutive blocks are all 2-blocks, then these blocks span $2k$ elements, a contradiction.
%
Thus in particular every set of $k$ consecutive blocks contains at least one block of size at least $3$. 

Our approach will be to focus on a frame $F_j$ of length $2k$. If we can show that $\nu^*(F_j) \geq 2\ell$ then we will 
done by the Local Discharging lemma. We consider two halves of $F_j$; we let $H$ be the first $k$ blocks of $F_j$ and $H'$ 
be the last $k$ blocks. By the remark in the previous paragraph both $H$ and $H'$ contain at least one block of length at least 
$3$, and so $\nu^*(F_j)\ge 2$. This actually finishes the proof for $\ell=1$. We will analyze the various possibilities for the multisets 
\[
    \cB = \setof{\abs{B}}{\text{$B$ is a block in $H$ with $\abs{B}\ge 3$}} \text{ and } 
        \cB' = \setof{\abs{B}}{\text{$B$ is a block in $H'$ with $\abs{B}\ge 3$}} .
\]
By symmetry we really only care about the multiset $\{\cB,\cB'\}$. We list the possibilities below, indexed by $\nu^*(F_j)$,
the total number of blocks.
\begin{center}
    \begin{tabular}{rl}
        $\nu^*(F_j)$ & $\set{\cB,\cB'}$ \\
        \hline
        $2$ & $\set{\set{3},\set{3}}$ \\
        $3$ & $\set{\set{3},\set{4}}, \set{\set{3},\set{3,3}}$ \\
        $4$ & $\set{\set{3},\set{5}}, \set{\set{3},\set{3,4}}, \set{\set{3},\set{3,3,3}}, \set{\set{3,3},\set{3,3}}, \set{\set{4},\set{4}}, \set{\set{4},\set{3,3}} $\\
        $5$ & $\set{\set{3},\set{6}}, \set{\set{3},\set{3,5}}, \set{\set{3},\set{4,4}}, \set{\set{3},\set{3,3,4}}, \set{\set{3},\set{3,3,3,3}}, \set{\set{4},\set{5}}$ \\
            & \quad $\set{\set{4},\set{3,4}}, \set{\set{4},\set{3,3,3}}, \set{\set{3,3},\set{3,4}}, \set{\set{3,3},\set{3,3,3}}$ 
    \end{tabular}
\end{center}

We will study these various possibilities and eliminate all those with $\nu^*(F_j)<2\ell$. Some are easy to eliminate, for instance any containing $\set{4}$; in that case the $k$ blocks in that half of $F_j$ cover exactly $2k+2$ elements, but either the first or last of these is a $2$-block whose removal leaves $k-1$ blocks covering $2k$ elements, a contradiction. A similar argument shows that $\set{6}$ is not an option for $\cB$ or $\cB'$.
Let's turn then to considering $\ell=2$. We need to show $\nu^*(F_j) \geq 4$.

\begin{mycases}
\case{$\set{\cB,\cB'}=\set{\set{3},\set{3}}$.}
There can be at most $k-1$ $2$-blocks between the $3$-blocks $B_i$ and $B_{i'}$ in $H$ and $H'$ respectively. 
Thus the $k+1$ blocks starting at $B_i$ contain $B_{i'}$ and $k-1$ other 2-blocks, so these blocks span $2k+4$ elements, a contradiction.

\case{$\set{\cB,\cB'}=\set{\set{3},\set{4}}$.}
This is ruled out since it contains $\set{4}$.

\case{$\set{\cB,\cB'}=\set{\set{3},\set{3,3}}$.}
By symmetry, we assume that $H$ contains two 3-blocks and $H'$ contains exactly one.
Now $H$ covers $2k+2$ elements, so the first block $B'$ of $H'$ is not a 2-block or else $H \cup \{B'\}$ covers $2k+4$ elements, a contradiction.
Thus, $B$ is the 3-block in $H'$.
However this now means that the set of $k+1$ blocks starting at the second 3-block in $H$ contains $B$ and covers $2k+4$ elements, a contradiction.
\end{mycases}

Thus, when $\ell = 2$ all frames have at least four units of charge and so we are done by the Local Discharging Lemma.
There only now remains the case $\ell=3$, where we need to show that $\nu^*(F_j)\ge 6$. We need to eliminate any cases for $\set{\cB,\cB'}$ with $\nu^*(F_j)\le 5$, and we have already dealt with all cases where $\set{\cB,\cB'}$ contains $\set{4}$ or $\set{6}$. 

\begin{mycases}
\case{$\set{\cB,\cB'}=\set{\set{3},\set{3}}$.}
There can be at most $k-1$ $2$-blocks between the $3$-blocks $B_i$ and $B_{i'}$ in $H$ and $H'$ respectively. 
Thus, the set of $k+2$ blocks starting at $B_i$ (if $i=j$) or $B_{i-1}$ (if $i > j$) contains $B_i$, $B_{i'}$, and $k$ other 2-blocks, so these blocks span $2k+6$ elements, a contradiction.

\case{$\set{\cB,\cB'}=\set{\set{3},\set{3,3}}$.}
By symmetry, we assume that $H$ contains two 3-blocks and $H'$ contains exactly one.
Now $H$ covers $2k+2$ elements, so the first two blocks $C, C'$ of $H'$ are not both 2-blocks or else $H \cup \{C, C'\}$ covers $2k+6$ elements, a contradiction.
Thus $B'$ is one of the first two blocks in $H'$.
If the second 3-block in $H$ is not the last block in $H$, then the set of $k+2$ blocks starting at the second 3-block in $H$ contains $C$ and $C'$ and covers $2k+6$ elements, a contradiction.
Suppose the last block in $H$ is a 3-block, but the second to last block in $H$ is a 2-block.
Then starting at the 2-block to the end of the frame has length $2k+6$, a contradiction.
Thus the last two blocks of $H$ are 3-blocks.
Since $k\geq \ell\geq 3$, the first block of $H$ is a 2-block.  Omitting that block from $H$ gives $k-3$ consecutive 2-blocks and two consecutive 3-blocks which has length $2k$, a contradiction.

\case{$\set{\cB,\cB'}=\set{\set{3},\set{5}}$.}
There are at most $k-1$ 2-blocks between the 5-block and the 3-block.
Thus, there is a set of $k + 1$ consecutive blocks that contains the 5-block, the 3-block, and $k-1$ 2-blocks.
These blocks cover $2k+6$ elements, a contradiction.

\case{$\set{\cB,\cB'}=\set{\set{3,3},\set{3,3}}$.}
Observe that $\sigma(H) = \sigma(H') = 2k + 4$.
Thus, the last block of $H$ and the first block of $H'$ are not 2-blocks.
Let $x$ be the number of 2-blocks between the 3-blocks in $H$ and $y$ the number of 2-blocks between the 3-blocks in $H'$.
If $x+ y \leq k-3$, then there is a set of $k+1$ consecutive blocks containing all four 3-blocks in $F_j$ and $k-3$ 2-blocks covering $2k+6$ elements, a contradiction.
Thus, $k-2\leq x+y$.
However, there is now a set of $k-1$ consecutive blocks containing the last 3-block in $H$, the first 3-block in $H'$, and $k-3$ 2-blocks, covering $2k$ elements, a contradiction.

\case{$\set{\cB,\cB'}=\set{\set{3},\set{3,3,3}}$.}
Consider the set of $k+1$ blocks starting at the last 3-block $B$ in $H$.
Let $y$ be the number of 2-blocks between $B$ and the 3-block $B'$ in $H'$; $y \leq k-1$ by our observation at the beginning of the proof. 
If $y \leq k-3$, then the set of $k-1$ consecutive blocks starting at $B$ contains $B'$ and $k-3$ 2-blocks, covering $2k$ elements, a contradiction.
Thus $k-2 \leq y \leq k-1$.
Then the set of $k-1$ consecutive blocks starting at the 3-block in $H$ before $B$ contains $B$, does not contain $B'$, and contains $k-3$ 2-blocks, covering $2k$ elements, a contradiction.

\case{$\nu^*(F_j) = 5$.}
By symmetry, we will assume that the blocks of $H$ contribute at least as much charge as the blocks in $H'$.
Recall that $\set{\cB,\cB'}$ does not contain either $\set{4}$ or $\set{6}$.

\begin{subcases}
\subcase{$\set{\cB,\cB'}=\set{\set{5,3},\set{3}}$.}
Observe that $\sigma(H) = 2k+4$, so the first block of $H'$ must be the 3-block.
Also, let $u$ be the number of 2-blocks at the beginning of $H$ and let $v$ be the number of 2-blocks at the end of $H$.
If $u + v \geq 2$, then removing two 2-blocks from the start or end of $H$ results in a consecutive set of $k-2$ blocks covering $2k$ elements, a contradiction.
Thus, the 5-block and 3-block in $H$ are among the first two and last two blocks of $H$.
Therefore there are at least $k-3$ 2-blocks between the 5-block and the 3-block.
If the 5-block is first, consider the set of $k-1$ blocks starting after the 5-block, which contains both 3-blocks and $k-3$ 2-blocks, covering $2k$ elements, a contradiction.
If the 3-block is first, consider the set of $k-2$ blocks starting two blocks after the 3-block, which contains the 5-block from $H$, the 3-block from $H'$, and $k-4$ 2-blocks, covering $2k$ elements, a contradiction.

\subcase{$\set{\cB,\cB'}=\set{\set{5},\set{3,3}}$.}
Observe that $\sigma(H') = 2k+2$ and $\sigma(H) = 2k+3$.
If the first block in $H'$ is a 2-block, then the $k-1$ blocks after the first 2-block cover $2k$ elements, a contradiction.
Thus the first block in $H'$ is a 3-block.
However, the blocks of $H$ with the first block in $H'$ now covers $2k+6$ elements, a contradiction.


\subcase{$\set{\cB,\cB'}=\set{\set{4,4},\set{3}}$.}
Observe that $\sigma(H)=2k+4$.  If $H$ starts or ends with a 4-block, then the rest of the blocks have length $2k$.  So the first and last blocks of $H$ must be 2-blocks.  Removing the 2-blocks from the start and end of $H$ has length $2k$.

\subcase{$\set{\cB,\cB'}=\set{\set{3,3,3,3},\set{3}}$.}
Observe that $\sigma(H) = 2k+4$, so the first block of $H'$ is the 3-block.
Let $u$ be the number of 2-blocks at the beginning of $H$ and $v$ be the number of 2-blocks at the end of $H$.
If $u+ v\geq 2$, then removing two 2-blocks in total from the ends of $H$ leaves a set of $k-2$ blocks covering $2k$ elements, a contradiction.
If the last two blocks of $H$ consist of one 2-block and one 3-block then the last $k+2$ blocks in $F_j$ span $2k+6$ elements, a contradiction.
Thus, the last two blocks of $H$ consist of two 3-blocks.
Finally, the set of $k+1$ blocks starting at the \emph{second} 3-block in $H$ contains four 3-blocks and hence spans $2k+6$ elements, a contradiction.

\subcase{$\set{\cB,\cB'}=\set{\set{3,3,4},\set{3}}$.}
Observe that $\sigma(H) = 2k+4$, so the first block of $H'$ is a 3-block.
If the first or last block of $H$ is a 4-block, then removing this 4-block results in a set of $k-1$ blocks spanning $2k$ elements, a contradiction.
Let $u$ be the number of 2-blocks at the beginning of $H$ and $v$ be the number of 2-blocks at the end of $H$.
If $u+ v\geq 2$, then removing two 2-blocks in total from the ends of $H$ leaves a set of $k-2$ blocks covering $2k$ elements, a contradiction.
Thus, there exists a non-2-block among the first two blocks of $H$.
If the first non-2-block is a 3-block, then the $k+1$ blocks following this 3-block consist of a 4-block, two 3-blocks, and $(k-2)$ 2-blocks, and hence span $2k+6$ elements, a contradiction.
Otherwise, $H$ begins with a 2-block and then a 4-block and hence ends with a 3-block.
However, the set of $k-1$ blocks starting with this final 3-block consists of two 3-blocks and $(k-3)$ 2-blocks, covering $2k$ elements, a contradiction.

\subcase{$\set{\cB,\cB'}=\set{\set{3,4},\set{3,3}}$.}
Observe that $\sigma(H) = 2k+3$, so the first block of $H'$ is not a 3-block and must be a 2-block.
However, $\sigma(H') = 2k+2$ so removing this first 2-block leads to a set of $k-1$ blocks spanning $2k$ elements, a contradiction.

\subcase{$\set{\cB,\cB'}=\set{\set{3,3,3},\set{3,3}}$.}
Observe that $\sigma(H') = 2k+2$, so the first block of $H'$ must be a 3-block.
However, since $\sigma(H) = 2k+3$, the blocks of $H$ and the first 3-block in $H'$ span $2k+6$ elements, a contradiction.
\end{subcases}
\end{mycases}

In all cases, we determined that $\nu^*(F) \geq 6$, and hence $\delta(X) \leq \frac{2k}{4k+6}$ for $\ell=3$.
\end{proof}

We expect that the charging method with the parameters above will work for all cases of $1 \leq \ell \leq k$, but the details become increasingly complicated.
The general idea is that the odd blocks cannot be too small and too far away, nor can they be too small and too close.
Some balance of odd blocks being large and the right distance away is required, which is what we have in our extremal periodic set.

\subsection{$S = \{ 1, 2\ell, k\}$}

Recall that for $k > 2$, Theorem~\ref{thm:intervalandk} implies
\[
	\dalpha(\{1,2,k\}) = \begin{cases}
		\frac{k}{3k+3} & \text{if } k \equiv 0\pmod 3\\
		\frac{1}{3} & \text{if } k \equiv 1\pmod 3\\
		\frac{1}{3} & \text{if } k \equiv 2\pmod 3.
	\end{cases}
\]

The set $\{1, 2i\}$ has $\dalpha(\{1,2i\}) = \frac{1}{2i+1}$ achieved by the independent set with block structure $2^{i-1}\ 3$.
Thus, the density $\dalpha(\{1,2i,k\})$ depends on the residue class of $k$ modulo $2i+1$, depending on how this block structure interacts with the distance $k$.
For example, see Theorem~\ref{thm:1-4-k} and its accompanying Figure~\ref{fig:1-4-k}.
This figure, and others like it, contains the computed values of $\dalpha(S)$ in black and the rational functions describing $\dalpha(S)$ for these residue classes in red\footnote{For conjectures, these rational functions are conjectured extensions of the computed values.}.

\begin{center}
\begin{minipage}{3in}{\begin{theorem}\label{thm:1-4-k}
For $k > 4$,
\[\displaystyle
\dalpha(\{1,4,k\}) = \begin{cases}
 \frac{2k}{5k+5} & \text{if } k \equiv 0 \pmod 5\\
 \frac{2}{5} & \text{if } k \equiv 1 \pmod 5\\
 \frac{2k+1}{5k+5} & \text{if } k \equiv 2 \pmod 5\\
 \frac{2k-1}{5k+5} & \text{if } k \equiv 3\pmod 5\\
 \frac{2}{5} & \text{if } k\equiv 4 \pmod 5.
\end{cases}\]
\end{theorem}
}
\end{minipage}
\qquad
\begin{minipage}{3in}
\begin{figure}[H]\centering
\begin{lpic}[]{"figures/plot-1-4-k"(3in,)}
\lbl[r]{0,73;\footnotesize$\dalpha$}
\lbl[]{5,126;\scriptsize$\frac{2}{5}$}
\lbl[]{5,27.3;\scriptsize$\frac{1}{3}$}

\lbl[t]{95,0;\footnotesize $k$}
\lbl[t]{29,5;\scriptsize $10$}
\lbl[t]{48,5;\scriptsize $20$}
\lbl[t]{66,5;\scriptsize $30$}
\lbl[t]{85,5;\scriptsize $40$}
\lbl[t]{104,5;\scriptsize $50$}
\lbl[t]{122.5,5;\scriptsize $60$}
\lbl[t]{141,5;\scriptsize $70$}
\lbl[t]{159.5,5;\scriptsize $80$}
\lbl[t]{178,5;\scriptsize $90$}
\end{lpic}
\caption{\label{fig:1-4-k}Computed values of $\dalpha(\{1,4,k\})$.}
\end{figure}
\end{minipage}
\end{center}

\begin{proof}
Lower bounds are given by the periodic sets in the following table.
\begin{table}[h]
\centering\renewcommand\arraystretch{1.5}
\begin{tabular}[h]{r@{$=$}l|r@{$=$}l|l}
\multicolumn{2}{c|}{Residue class} &
\multicolumn{2}{c|}{Density} &
\multicolumn{1}{c}{Extremal Set}\\ \hline
	\quad$k$&$5i$ &\quad $\frac{2k}{5k+5}$&$\frac{2i}{5i+1}$ & $(2\ 3)^{i-1}\ 3^2$\\ \hline
	$k$&$5i+1$ & \multicolumn{2}{c|}{$\frac{2}{5}$} & $2\ 3$ \\ \hline
	$k$&$5i+2$ & $\frac{2k+1}{5k+5}$&$\frac{2i+1}{5i+3}$ & $(2\ 3)^i\ 3$ \\ \hline
	$k$&$5i+3$ & $\frac{2k-1}{5k+5}$&$\frac{2i+1}{5i+4}$ & $(2\ 3)^{i-1}\ 3^3$ \\ \hline
	$k$&$5i+4$ & \multicolumn{2}{c|}{$\frac{2}{5}$} & $2\ 3$ \\ \hline
\end{tabular}
\caption{Parameterizations and  lower bounds for $S = \{1,4,k\}$.}
\end{table}

Let $S = \{1, 4, k\}$, and let $X$ be an independent set in the distance graph generated by $S$.
Since $1 \in S$, all blocks have size at least two.
Since $4 \in S$, there are no 4-blocks and there are no consecutive pairs of 2-blocks.

Let $a = 2$ and $b = 5$ for all cases in this proof.
The initial charge $\mu(B)$ is positive except if $B$ is a 2-block, where $\mu(B) = -1$.
Perform the following first-stage discharging rule:

\begin{dischargingrule}
(S1) Every 2-block $B_j$ pulls one unit of charge from $B_{j+1}$.
\begin{casefig}
			\scalebox{\casefigratio}{\begin{lpic}[]{"figures/Obs-4pK-RuleS1"(,25mm)}
			\lbl[b]{14,18.5;\footnotesize +1}
			\end{lpic}}
		\caption{Rule (S1).}
\end{casefig}
\end{dischargingrule}

Since $\mu(B_j) = -1$ for a 2-block $B_j$, $\mu(B_j) \geq 1$ for all other blocks, and every block of size at least three loses at most one unit of charge by (S1), the rule (S1) results in all blocks having nonnegative charge.
Recall that there are no 4-blocks, so any block of size larger than three has at least 4 units of charge remaining after (S1).

If we set $t = 1$ and $c = 0$, we have an upper bound of $\frac{at}{bt+c} = \frac{2}{5}$.
This matches the lower bound for $k \equiv 1 \pmod 5$ and $k \equiv 4\pmod 5$.
We now consider the residue class of $k$ modulo $5$.

\begin{mycases}
\case{$k = 5i+2$.}
Let $t = 2i+1$ and $c = 1$.
Thus $\frac{at}{bt+c} = \frac{4i+2}{10i+6} = \frac{2i+1}{5i+3}$.
Since every frame of length $t$ contains no consecutive pairs of 2-blocks, $F$ must contain at least $i$ blocks of size at least three.
If $\nu^*(F) < 1$, then all of these blocks have size exactly three.

Suppose $\nu^*(F) = 0$.
Then there are at least $i$ 3-blocks in $F$, each of which is preceded by a 2-block.
If there are exactly $i$ 3-blocks in $F$, then $F$ has block structure $(2\ 3)^i\ 2$ and $\sigma(F) = 5i+2 = k \in S$, a contradiction.
Otherwise, there are exactly $i+1$ 3-blocks in $F$, and $F$ has block structure $3\ (2\ 3)^i$.
Since $\nu^*(F) =0$, the first 3-block in $F$ is preceded by a 2-block.
Thus, the frame $F'$ starting at this 2-block has $\sigma(F') = 5i+2 = k \in S$, a contradiction.

Therefore, all frames $F$ have $\nu^*(F) \geq 1$, and no second-stage discharging is required.

\case{$k = 5i$.}
Let $t = 2i$ and $c = 2$.
Thus $\frac{at}{bt+c} = \frac{4i}{10i+2} = \frac{2i}{5i+1}$.
Since every frame of length $t$ contains no consecutive pairs of 2-blocks, $F$ must contain at least $i$ blocks of size at least three.
If $\nu^*(F) < 2$, then all of these blocks have size exactly three.

Suppose $\nu^*(F) = 0$.
Then there are exactly $i$ 3-blocks in $F$, each of which is preceded and followed by a 2-block.
However, this implies $\sigma(F) = 5i= k \in S$, a contradiction.

Suppose $\nu^*(F) = 1$.
Then there is exactly one 3-block $B_j$ in $F$ that is preceded by a block $B_{j-1}$ with size at least three.
Every other 3-block in $F$ is preceded by a 2-block (that may or may not be in $F$).
If $B_{j-1}$ is in the frame $F$, then it is preceded by a 2-block.
Observe that $F$ has block structure $3^e\ (2\ 3)^q\ 3\ (2\ 3)^{i-1-q}\ 2^{1-e}$, where $e \in \{0,1\}$.
If $e = 0$, then $\sigma(F) = 5i = k \in S$, a contradiction.
If $e = 1$, then the first 3-block is preceded by a 2-block, so the frame $F'$ starting at that 2-block has $\sigma(F') = 5i = k \in S$, a contradiction.

Thus, all frames have at least two units of charge, and no second-stage discharging is required.

\case{$k = 5i+3$.}
Let $t = 2i+1$ and $c=3$.
Thus $\frac{at}{bt+c} = \frac{4i+2}{10i+8} = \frac{2i+1}{5i+4}$.
We use the following second-stage discharging rule.

\begin{casecomment}
\begin{dischargingrule}
(S2) If $\sigma(F_j) = 5i + 2$, then $F_j$ pulls 1 unit of charge from each of $F_{j+1}$, $F_{j+2}$, and $F_{j+3}$.
\end{dischargingrule}
\end{casecomment}

Observe that if the $2i+1$ blocks of $F_j$ cover $5i+2$ elements, then $F_j$ has block structure $2\ (3\ 2)^i$.
Thus the block $B_{j+t}$ is not a 2-, 3-, or 4-block and hence has size at least five.
Since $k > 4$, $i \geq 1$ and $t \geq 3$.
Thus the block $B_{j+t}$ is contained in each of $F_{j+1}$, $F_{j+2}$, and $F_{j+3}$.

We claim that if $\nu^*(F_j) < 3$, then $\sigma(F_j) = 5i+2$ and the rule (S2) applies to pull 3 units of charge to $F_j$.
Consider a frame $F_j$ and its charge value $\nu^*(F_j)$ before the rule (S2) is applied.

\begin{subcases}
\subcase{$\nu^*(F_j) = 0$.}
In this case, every 3-block in $F_j$ is preceded by a 2-block.
Thus, $F_j$ has block structure $3^e\ (2\ 3)^i\ 2^{1-e}$ for some $e \in \{0,1\}$.
If $e = 1$, then $\sigma(F_j) = 5i+3 = k \in S$, a contradiction.
Thus, $\sigma(F_j) = 5i+2$.
Therefore the rule (S2) applies and $F_j$ receives 3 units of charge.
Since $F_j$ does not contain a block of size at least 5, (S2) does not pull any charge from $F_j$.
Hence $\nu'(F_j) \geq 3$.

\subcase{$\nu^*(F_j) = 1$.}
In this case, exactly one 3-block in $F_j$ is not preceded by a 2-block; such a 3-block is either the first block of $F_j$ or there are two consecutive 3-blocks in $F$.
Thus, $F_j$ has block structure $3^e\ (2\ 3)^q\ (3\ 2)^{i-q}\ 3^{1-e}$, for some $q \in \{0,\dots,i\}$ and $e\in \{0,1\}$.
However, this implies that $\sigma(F_j) = 5i+3 = k \in S$, a contradiction.

\subcase{$\nu^*(F_j) = 2$.}
In this case, exactly two 3-blocks in $F_j$ are not preceded by 2-blocks.
At least one of these blocks is not the first block $B_j$ in $F_j$, and hence is preceded by a 3-block in $F_j$.

If the first block $B_j$ of $F_j$ is a 2-block, then $F_j$ has block structure \[(2\ 3)^q\ 3\ (2\ 3)^{p}\ (3\ 2)^{i-(p+q)},\] where $p, q \in \{0,\dots,i\}$ with $p+q \leq i$.
However, in this case $\sigma(F_j) = 5i+3 = k \in S$, a contradiction.

If the first block $B_j$ of $F_j$ is a 3-block with $\mu^*(B_j)=0$,  then $F_j$ has block structure \[3\ (2\ 3)^q\ 3\ (2\ 3)^{p}\ (3\ 2)^{i-(p+q+1)}\ 3,\] where $p, q \in \{0,\dots,i-1\}$ with $p+q \leq i-1$.
Since $\mu^*(B_j) = 0$, the block $B_{j-1}$ preceding the frame is a 2-block.
However, the frame $F_{j-1}$ then has $\sigma(F_{j-1}) = 5i+3 = k \in S$, a contradiction.

If the first block of $F_j$ is not preceded by a 2-block, then $F_j$ has block structure \[3\ (2\ 3)^q\ 3\ (2\ 3)^{i-(q+1)}\ 2,\]
where $q \in \{0,\dots,i-1\}$.
However, $\sigma(F_j) = 5i+3 = k \in S$, a contradiction.

\subcase{$\nu^*(F_j) \geq 3$.}
In this case, we have enough charge before using the second-stage discharging rule (S2).
Thus, we must only verify that if (S2) pulls a units of charge from $F_j$,  then $\nu^*(F_j) \geq 4$.
However, we observed earlier that $F_j$ contains a block of size at least 5, and such a block has $\mu^*$-charge at least 4.
Therefore, $\nu'(F_j) \geq 3$.
\hfill \qedhere
\end{subcases}
\end{mycases}
\end{proof}

The other cases for $S = \{1, 2\ell, k\}$ appear to have $\dalpha(S)$ given by a set of rational functions depending on the residue of $k$ modulo $2\ell+1$.

\begin{center}
\begin{minipage}{3in}{
\begin{conjecture}
For $k > 6$ with\\ $k \notin \{ 7, 10, 12, 17 \}$,
\[\dalpha(\{1,6,k\}) = \begin{cases}
\frac{3k}{7k+7} & \text{if } k \equiv 0 \pmod{7}\\
\frac{3}{7} & \text{if } k \equiv 1 \pmod{7}\\
\frac{3k+1}{7k+7} & \text{if } k \equiv 2 \pmod{7}\\
\frac{3k-2}{7k+7} & \text{if } k \equiv 3 \pmod{7}\\
\frac{3k+2}{7k+7} & \text{if } k \equiv 4 \pmod{7}\\
\frac{3k-1}{7k+7} & \text{if } k \equiv 5 \pmod{7}\\
\frac{3}{7} & \text{if } k \equiv 6 \pmod{7}.
\end{cases}\]
\end{conjecture}
}
\end{minipage}
\qquad
\begin{minipage}{3in}
\begin{figure}[H]
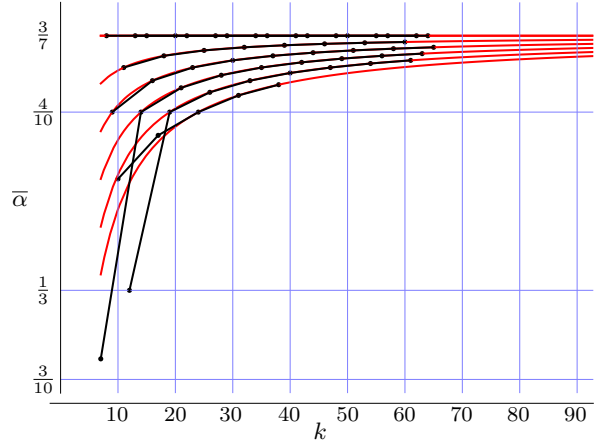
\centering
\begin{lpic}[]{"figures/plot-1-6-k"(3in,)}
\lbl[r]{0,73;\footnotesize$\dalpha$}
\lbl[]{5,126;\scriptsize$\frac{3}{7}$}
\lbl[]{5,101;\scriptsize$\frac{4}{10}$}
\lbl[]{5,43;\scriptsize$\frac{1}{3}$}
\lbl[]{5,14;\scriptsize$\frac{3}{10}$}

\lbl[t]{95,0;\footnotesize $k$}
\lbl[t]{29,5;\scriptsize $10$}
\lbl[t]{48,5;\scriptsize $20$}
\lbl[t]{66,5;\scriptsize $30$}
\lbl[t]{85,5;\scriptsize $40$}
\lbl[t]{104,5;\scriptsize $50$}
\lbl[t]{122.5,5;\scriptsize $60$}
\lbl[t]{141,5;\scriptsize $70$}
\lbl[t]{159.5,5;\scriptsize $80$}
\lbl[t]{178,5;\scriptsize $90$}
\end{lpic}
\caption{\label{fig:1-6-k}Computed values of $\dalpha(\{1,6,k\})$.}
\end{figure}
\end{minipage}
\end{center}

See Table~\ref{tab:1-6-k} for the conjectured extremal independent sets for $S = \{1,6,k\}$ and $k  \geq 21$.
We have proofs that these are extremal for all residue classes except the class $k = 7i+3$.
We omit the proofs, as they use very similar techniques to the proof of Theorem~\ref{thm:1-4-k}.

\begin{table}[htp]\centering
\scalebox{0.75}{
\renewcommand\arraystretch{1.5}
\begin{tabular}[h]{r@{$=$}l|r@{$=$}l|l}
\multicolumn{2}{c|}{Residue class} &
\multicolumn{2}{c|}{Density} &
\multicolumn{1}{c}{Extremal Set}\\ \hline
	\quad$k$&$7i$ &\quad $\frac{3k}{7k+7}$&$\frac{3i}{7i+1}$ & $(2\ 2\ 3)^{i-2}\ (2\ 3)^3$ \\ \hline
	$k$&$7i+1$ & \multicolumn{2}{c|}{$\frac{3}{7}$} & $2\ 2\ 3$ \\ \hline
	$k$&$7i+2$ & $\frac{3k+1}{7k+7}$&$\frac{3i+1}{7i+3}$ & $(2\ 2\ 3)^{i-1}\ (2\ 3)^2$ \\ \hline
	$k$&$7i+3$ & $\frac{3k-2}{7k+7}$&$\frac{3i+1}{7i+4}$ & $(2\ 2\ 3)^{i-1}\ 2\ 3\ 4\ 3$\\ \hline
	$k$&$7i+4$ & $\frac{3k+2}{7k+7}$&$\frac{3i+2}{7i+5}$ & $(2\ 2\ 3)^{i}\ 2\ 3$\\ \hline
	$k$&$7i+5$ & $\frac{3k-1}{7k+7}$&$\frac{3i+2}{7i+6}$ & $(2\ 2\ 3)^{i-2}\ (2\ 3)^4$ \\ \hline
	$k$&$7i+6$ & \multicolumn{2}{c|}{$\frac{3}{7}$} & $2\ 2\ 3$ \\ \hline
\end{tabular}
}
\caption{\label{tab:1-6-k}Parameterizations and  lower bounds for $S = \{1,6,k\}$.}
\end{table}

\begin{center}
\begin{minipage}{3in}{
\begin{conjecture}
For $k > 8$ with\\ $k \notin \{9, 10, 14, 16, 18, 23, 25, 32  \}$,
\[
\dalpha(\{1,8,k\}) = \begin{cases}
\frac{4k}{9k+9} & \text{if } k \equiv 0 \pmod{9}\\
\frac{4}{9} & \text{if } k \equiv 1 \pmod{9}\\
\frac{4k+1}{9k+9} & \text{if } k \equiv 2 \pmod{9}\\
\frac{4k+24}{9k+72} & \text{if } k \equiv 3 \pmod{9}\\
\frac{4k+2}{9k+9} & \text{if } k \equiv 4 \pmod{9}\\
\frac{4k+1}{9k+16} & \text{if } k \equiv 5 \pmod{9}\\
\frac{4k+3}{9k+9} & \text{if } k \equiv 6 \pmod{9}\\
\frac{4k-1}{9k+9} & \text{if } k \equiv 7 \pmod{9}\\
\frac{4}{9} & \text{if } k \equiv 8 \pmod{9}.
\end{cases}\]
\end{conjecture}
}
\end{minipage}
\qquad
\begin{minipage}{3in}
\begin{figure}[H]\centering
\begin{lpic}[]{"figures/plot-1-8-k"(3in,)}
\lbl[r]{0,73;\footnotesize$\dalpha$}
\lbl[]{5,126;\scriptsize$\frac{4}{9}$}
\lbl[]{5,93.5;\scriptsize$\frac{4}{10}$}
\lbl[]{5,44;\scriptsize$\frac{1}{3}$}
\lbl[]{5,19;\scriptsize$\frac{3}{10}$}

\lbl[t]{95,0;\footnotesize $k$}
\lbl[t]{29,5;\scriptsize $10$}
\lbl[t]{48,5;\scriptsize $20$}
\lbl[t]{66,5;\scriptsize $30$}
\lbl[t]{85,5;\scriptsize $40$}
\lbl[t]{104,5;\scriptsize $50$}
\lbl[t]{122.5,5;\scriptsize $60$}
\lbl[t]{141,5;\scriptsize $70$}
\lbl[t]{159.5,5;\scriptsize $80$}
\lbl[t]{178,5;\scriptsize $90$}
\end{lpic}
\caption{\label{fig:1-8-k}Computed values of $\dalpha(\{1,8,k\})$.}
\end{figure}
\end{minipage}
\end{center}

\subsection{$S = \{1, k, k + 2j+1\}$}

\begin{center}
\begin{minipage}{3in}{
\begin{theorem}[$S = \{1, k, k+1\}$]
	Let $k \geq 2$. Then,
\[	\dalpha(\{1,k,k+1\}) = \begin{cases}
			\frac{2k}{6k+3} & k \equiv 0 \pmod 3\\
			\frac{1}{3} & k \equiv 1 \pmod 3\\
			\frac{k+1}{3k+6} & k \equiv 2 \pmod 3.
		\end{cases}\]
\end{theorem}
}
\end{minipage}
\qquad
\begin{minipage}{3in}
\begin{figure}[H]\centering
\begin{lpic}[]{"figures/plot-1-k-kp1"(3in,)}
\lbl[r]{0,73;\footnotesize$\dalpha$}
\lbl[]{5,126;\scriptsize$\frac{1}{3}$}
\lbl[]{5,85.25;\scriptsize$\frac{3}{10}$}
\lbl[]{5,24.2;\scriptsize$\frac{1}{4}$}

\lbl[t]{95,0;\footnotesize $k$}
\lbl[t]{29,5;\scriptsize $10$}
\lbl[t]{48,5;\scriptsize $20$}
\lbl[t]{66,5;\scriptsize $30$}
\lbl[t]{85,5;\scriptsize $40$}
\lbl[t]{104,5;\scriptsize $50$}
\lbl[t]{122.5,5;\scriptsize $60$}
\lbl[t]{141,5;\scriptsize $70$}
\lbl[t]{159.5,5;\scriptsize $80$}
\lbl[t]{178,5;\scriptsize $90$}
\end{lpic}
\caption{\label{fig:1-k-kp1}Computed values of $\dalpha(\{1,k,k+1\})$.}
\end{figure}
\end{minipage}
\end{center}

\begin{proof}
To show our lower bounds, observe the following periodic sets are independent in the respective distance graphs:

\begin{table}[H]
\centering
\begin{tabular}[h]{c|c|c}
	Case & Parameterization & Periodic Set \\
	\hline&&\\[-2ex]
	$k \equiv 0 \pmod 3$ & $k = 3i$   & $2\ 3^{i-1}\ 5\ 3^{i-1}$	  \\[1ex]
	$k \equiv 1 \pmod 3$ & $k = 3i+1$ & $3$   \\[1ex]
	$k \equiv 2 \pmod 3$ & $k = 3i+2$ & $3^{i-1}\ 4$
\end{tabular}
\caption{\label{tab:kplus1vals}Extremal periodic sets for $S = \{1, k, k+1\}$.}
\end{table}

For the upper bounds, let $X$ be an independent set in $G(\{1,k,k+1\})$.
Throughout the proof let $a=2$ and $b=6$, so each block $B_i$ is assigned charge $\mu(B_i) = 2|B_i| - 6$.
We begin with some observations that hold for all $k$.

\begin{observation}\label{obs:twoblocks}
	Let $B_i$ be a 2-block and $\varphi_2(B_i)$ be the block containing $x_i + k$.
	Then $|\varphi_2(B_i)| \geq 5$.
	Also, if a block $B_j$ has $\varphi_2^{-1}(B_j) \neq \varnothing$, then $|B_j| \geq 2|\varphi_2^{-1}(B_j)| + 3$.
\end{observation}

\begin{casefig}
			\scalebox{\casefigratio}{\begin{lpic}[]{"figs-unique/Observation2Blocks"(,22.5mm)}
			   \lbl[t]{88,20;\small$\varphi_2$}
			   \lbl[t]{48,20;\small$\psi_2$}
				\lbl[]{68.25,10.25;\footnotesize$x_{j}$}
				\lbl[]{77.75,10.25;\tiny$x_{j+1}$}
				\lbl[b]{70,15;$B_{j}$}
				\lbl[b]{15,15;$\psi_2(B_j)$}
				\lbl[b]{123,15;$\varphi_2(B_j)$}
			\end{lpic}}
	\caption{Observation \ref{obs:twoblocks} and a 2-block $B_j$.}
\end{casefig}

\begin{casefig}
			\scalebox{\casefigratio}{\begin{lpic}[]{"figs-unique/Observation5Blocks"(,22.5mm)}
			   \lbl[t]{90,19.5;\small$\psi_2$}
			   \lbl[t]{42,19.5;\small$\varphi_2$}
				\lbl[]{59,9.75;\footnotesize$x_{k}$}
				\lbl[t]{70,12;$B_{k}$}
				\lbl[t]{15,3;\small$\varphi_2^{-1}(B_k)$}
				\lbl[b]{124,17;\small$\psi_2^{-1}(B_k)$}
				\lbl[t]{42.5,4;\footnotesize\parbox{50mm}{\centering$\leq 3t - 2(|\varphi_2^{-1}(B_k)|+1)$\\ elements}}
				\lbl[t]{104,4; \footnotesize\parbox{50mm}{\centering$\leq 3t - 2(|\psi_2^{-1}(B_k)|+1)$\\ elements}}
			\end{lpic}}
	\caption{\label{fig:obsbigblock}Observation \ref{obs:twoblocks} and a block $B_k$.}
\end{casefig}

Note that $(3+i)$-blocks (where $-1 \leq i$) have charge $2i$.
Also, if a block $B_i$ has $\varphi_2^{-1}(B_i) \neq \varnothing$, then $\mu(B_i) \geq 4|\varphi_2^{-1}(B_i)|$.

The remainder of the proof is broken into cases based on the residue of $k$ modulo $3$.

\begin{mycases}

\case{$k \equiv 1 \pmod 3$.}
Let $t = 1$, and $c = 0$.
Thus $\frac{at}{bt+c} = \frac{1}{3}$.
We apply the following discharging rule in Stage 1.

\begin{dischargingrule}
(S1a) Every 2-block $B_i$ pulls charge 2 from $\varphi_2(B_i)$.
\end{dischargingrule}

After the rule (S1a), all 2-blocks have charge zero.
If a block $B_i$ had charge reduced, then $\varphi_2^{-1}(B_i) \neq \varnothing$ and $\mu^*(B_i) \geq 2|\varphi_2^{-1}(B_i)| > 0$.
With no Stage 2 discharging, we have every frame has nonnegative charge, so $\nu'(F_j) \geq c$ always and hence $\delta(X) \leq \frac{1}{3}$.

\case{$k \equiv 0 \pmod 3$.}
Set $t = \frac{k}{3}$ and $c = 1$, thus $\frac{at}{bt+c} = \frac{2k}{6k+3}$.

For Stage 1, we use the rule (S1b) below.

\begin{casecomment}
\begin{dischargingrule}
(S1b) Every 2-block $B_i$ pulls charge 3 from $\varphi_2(B_i)$.
\end{dischargingrule}
\end{casecomment}

After Stage 1, every 2-block has charge at least 1, 3-blocks have charge zero, and 4-blocks have charge equal to 1.
If a block $B_i$ had charge pulled in Stage 1, then $\varphi_2^{-1}(B_i) \neq \varnothing$ and $\mu^*(B_i) \geq |\varphi_2^{-1}(B_i)| \geq 1$.

We do not perform any discharging in Stage 2.
Thus, if a frame $F_j$ has zero charge $\nu^*(F_j) = 0$, then all blocks in $F_j$ are 3-blocks.
This implies $\sigma(F_j) = 3t = k$, but this is a contradiction.
Therefore, $\nu^*(F_j) \geq 1$ for all frames and $\delta(X) \leq \frac{2k}{6k+3}$.

\case{$k \equiv 2 \pmod 3$\footnote{This proof is adapted almost directly from~\cite{UniqueSaturation}. We include it for completeness.}.}
Let $t = \frac{k+1}{3}$ and $c = 2$, so $\frac{at}{bt+c} = \frac{k+1}{3k+6}$.

For Stage 1, we use the rule (S1a).
We must also use Stage 2 discharging, since frames can have charge zero if they contain only 2- and 3-blocks.
However, since $\sigma(F_j) \neq 3t = k+1$ and $\sigma(F_j) \neq 3t-1 = k$, every frame with $\nu^*(F_j) = 0$ contains at least two 2-blocks that are separated only by 3-blocks.

\begin{claim}\label{claim:consec2blocks}
For two 2-blocks $B_i$ and $B_{i'}$ separated by only 3-blocks, there are strictly fewer frames containing both $B_i$ and $B_{i'}$ then the number of frames containing both $\varphi_2(B_i)$ and $\varphi_2(B_{i'})$.
\end{claim}

\begin{proof}[Proof.]
We assume $i < i'$.
If $\varphi_2(B_i) = \varphi_2(B_{i'})$, then there are $t$ frames containing $\varphi_2(B_i)$ and at most $t-1$ containing both $B_i$ and $B_{i'}$.
Hence we assume $\varphi_2(B_i) \neq \varphi_2(B_{i'})$ which implies that there is at least one 3-block between $B_i$ and $B_{i'}$ (so blocks $B_{i+1},\dots,B_{i'-1}$ are 3-blocks).
Observe $x_{i+1} + k$ is contained in $\varphi_2(B_i)$.
Further, the generators $k$ and $k+1$ ensure that the elements $x_{i+j} + k$ and $x_{i+j} + k+1$ are not contained in $X$.
Thus, between $\varphi_2(B_i)$ and $\varphi_2(B_{i'})$, every block has size at least three.
This implies that there are strictly fewer than $i'-i-1$ blocks between $\varphi_2(B_i)$ and $\varphi_2(B_{i'})$, so there are strictly more frames containing both $\varphi_2(B_i)$ and $\varphi_2(B_{i'})$ than the number of frames containing both $B_i$ and $B_{i'}$.
\end{proof}

Let $\mathcal{F}^{(1)}_{i,i'}$ be the set of frames containing both $B_i$ and $B_{i'}$ and $\mathcal{F}^{(2)}_{i,i'}$ be the set of frames containing both $\varphi_2(B_i)$ and $\varphi_2(B_{i'})$.
There exists an injection $f_{i,i}$ from $\mathcal{F}^{(1)}_{i,i'}$ to $\mathcal{F}^{(2)}_{i,i'}$.
We use the rule (S2) for Stage 2 discharging.

\begin{dischargingrule}
(S2) Every frame $F_j$ containing two 2-blocks $B_i$, $B_{i'}$ that are separated by only 3-blocks pulls charge 2 from $f_{i,i'}(F_j)$.
\end{dischargingrule}

We claim $\nu'(F_j)\geq 2$ for all frames $F_j$.
If a frame did not have charge pulled by rule (S2), then either it contained a 4-block (and has enough charge) or it received charge by rule (S2).

If a frame $F$ had charge pulled by rule (S2), then $F$ is in the image of some map $f_{i,i'}$ where $B_i$ and $B_{i'}$ are 2-blocks separated by only 3-blocks.
Let $i_1<\dots<i_\ell$ and $i_1'<\dots<i_\ell'$ be indices with $i_j < i_j'$ such that $F$ had charge 2 pulled by rule (S2) using the map $f_{i_j,i_j'}$.
Thus, $F$ contains each $\varphi_2(B_{i_j})$ and $\varphi_2(B_{i_j'})$.
Note that $i_j' \leq i_{j+1}$, so there are at least $\ell+1$ distinct 2-blocks in the list $B_{i_1}, B_{i_1'}, \dots, B_{i_\ell}, B_{i_\ell'}$.
So, $|\union_{B \in F} \varphi_2^{-1}(B)|\geq \ell+1$.
Finally, $\nu^*(F) \geq \sum_{B\in F} \mu^*(B) \geq \sum_{B \in F} 2|\varphi_2^{-1}(B)| \geq 2|\union_{B \in F} \varphi_2^{-1}(B)| \geq 2\ell + 2$.
Thus, $\nu'(F) = \nu^*(F) - 2\ell \geq 2$, which completes our proof that $\delta(X) \leq \frac{k+1}{3k+6}$.
\end{mycases}
\end{proof}

\begin{center}
\begin{minipage}{3in}{
\begin{theorem}
	Let $k \geq 3$. Then,
\[	\dalpha(\{1,k,k+3\}) = \begin{cases}
			\frac{2k+5}{5k+20} & k \equiv 0 \pmod 5\\
			\frac{2}{5} & k \equiv 1 \pmod 5\\
			\frac{2k+6}{5k+20} & k \equiv 2 \pmod 5\\
			\frac{4k+3}{10k+15} & k \equiv 3 \pmod 5\\
			\frac{2k+7}{5k+20} & k \equiv 4 \pmod 5.
		\end{cases}\]
\end{theorem}
}
\end{minipage}
\qquad
\begin{minipage}{3in}
\begin{figure}[H]\centering
\begin{lpic}[]{"figures/plot-1-k-kp3"(3in,)}
\lbl[r]{0,73;\footnotesize$\dalpha$}
\lbl[]{5,126;\scriptsize$\frac{2}{5}$}
\lbl[]{5,27;\scriptsize$\frac{1}{3}$}

\lbl[t]{95,0;\footnotesize $k$}
\lbl[t]{29,5;\scriptsize $10$}
\lbl[t]{48,5;\scriptsize $20$}
\lbl[t]{66,5;\scriptsize $30$}
\lbl[t]{85,5;\scriptsize $40$}
\lbl[t]{104,5;\scriptsize $50$}
\lbl[t]{122.5,5;\scriptsize $60$}
\lbl[t]{141,5;\scriptsize $70$}
\lbl[t]{159.5,5;\scriptsize $80$}
\lbl[t]{178,5;\scriptsize $90$}
\end{lpic}
\caption{\label{fig:1-k-kp3}Computed values of $\dalpha(\{1,k,k+3\})$.}
\end{figure}
\end{minipage}
\end{center}
\begin{proof}
To show our lower bounds, observe the following periodic sets are independent in the respective distance graphs:

\begin{table}[H]
\centering
\begin{tabular}[h]{c|r@{\,$=$\,}l|r@{\,$=$\,}l|c}
     Case  &   \multicolumn{2}{c|}{Parameterization}   &   \multicolumn{2}{c|}{Parameterized Density}    & Periodic Set \\
     \hline&\multicolumn{2}{c|}{}&\multicolumn{2}{c|}{}&\\[-2ex]
     $k \equiv 0 \pmod 5$  &  \qquad $k$ & $5i$    & $\frac{2k+5}{5k+20}$&$\frac{2i+1}{5i+4}$ &  $(2\ 3)^{i-1}\ 3^4$  \\[1ex]
     $k \equiv 1 \pmod 5$  &   $k$&$5i+1$  & $\frac{2}{5}$&$\frac{2}{5}$&   $2\ 3$   \\[1ex]
     $k \equiv 2 \pmod 5$  &   $k$&$5i+2$  & $\frac{2k+6}{5k+20}$&$\frac{2i+2}{5i+6}$&   $(2\ 3)^i\ 3^2$ \\[1ex]
     $k \equiv 3 \pmod 5$  &   $k$&$5i+3$ &  \qquad$\frac{4k+3}{10k+15}$&$\frac{4i+3}{10i+9}$ &   $(2\ 3)^i\ 2\ (2\ 3)^i\ 2\ 5$ \\[1ex]
     $k \equiv 4 \pmod 5$  &   $k$&$5i+4$ &  $\frac{2k+7}{5k+20}$&$\frac{2i+3}{5i+8}$&   $(2\ 3)^{i+1}\ 3$
\end{tabular}
\caption{\label{tab:kplus3vals}Extremal periodic sets for $S = \{1, k, k+3\}$.}
\end{table}

Let $X$ be an independent set in $G(\{1,k,k+3\})$, we will show the density of $X$ is bounded above by the prescribed values.
Always, let $a = 2$ and $b = 5$.
Note that there are no 1-blocks, every 2-block has charge $-1$, and every larger block has positive charge.

\begin{definitiontheorem}
For a 2-block $B_i$, the elements $x_i + k + 3$ and $(x_i + 2)+k$ are both forbidden to be in $X$.
Thus, there is a block $\varphi_2(B_i)$ that contains both $x_i+k+2$ and $x_i+k+3$.
There is also a block $\psi_2(B_i)$ that contains both $x_i-(k + 2)$ and $x_i-(k + 1)$.
Hence $\varphi_2(B_i)$ and $\psi_2(B_i)$ have size at least three and can not have size four.
For a block $B$, $\varphi^{-1}_2(B)$ is the set of 2-blocks $B_i$ such that $\varphi(B_i)=B$.
\end{definitiontheorem}

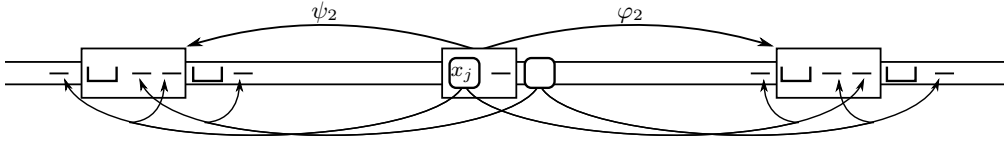
\begin{casefig}
			\scalebox{\casefigratio}{\begin{lpic}[]{"figures/Obs-Kp3-Phi2"(,20mm)}
			\lbl[b]{64.3,9;\scriptsize $x_j$}
			\lbl[b]{45,17.5;\footnotesize $\psi_2$}
			\lbl[b]{88,17.5;\footnotesize $\varphi_2$}
			\end{lpic}}
\end{casefig}

\begin{observationtheorem}\label{obs:kplus3psi2}\label{obs:kplus3varphi2}
     If $B_j,B_{j+1}\dots,B_{i+\ell-1}$ are $\ell$ consecutive 2-blocks, then $\varphi_2(B_j) = \varphi_2(B_{j+s})$ and $\psi_2(B_j) = \psi_2(B_{j+s})$ for all $s \in \{1,\dots,\ell-1\}$ and $|\varphi_2(B_j)| \geq 2\ell+1$.
\end{observationtheorem}
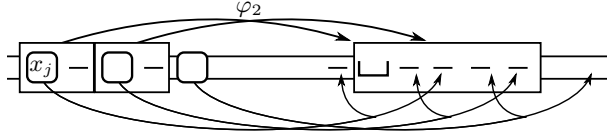
\begin{casefig}
			\scalebox{\casefigratio}{\begin{lpic}[]{"figures/Obs-Kp3-Consec2"(,20mm)}
			\lbl[b]{4.8,9;\footnotesize $x_j$}
			\lbl[b]{34,17.5;\footnotesize $\varphi_2$}
			\end{lpic}}
\end{casefig}

For all cases below, we use the following Stage 1 discharging rule.

\begin{dischargingrule}
(S1) Every 2-block $B_i$ pulls charge 1 from $\varphi_2(B_i)$.
\end{dischargingrule}

Observe that after applying the Stage 1 rule (S1), all blocks have nonnegative $\mu^*$-charge, and hence all frames have nonnegative $\nu^*$-charge.
Further observe the following:
\begin{cem}
\item A 3-block $B_j$ has $\mu^*$-charge 0 if and only if $B_j = \varphi_2(B)$ for some 2-block $B$.
\item 4-blocks have $\mu^*$-charge 3.
\item Every block $B$ has size at least $2|\varphi^{-1}_2(B)|+1$.
\item Every block $B$ with $|B| \geq 3$ has $\mu^*$-charge $2|B|-5-|\varphi^{-1}_2(B)|\geq 3|\varphi^{-1}_2(B)|-3$
\end{cem}
Thus, if a frame $F$ has charge strictly less than three, then the frame $F$ can only contain 2- and 3-blocks.
We say that a 3-block is \emph{heavy} if it has $\mu^*$-charge 1.
Thus, if $\nu^*(F) < 3$, then $\nu^*(F)$ is equal to the number of heavy 3-blocks in $F$.

%

\begin{observationtheorem}\label{obs:2con3blocks}
For every pair of consecutive 3-blocks $B_j$, $B_{j+1}$, at least one of $B_j$ or $B_{j+1}$ is heavy.
\end{observationtheorem}

\begin{observationtheorem}\label{obs:3con3blocks}
If $B_j, B_{j+1}$, and $B_{j+2}$ are three consecutive 3-blocks where only $B_{j+1}$ is heavy, then the 2-blocks in $\varphi_2^{-1}(B_j)$ and $\varphi_2^{-1}(B_{j+2})$ are separated by a 4-block, and there are exactly $k-7$ elements between that 4-block and $B_j$.
\end{observationtheorem}

Recall that all frames have nonnegative $\mu^*$-charge after applying (S1), so all frames have nonnegative $\nu^*$-charge.
Since $f_j$ is an injection each frame $F$ has the (S2) rule applied at most once to $F$.
The table below shows the values of $t$ and $c$ for the various cases based on the residue of $k$ modulo $5$.

\begin{table}[H]
\centering
\begin{tabular}[h]{r@{\,$=$\,}l|r@{\,$=$\,}l|r@{\,$=$\,}l|c}
     \multicolumn{2}{c|}{Case}  &   \multicolumn{2}{c|}{$t$ value}   &   \multicolumn{2}{c|}{$c$ value}    & Parameterized Density \\
     \hline \multicolumn{2}{c|}{}&\multicolumn{2}{c|}{}&\multicolumn{2}{c|}{}&\\[-2ex]
     $k$  & $5i$ & $t$ & $2i+1$ &\, $c$ & $3$ & $\dfrac{2i+1}{5i+4}$ \\[2ex]
     $k$  & $5i+1$ & $t$ & $1$ & $c$ & $0$ & $\dfrac{2}{5}$ \\[2ex]
     $k$  & $5i+2$ & $t$ & $2i+2$ & $c$ & $2$ & $\dfrac{2i+2}{5i+6}$ \\[2ex]
     $k$  & $5i+3$ & $t$ & $4i+3$ & $c$ & $3$ & $\dfrac{4i+3}{10i+9}$ \\[2ex]
     $k$  & $5i+4$ & $t$ & $4i+6$ & $c$ & $2$ & $\dfrac{2i+3}{5i+8}$ \\[2ex]
\end{tabular}
\caption{\label{tab:kplus3valstc}Values for $t$ and $c$ for $S = \{1, k, k+3\}$.}
\end{table}

Suppose $k \equiv 1 \pmod 5$.
Use Stage 1 discharging rule (S1) to distribute charge to the 2-blocks.
Observe all blocks have nonnegative $\mu^*$-charge, and hence all frames have nonnegative $\nu^*$-charge.
We do not discharge in Stage 2 and find $\delta(X)\leq \frac{2}{5}$.

For the other cases we need to use the stage two discharging rules.
First we prove several results that show frames with small charge can not have certain structure.  Let $F$ be any frame with charge $0\leq \nu^*(F)\leq 2$, which implies $F$ can only have blocks of size 2 or 3.

\begin{claim}\label{claim:no23}
If $k\not\equiv 1 \pmod 5$, then every frame $F$ with $\nu^*(F) < 3$ contains either two consecutive 2-blocks or a heavy 3-block.
\end{claim}
\begin{proof}
If $F$ has $\nu^*$-charge less than three, then $F$ contains no block of size at least four.
If it also contains no consecutive 2-blocks or heavy 3-blocks, then the frame has block structure $3^e\ (2\ 3)^p\ 2^f$, where $e,f\in \{0,1\}$, $p$ is a nonnegative integer, and $e+f+2p=t$.
\begin{mycases}
\case{$k\equiv 0 \pmod 5$.} In this case $(2\ 3)^i$ appears which has length $5i=k\in S$.
\case{$k\equiv 2 \pmod 5$.}  In this case $(2\ 3)^i\ 2$ appears which has length $5i+2=k\in S$.
\case{$k\equiv 3 \pmod 5$.}  In this case $3\ (2\ 3)^i$ appears which has length $5i+3=k\in S$.
\case{$k\equiv 4 \pmod 5$.}  In this case $(2\ 3)^{i+1}\ 2$ appears which has length $5i+7=k+3\in S$.
\end{mycases}
\vspace{-2em}
\end{proof}

\begin{claim}\label{claim:no22}
If $k\not\equiv 1 \pmod 5$, then there is no frame $F$ with $\nu^*(F) < 3$ with exactly one pair of consecutive 2-blocks and no heavy 3-block.
\end{claim}

\begin{proof}
Suppose $F$  has exactly one pair of consecutive 2-blocks and no 3-blocks are heavy.
Thus, every 3-block in $F$ is preceded and followed by a 2-block.
The frame has block structure $3^e\ (2\ 3)^p\  2\ 2\ (3\ 2)^q\  3^f$, where $e,f\in \{0,1\}$ and $p$ and $q$ are nonnegative integers, where $e+f+2p+2q+2=t$.

\begin{mycases}
\case{$k\equiv 0 \pmod 5$.}
If $p\geq i$ or $q\geq i$, then $(2\ 3)^i$ or $(3\ 2)^i$ appears and has length $k$.
Otherwise $p=q=i-1$, and either $e$ or $f$ is 1.
If $e=1$, then $3\ (2\ 3)^{i-1}\ 2$ has length $5i$ and if $f=1$, then $2\ (3\ 2)^{i-1}\ 3$ has length $5i$.

\case{$k\equiv 2 \pmod 5$.}
Notice that in all cases $3^e\  (2\ 3)^p\ 2$ or $2\ (3\ 2)^q\ 3^f$ contains $3\ (2\ 3)^i\ 2$ or $2\ (3\ 2)^i\ 3$ which has length $k$.

\case{$k\equiv 3 \pmod 5$.}  Notice that $(3\ 2)^i\ 3$ and $3\ (2\ 3)^i$ have length $5i+3$ and at least one must appear in the frame.

\case{$k\equiv 4 \pmod 5$.}  Notice that $(2\ 3)^i\ 2\ 2$ and $2\ 2\ (3\ 2)^i$ have length $5i+4$ and at least one must appear in the frame.
\end{mycases}
\vspace{-2em}
\end{proof}

\begin{claim}\label{claim:no33}
Suppose $k\not\equiv 1 \pmod 5$ and $F$ is a frame with $\nu^*(F) < 3$, exactly one heavy 3-block, and no consecutive pair of 2-blocks.
Then $F$ has at least three consecutive 3-blocks.
\end{claim}
\begin{proof}
By Claim~\ref{claim:no22}, if $\nu^*(F) < 3$ and $F$ contains no consecutive pair of 2-blocks, then every 2-block is preceded and followed by a 3-block.
By the proof of Claim~\ref{claim:no22}, the heavy 3-block must be in a pair of consecutive 3-blocks; suppose this pair is not in a set of three consecutive 3-blocks.
Thus the frame has block structure $2^e\ (3\ 2)^p\  3\ 3\  (2\ 3)^q\  2^f$, where $e,f\in \{0,1\}$, $p$ and $q$ are nonnegative integers, and $e+f+2p+2q+2=t$.

\begin{mycases}
\case{$k\equiv 0 \pmod 5$.}
Since $t = 2i+1$ is odd, $e+f = 1$ and hence $p+q = i-1$.
The blocks with structure $(3\ 2)^p\ 3\ 3\ (2\ 3)^q$ cover exactly $5i+1$ elements.
Since $e+f = 1$, there is a 2-block either preceding or following these blocks, so $\sigma(F) = 5i+3 = k+3 \in S$, a contradiction.

\case{$k\equiv 2 \pmod 5$.}
Since $t = 2i+2$ is even, $e = f$.
If $e = f = 1$, then $p + q = i-1$ and $\sigma(F) = 5i+5 = k+3 \in S$, a contradiction.
Thus, $e = f = 0$, $p + q = i$, and $\sigma(F) = 5i+6$.
Let $B$ be the first 2-block in $F$.
Then there is a distance of $5i+4 = k+1$ from the start of $B$ to the first element to the right of $F$.
Thus, $\varphi_2(B)$ is the block immediately after $F$.
This means that the final 3-block in $F$ is heavy, so there are at least two heavy 3-blocks in $F$.

\case{$k\equiv 3 \pmod 5$.}
Since $t = 4i+3$ is odd, $e + f = 1$ and $p+q = 2i-1$.
This implies that the blocks with structure $(3\ 2)^p\ 3\ 3\ (2\ 3)^q$ contains a consecutive set of blocks covering exactly $5i+6$ blocks.
Since $5i+6 = k+3 \in S$, this is a contradiction.

\case{$k\equiv 4 \pmod 5$.}
Since $t = 4i+6$ is even, $e = f$ and $p + q \geq 2i+2$.
Thus, at least one of $p$ and $q$ is at least $i + 1$.
If $e = 1$ and $p \geq i+1$, then the first $2i+3$ blocks of $F$ have block structure $2\ (3\ 2)^{i+1}$ and cover $5i+7$ elements; $5i+7 = k+3 \in S$, a contradiction.
If $f = 1$ and $q \geq i+1$, then the last $2i+3$ blocks of $F$ have block structure $(2\ 3)^{i+1}\ 2$ and cover $5i+7$ elements; $5i+7 = k+3 \in S$, a contradiction.
If $e = f = 0$, then at least one of $p$ or $q$ is at least $i + 2$, and the block structure $(2\ 3)^{i+1}\ 2$ appears in $F$, a contradiction.
\end{mycases}
\vspace{-2em}
\end{proof}

\begin{claim}\label{claim:no333}
If $k\not\equiv 1 \pmod 5$, then any frame $F$ with $\nu^*(F) < 3$ and no consecutive pair of 2-blocks has at least two heavy 3-blocks.
\end{claim}
\begin{proof}
By Claim~\ref{claim:no33}, such a frame with exactly one heavy 3-block must have three consecutive 3-blocks.
The frame has block structure $2^e\ (3\ 2)^p\  3\ 3\ 3\  (2\ 3)^q\ 2^f$, where $e,f\in \{0,1\}$, $p$ and $q$ are nonnegative integers, and $e+f+2p+2q+3=t$.
Note that we are in the case where Observation~\ref{obs:3con3blocks} applies.

\begin{mycases}
\case{$k\equiv 0 \pmod 5$.}
Since $t = 2i+1$ is odd, $e = f$.
If $e = f = 1$, then $F$ has block structure $2\ (3\ 2)^p\ 3\ 3\ 3\ (2\ 3)^q\ 2$ and $\sigma(F) = 5i+3 = k + 3 \in S$, a contradiction.
Thus, $e = f = 0$ and $F$ has block structure $(3\ 2)^q\ 3\ 3\ 3\ (2\ 3)^q$ and $\sigma(F) = 5i+5$.
However, if we ignore the first two blocks in $F$, the remaining $2i$ blocks cover exactly $5i = k$ elements, a contradiction.

\case{$k\equiv 2 \pmod 5$.}
Observe that $\sigma(F) \geq 5i+3$.
Let $B_j, B_{j+1},$ and $B_{j+2}$ be the three consecutive 3-blocks in $F$, where $B_j$ and $B_{j+2}$ are not heavy by assumption.
Let $B_i, B_{i+1}$, and $B_{i+2}$ be the 2-, 4-, and 2-blocks where $\varphi_2(B_i) = B_j$ and $\varphi_2(B_{i+2}) = B_{j+2}$.
By Observation~\ref{obs:3con3blocks}, there are exactly $k-7$ elements strictly between $B_{i+1}$ and $B_{j}$.
In particular, every 2-block $B$ in this region is either contained within $F$ or $\varphi_2(B)$ is contained in $F$.
Moreover, since $k-7 = 5i-5$ we have that this number is a multiple of 5.
Thus, if the blocks between $B_{i+1}$ and $B_j$ alternate between 2- and 3-blocks, they start with the block structure $(2\ 3)^*$ but end with the block structure $(3\ 2)^*$.
This implies that one of the following cases occurs between $B_{i+1}$ and $B_j$: either there exists a block of size at least four, a pair of consecutive 3-blocks, or a pair of consecutive 2-blocks.
If there is a block of size four, it appears before $F$, but then there exists a heavy 3-block in $F$ to the right of $B_{j+2}$, so $F$ contains at least two heavy 3-blocks.
If there is a consecutive pair of 3-blocks, at least one of them is heavy and hence appears to the left of $F$; however this implies there is a heavy 3-block in $F$ to the right of $B_{j+2}$ and $F$ contains at least two heavy 3-blocks.
Finally, if there is a pair of consecutive 2-blocks $B$ and $B'$, then $\varphi_2(B)$ is in $F$ and thus $F$ contains a large block, a contradiction.

\case{$k\equiv 3 \pmod 5$.}
Since $t = 4i+3$, $p + q \geq 2i-1$.
Thus $F$ contains a consecutive set of blocks with block structure $2\ (3\ 2)^a\ 3\ 3\ 3\ (2\ 3)^b$ where $a + b = i-1$ and these blocks cover $5i+6$ elements, and $5i+6 = k+3 \in S$, a contradiction.

\case{$k\equiv 4 \pmod 5$.}
Since $t = 4i+6$, $p+q \geq 2i$.
Thus $F$ contains a consecutive set of blocks with block structure $(3\ 2)^a\ 3\ 3\ 3\ (2\ 3)^b$, where $a + b = i-1$ and these blocks cover $5i+4$ elements, and $5i+4 = k \in S$, a contradiction.
\end{mycases}
\vspace{-2em}
\end{proof}

\begin{claim}\label{claim:no3333}
If $k\equiv 0 \pmod 5$ or $k\equiv 3 \pmod 5$, then any frame $F$ with no pair of consecutive 2-blocks and  at least two disjoint pairs of consecutive 3-blocks has $\nu^*(F) \geq 3$.
\end{claim}
\begin{proof}
Suppose $\nu^*(F) < 3$, so $F$ contains no blocks of size at least four.
If $F$ contains at least two disjoint pairs of consecutive 3-blocks, then there exists at least one 2-block between these pairs, or these pairs form a sequence of four consecutive 3-blocks.
Thus, $F$ has  block structure $2^e\ (3\ 2)^p\ 3\ 3\  2\ (3\ 2)^q\  3\ 3\  (2\ 3)^r\ 2^f$, where $e,f\in \{0,1\}$, $p$, $q$ and $r$ are nonnegative integers, and $e+f+2p+2q+2r+5=t$, or $F$ has block structure $2^e\ (3\ 2)^p\ 3\ 3\ 3\ 3\ (2\ 3)^q\ 2^f$ where $e+f+2p+2q+4=t$.

\begin{mycases}
\case{$k\equiv 0 \pmod 5$.}
Suppose  $F$ has  block structure $2^e\ (3\ 2)^p\ 3\ 3\  2\ (3\ 2)^q\  3\ 3\  (2\ 3)^r\ 2^f$.
Since $t = 2i+1$ is odd, $e = f$.
If $e = f = 1$, then $\sigma(F) = 5i+2$, but ignoring the first 2-block reveals a set of consecutive blocks covering exactly $k$ elements, a contradiction.
If $e = f = 0$, then $\sigma(F) = 5i+4$.
However, the final 3-block in $F$ appears $k+2$ positions to the right of the first 3-block, so this final 3-block is heavy.
Thus, there are at least three heavy 3-blocks in $F$ and $\nu^*(F) \geq 3$.

Now suppose $F$ has block structure $2^e\ (3\ 2)^p\ 3\ 3\ 3\ 3\ (2\ 3)^q\ 2^f$.
Since $t = 2i+1$ is odd, $e = f$.
If $e = f= 1$, then $\sigma(F) = 5i+3 = k +3 \in S$, a contradiction.
Thus $e = f = 0$ and $\sigma(F) = 5i+4$.
However, the final 3-block in $F$ appears $k+2$ positions to the right of the first 3-block, so this final 3-block is heavy.
Thus, there are at least three heavy 3-blocks in $F$ and $\nu^*(F) \geq 3$.

\case{$k\equiv 3 \pmod 5$.}
Suppose  $F$ has  block structure $2^e\ (3\ 2)^p\ 3\ 3\  2\ (3\ 2)^q\  3\ 3\  (2\ 3)^r\ 2^f$ where $e, f \in \{0,1\}$, $p, q, r \geq 0$, and $e + f + 2p + 2q + 2r + 5 = t = 4i+3$.
Since $4i+3$ is odd, $e=f$ and hence $p + q + r \geq 2i-2$.
If $p + q \geq i$, then there exists a set $F$ of $2i+2$ blocks surrounding the first pair of consecutive 3-blocks such that $\sigma(F) = 5i+6 = k+3\in S$, a contradiction and hence $p +q \leq i - 1$.
Similarly, $q + r \leq i-1$.
Thus $2i-2 \leq p+q+r \leq p+2q+r \leq 2i-2$, and equality holds.
Therefore $q = 0$ and $p + r = 2i-2$.
This implies that either $p \geq i-1$ or $r \geq i-1$.
Thus at least one of the block structures $(2\ 3)^{i-1}\ 3\ 3\ 2\ 3$ or $3\ 2\ 3\ 3\ (2\ 3)^{i-1}$ appears within $F_j$ and these structures cover exactly $5i+6$ elements, a contradiction.
\end{mycases}
\vspace{-2em}
\end{proof}

Now we show that after Stage 2 discharging each frame $F$ has charge at least $c$.
We break these cases by the residue class of $k$, then consider the $\nu^*$-charge on $F$.
We order our cases by the fraction $\frac{ci}{t}$ in increasing order.

For Cases~\ref{case:kp34mod5} and \ref{case:kp33mod5} we use the following second-stage discharging rule:

Suppose $B_j$, $B_{j+1}$, and $B_{j+2}$ are consecutive 2-blocks.
There is an injection $f_j$ from frames containing both $B_j$ and $B_{j+1}$ to frames containing $\varphi_2(B_i)$ where for $f_j(F)=F'$ the block $B_j$ appears in the same position of $F$ as $\varphi_2(B_j)$ appears in $F'$.
\begin{dischargingrule}
(S2a) Let $B_j, B_{j+1}, B_{j+2}$ be the first triple of consecutive 2-blocks in a frame $F$.
Then $F$ pulls charge $\max\{0, c-\nu^*(F)\}$ from $f_j(F)$.
\end{dischargingrule}

Suppose $B_j$ and $B_{j+1}$ are consecutive 2-blocks.
There is an injection $g_j$ from frames containing both $B_j$ and $B_{j+1}$ to frames containing $\varphi_2(B_i)$ where for $g_j(F)=F'$ the block $B_j$ appears in the same position of $F$ as $\varphi_2(B_j)$ appears in $F'$.

\begin{dischargingrule}
(S2b) Suppose that $F$ has no triple of consecutive 2-blocks.
Let $B_j, B_{j+1}$ and $B_i,B_{i+1}$ be the first two disjoint pairs of consecutive 2-blocks in $F$, where all blocks between $B_{j+1}$ and $B_i$ have size at most 3.
Then $F$ pulls charge $\max\{0, c-\nu^*(F)\}$ from $g_j(F)$.
\end{dischargingrule}

Observe that when (S2b) applies to a frame $F$, the requirement that all blocks between $B_{j+1}$ and $B_i$ have size at most 3 implies that $g_j(F)$ contains both $\varphi_2(B_{j})$ and $\varphi_2(B_i)$.

\begin{mycases}
\case{\label{case:kp34mod5}$k = 5i+4$, $t = 4i+6$, and $c = 2$.}
Consider a frame $F$.
If $F$ contains no pair of consecutive 2-blocks, then by Claim~\ref{claim:no333} we have $\nu^*(F) \geq 2$.
If $F$ contains a pair $B_j$, $B_{j+1}$ of consecutive 2-blocks, then $F$ contains one of $\varphi_2(B_j)$ or $\psi_2(B_j)$, and hence $\nu^*(F) \geq 3$.

Observe that if the rule (S2a) pulls charge from $F$, then $F$ loses at most 2 units of charge.
Moreover, the definition of (S2a) implies that there exists a block $B$ in $F$ of size at least 7, and thus $\nu^*(F) \geq 6$, so $\nu'(F) \geq 4$.

Observe that if the rule (S2b) pulls charge from $F$, then $F$ loses at most 2 units of charge.
Moreover, the definition of (S2b) implies that there exists blocks $B$ and $B'$ in $F$, each of size at least 6, and thus $\nu^*(F) \geq 6$, so $\nu'(F) \geq 4$.

Thus, suppose $\nu^*(F) < 3$, so all blocks in $F$ have size at most three.
If $F$ contains a triple of consecutive 2-blocks, then the rule (S2a) applies and $\nu'(F) = 3$.
If $F$ contains two pairs of consecutive 2-blocks, then the rule (S2b) applies and $\nu'(F) = 3$.
Thus, assume that there exists no triple of consecutive 2-blocks and exactly one pair $B_j, B_{j+1}$ of consecutive 2-blocks.
Then, other than $B_j$ and $B_{j+1}$, all pairs of consecutive blocks in $F$ cover at least five elements.
Since $t = 4i+6$, this implies that there are at least $2i+2$ blocks to the right of $B_{j+1}$ or to the left of $B_j$.
These $2i$ blocks cover at least $5i+5$ elements, which implies that either $\varphi_2(B_j)$ is in $F$ or $\psi_2(B_{j+1})$ is in $F$.
Thus $\nu^*(F) \geq 3$, contradicting our assumption that $\nu^*(F) < 3$.

\case{\label{case:kp33mod5}$k = 5i+3$, $t = 4i+3$, and $c = 3$.}
Consider a frame $F$.
If $F$ contains no pair of consecutive 2-blocks, then by Claim~\ref{claim:no3333} we have $\nu^*(F) \geq 3$.
Suppose $F$ contains a pair $B_j$, $B_{j+1}$ of consecutive 2-blocks.
If $F$ contains at least one of $\varphi_2(B_j)$ or $\psi_2(B_j)$, then $\nu^*(F) \geq 3$.

Observe that if the rule (S2a) pulls charge from $F$, then $F$ loses at most 3 units of charge.
Moreover, the definition of (S2a) implies that there exists a block $B$ in $F$ of size at least 7, and thus $\nu^*(F) \geq 6$, so $\nu'(F) \geq 3$.

Observe that if the rule (S2b) pulls charge from $F$, then $F$ loses at most 3 units of charge.
Moreover, the definition of (S2b) implies that there exists blocks $B$ and $B'$ in $F$, each of size at least 6, and thus $\nu^*(F) \geq 6$, so $\nu'(F) \geq 3$.

Thus, suppose $\nu^*(F) < 3$, so all blocks in $F$ have size at most three.
If $F$ contains a triple of consecutive 2-blocks, then the rule (S2a) applies and $\nu'(F) = 3$.
If $F$ contains two pairs of consecutive 2-blocks, then the rule (S2b) applies and $\nu'(F) = 3$.
Thus, assume that there exists no triple of consecutive 2-blocks and exactly one pair $B_j, B_{j+1}$ of consecutive 2-blocks.
Then, other than $B_j$ and $B_{j+1}$, all pairs of consecutive blocks in $F$ cover at least five elements.
Since $t = 4i+3$, this implies that there are at least $2i+1$ blocks to the right of $B_{j+1}$ or to the left of $B_j$.
These $2i$ blocks cover at least $5i+2$ elements, which implies that either $\varphi_2(B_j)$ is in $F$ or $\psi_2(B_{j+1})$ is in $F$.
Thus $\nu^*(F) \geq 3$, contradicting our assumption that $\nu^*(F) < 3$.

\begin{casecomment}
For the remaining cases, we do not use the rules (S2a) or (S2b), but use the following discharging rule among the frames.

\begin{dischargingrule}
(S2c) Let $B_jB_{j+1}$ be the first pair of consecutive 2-blocks in a frame $F$.
Then $F$ pulls charge $\max\{0, c-\nu^*(F)\}$ from $g_j(F)$.
\end{dischargingrule}
Before we consider the remaining cases for the residue class of $k$, we prove a small claim.

\begin{claim}\label{claim:consec2blocks}
If $c \leq 3$ and a frame $F$ with $\nu^*(F) < 6$ loses charge by (S2c), then $F$ does not contain a pair of consecutive 2-blocks.
\end{claim}

\begin{proof}
If $F$ loses charge by (S2c), then let $F'$ be the frame with $g_j(F') = F$.
Since (S2c) applies to $F'$, it must be that $\nu^*(F') < 2$.
Thus there are no blocks of size at least four in $F'$.
Let $B_{j}$ and $B_{j+1}$ be the left-most pair of consecutive 2-blocks in $F'$, and $\varphi_2(B_j)$ appears in the same position in $F$ as $B_j$ appears in $F'$.
Moreover, since $\nu^*(F) < 6$, we have that $\varphi_2(B_j)$ is the only block in $F$ of size at least four.
Suppose that $B_i$ and $B_{i+1}$ are a pair of consecutive 2-blocks in $F$; select them to be closest to $\varphi_2(B_j)$.
Let $\ell$ be the number of blocks between $\varphi_2(B_j)$ and $B_i, B_{i+1}$.
These $\ell$ blocks between $\varphi_2(B_j)$ and $B_i, B_{i+1}$ are 2- and 3-blocks with no consecutive 2-blocks.
Thus, there are no consecutive 2-blocks between $B_j, B_{j+1}$ and $\psi_2(B_i)$, and hence there are at most $\ell$ blocks between $B_j, B_{j+1}$ and $\psi_2(B_i)$.
Thus, $\psi_2(B_i)$ is in $F'$ and hence $\nu^*(F') \geq 3$ and (S2c) pulls no charge from $F$ to $F'$.
\end{proof}
\end{casecomment}

\case{$k = 5i+2$, $t = 2i+2$, and $c = 2$.}
Consider a frame $F$.

\begin{subcases}
\subcase{$\nu^*(F) < 2$.}
If $F$ contains two consecutive 2-blocks, then $F$ pulls charge $3-\nu^*(F)$ by (S2c) and thus $\nu'(F) \geq 3$.
Claim~\ref{claim:no333} implies that if $F$ has no two consecutive 2-blocks, then $\nu^*(F) \geq 2$.

\subcase{$\nu^*(F) = 2$.}
The frame $F$ does not lose charge by (S2c), thus $\nu'(F) = 2$.

\subcase{$\nu^*(F) = 3$.}
If (S2) pulls at most one unit of charge from $F$, then $\nu'(F) = 2$, and $F$ retains enough charge.
Thus, if $F$ does lose charge by (S2c), then $F$ contains a 5-block $B$ with $|\varphi_2^{-1}(B)| = 2$, so $\mu^*(B) = 3$.
Thus all other blocks in $F$ are 2- and 3-blocks, and the 3-blocks are adjacent to 2-blocks on both sides; by Claim~\ref{claim:consec2blocks} there are no consecutive 2-blocks in $F$.
Thus, assume that (S2c) pulls two units of charge from $F$.
By the definition of $g_j$, the frame $F'$ such that $g_j(F') = F$ contains the first block of $\varphi_2^{-1}(B)$ at the same position in $F'$ as $B$ in $F$.
Since $\nu^*(F') = 0$ and $\nu^*(F) = 3$, the blocks following $\varphi_2^{-1}(B)$ in $F'$ and the blocks preceding $B$ in $F$ are 3- and 2-blocks, alternating.
Since $t = 2i+2$, the frames $F'$ and $F$ overlap in at least one block, thus the $i+1$ blocks to the left of $B$ have block structure $(2\ 3)^i\ 2$, which covers $5i+2$ elements, and $5i+2 =k \in S$, a contradiction.
Thus, it is impossible to have a frame $F$ with $\nu^*(F) = 3$ and $\nu'(F) = 1$ and hence $\nu'(F) \geq 2$.

\subcase{$\nu^*(F) = 4$.}
The rule (S2c) can only apply to $F$ at most once and pull at most $2$ units of charge from $F$.
Thus $\nu'(F) \geq 2$.
\end{subcases}

\case{$k = 5i$, $t = 2i+1$, and $c = 3$.}
Consider a frame $F$.

\begin{subcases}
\subcase{$\nu^*(F) < 3$.}
If $F$ contains two consecutive 2-blocks, then $F$ pulls charge $3-\nu^*(F)$ by (S2) and thus $\nu'(F) \geq 3$.

\subcase{$\nu^*(F) = 3$.}
If $F$ does not lose charge by rule (S2c), then $\nu'(F) = 3$.
If $F$ does lose charge by rule (S2c), then $F$ contains exactly one 5-block and none of the 3-blocks  in $F$ are heavy; by Claim~\ref{claim:consec2blocks} there are no consecutive 2-blocks in $F$.
Thus $F$ has block structure $2^e\ (3\ 2)^p\ 5\ (2\ 3)^q\ 2^f$, where $e, f\in \{0,1\}$ and $p,q\geq 0$ such that $e + f + 2p + 2q = t-1 = 2i$.
Thus $e = f$ and $p + q \geq i-1$.
If $e = f = 1$, then $p+q = i-1$ and the blocks with structure $(3\ 2)^p\ 5\ (2\ 3)^q$ cover exactly $5i$ elements, a contradiction.
If $e = f = 0$, then $p + q = i$ and $\sigma(F) = 5i+5$, but $F$ starts or ends with two blocks covering 5 elements, so ignoring these two blocks reveals consecutive blocks covering exactly $5i$ elements, a contradiction.

\subcase{$\nu^*(F) = 4$.}
If $F$ does not lose charge by rule (S2c), then $\nu'(F) = 4$.
Otherwise, $F$ contains exactly one 5-block and exactly one heavy 3-block; by Claim~\ref{claim:consec2blocks} there are no consecutive 2-blocks in $F$.
Thus $F$ has block structure $2^e\ (3\ 2)^p\ 5\ (2\ 3)^q 3\ (2\ 3)^r\ 2^f$ or $2^e\ (3\ 2)^p\ 3\ (3\ 2)^q 5\ (2\ 3)^r\ 2^f$, where $e, f\in \{0,1\}$ and $p,q,r \geq 0$ such that $e + f + 2p + 2q + 2r = t-2 = 2i-1$.
Thus $e+f=1$ and $p + q + r = i-1$.
This implies that $\sigma(F) = 5i+5$ and either the 5-block is the first or last block, or either the first two blocks or last two blocks cover five elements.
In any case, there exists a set of consecutive blocks covering exactly $5i$ elements, a contradiction.

\subcase{$\nu^*(F) = 5$.}
If $F$ does not lose charge by rule (S2c), then $\nu'(F) = 5$.
Otherwise, $F$ contains exactly one 5-block and exactly two heavy 3-blocks; by Claim~\ref{claim:consec2blocks} there are no consecutive 2-blocks in $F$.
Thus, there exist integers $e, f\in \{0,1\}$ and $p,q,r,s\geq 0$ such that $e + f + 2p + 2q + 2r + 2s = t-3 = 2i-2$ where $F$ has block structure among the following:
\[2^e\ (3\ 2)^p\ 3\ (3\ 2)^q\ 3\ (3\ 2)^r\ 5\ (2\ 3)^s\ 2^f,\]
\[2^e\ (3\ 2)^p\ 3\ (3\ 2)^q\ 5\ (3\ 2)^r\ 3\ (2\ 3)^s\ 2^f,\]
\[2^e\ (3\ 2)^p\ 5\ (3\ 2)^q\ 3\ (3\ 2)^r\ 3\ (2\ 3)^s\ 2^f.\]
Observe that $e = f$.
If $e = f = 1$, then $\sigma(F) = 5i+5$, and removing either the first two blocks or the last two blocks will result in a set of consecutive blocks spanning $5i$ elements, a contradiction.
Thus $e =f = 0$ and $\sigma(F) =  5i+6$, but one of the first or last blocks is a 3-block, so removing this block will result in a set of consecutive blocks spanning $5i+3$ elements; $5i+3 = k+3 \in S$, a contradiction.

\subcase{$\nu^*(F) \geq 6$.}
The rule (S2) can only apply to $F$ at most once and pull at most $3$ units of charge from $F$.
Thus $\nu'(F) \geq 3$.
\end{subcases}
\end{mycases}
In every case, we verified the hypothesis of the Local Discharging Lemma, and therefore the stated upper bounds hold.
\end{proof}



Based on computed values of $\dalpha(S)$ for similar sets $S = \{1, k, k+2j+1\}$ as shown in Figures~\ref{fig:1-k-kp5} and \ref{fig:1-k-kp7}, we form the following conjectures.

\begin{center}
\begin{minipage}{3in}{
\begin{conjecture}
	Let $k \geq 6$ and $k \notin \{7, 12\}$. Then
	\[
		\dalpha(\{1,k,k+5\}) = \begin{cases}
\frac{3k+14}{7k+42} & k \equiv 0 \pmod{7}\\
\frac{3}{7} & k \equiv 1 \pmod{7}\\
\frac{3k+15}{7k+42} & k \equiv 2 \pmod{7}\\
\frac{6k+10}{14k+35} & k \equiv 3 \pmod{7}\\
\frac{3k+16}{7k+42} & k \equiv 4 \pmod{7}\\
\frac{3k+13}{7k+42} & k \equiv 5 \pmod{7}\\
\frac{3k+17}{7k+42} & k \equiv 6 \pmod{7}.
		\end{cases}\]
\end{conjecture}
}
\end{minipage}
\qquad
\begin{minipage}{3in}
\begin{figure}[H]\centering
\begin{lpic}[]{"figures/plot-1-k-kp5"(3in,)}
\lbl[r]{0,73;\footnotesize$\dalpha$}
\lbl[]{5,126;\scriptsize$\frac{3}{7}$}
\lbl[]{5,95;\scriptsize$\frac{4}{10}$}
\lbl[]{5,22;\scriptsize$\frac{1}{3}$}

\lbl[t]{95,0;\footnotesize $k$}
\lbl[t]{29,5;\scriptsize $10$}
\lbl[t]{48,5;\scriptsize $20$}
\lbl[t]{66,5;\scriptsize $30$}
\lbl[t]{85,5;\scriptsize $40$}
\lbl[t]{104,5;\scriptsize $50$}
\lbl[t]{122.5,5;\scriptsize $60$}
\lbl[t]{141,5;\scriptsize $70$}
\lbl[t]{159.5,5;\scriptsize $80$}
\lbl[t]{178,5;\scriptsize $90$}
\end{lpic}
\caption{\label{fig:1-k-kp5}Computed values of $\dalpha(\{1,k,k+5\})$.}
\end{figure}
\end{minipage}
\end{center}
\begin{center}
\begin{minipage}{3in}{
\begin{conjecture}
Let $k \geq 8$ and\\ $k \notin\{ 9,11,16,18,25 \}$. Then
\[
\dalpha(\{1,k,k+7\}) = \begin{cases}
\frac{4k+27}{9k+72} & k \equiv 0 \pmod{9}\\\
\frac{4}{9} & k \equiv 1 \pmod{9}\\
\frac{4k+28}{9k+72} & k \equiv 2 \pmod{9}\\
\frac{8k+21}{18k+63} & k \equiv 3 \pmod{9}\\
\frac{4k+29}{9k+72} & k \equiv 4 \pmod{9}\\
\frac{4k-2}{9k+9} & k \equiv 5 \pmod{9}\\
\frac{4k+30}{9k+72} & k \equiv 6 \pmod{9}\\
\frac{4k+26}{9k+72} & k \equiv 7 \pmod{9}\\
\frac{4k+31}{9k+72} & k \equiv 8 \pmod{9}.
\end{cases}
\]
\end{conjecture}
}
\end{minipage}
\qquad
\begin{minipage}{3in}
\begin{figure}[H]
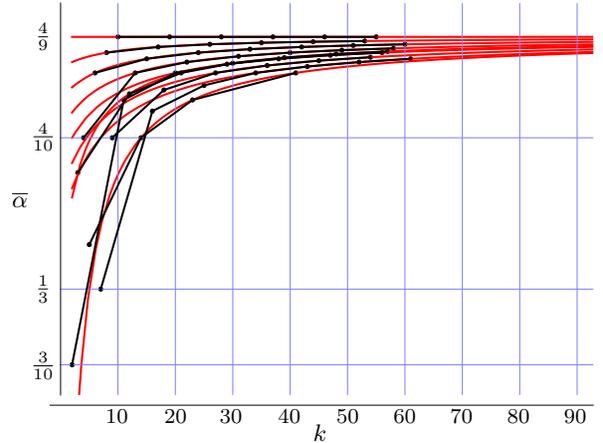
\centering
\begin{lpic}[]{"figures/plot-1-k-kp7"(3in,)}
\lbl[r]{0,73;\footnotesize$\dalpha$}
\lbl[]{5,126;\scriptsize$\frac{4}{9}$}
\lbl[]{5,93.5;\scriptsize$\frac{4}{10}$}
\lbl[]{5,44;\scriptsize$\frac{1}{3}$}
\lbl[]{5,19;\scriptsize$\frac{3}{10}$}

\lbl[t]{95,0;\footnotesize $k$}
\lbl[t]{29,5;\scriptsize $10$}
\lbl[t]{48,5;\scriptsize $20$}
\lbl[t]{66,5;\scriptsize $30$}
\lbl[t]{85,5;\scriptsize $40$}
\lbl[t]{104,5;\scriptsize $50$}
\lbl[t]{122.5,5;\scriptsize $60$}
\lbl[t]{141,5;\scriptsize $70$}
\lbl[t]{159.5,5;\scriptsize $80$}
\lbl[t]{178,5;\scriptsize $90$}
\end{lpic}
\caption{\label{fig:1-k-kp7}Computed values of $\dalpha(\{1,k,k+7\})$.}
\end{figure}
\end{minipage}
\end{center}


\section{Computational Methods}
\label{sec:computation}

We obtained the values of $\dalpha(S)$ given in the theorems and conjectures of Section~\ref{sec:threegens} by computing the independence ratio and looking for patterns.  For a given family of sets of generators parameterized by $k$, we computed $\alpha(S)$ for enough fixed sets $S$ in the family until we had enough data to conjecture a formula for $\dalpha(S)$ in terms of $k$.  As $G(S)$ is an infinite graph and computing the independence number of a graph is in general difficult, we describe here our approach.

To compute $\dalpha(S)$ for a fixed set $S$, we recall that for integers $n$ and $m$, we have the inequalities
\[
	\frac{\alpha(G(n,S))}{n} \leq \dalpha(S) \leq \frac{\alpha(G(S)[m])}{m},
\]
where $G(S)[m]$ is the subgraph of $G(S)$ induced on the interval $[m]$.
Thus, we will find maximum independent sets in $G(n,S)$ and $G(S)[m]$ for $n$ and $m$ growing until the largest lower bound matches the smallest upper bound.

Finding independent sets in a graph $G$ is equivalent to finding cliques in the complement of $G$.
Ba{\v{s}}i{\'c} and Ili{\'c}~\cite{BICliqueCirculant,IB10} previously computed some clique numbers and chromatic numbers for certain classes of circulant graphs using a backtracking search.
In~\cite{IB10}, they used Niskanen and \"Osterg\r{a}rd's \textsl{cliquer}~\cite{cliquer} as part of their implementation, but gave no other details.
We use a slight modification of the \textsl{cliquer} algorithm to compute lower and upper bounds on $\dalpha(S)$.

The \textsl{cliquer} algorithm greatly depends on the ordering of the vertices of the input graph.
For $G(n,S)$ and for $G(S)[n]$, we will use the ordering $1, 2, \dots, n$ in order to exploit the vertex-transitivity of $G(n,S)$ and $G(S)$, respectively.
We will focus first on the distance subgraphs $G(S)[n]$; a similar algorithm can be applied to the circulant graphs $G(n,S)$.
Define $\alpha(n)$ to be the largest size of an independent set in $G(S)[n]$.
Observe that for each $i\in [n]$, the subgraph $G(S)[i]$ is a subgraph of $G(S)[n]$.
Thus, we will compute the values of $\alpha(i)$ in increasing order of $i$, and use previous values in our later computation.
Also observe that $\alpha(i) \leq \alpha(i+1)\leq \alpha(i) + 1$.
Thus, in order to compute $\alpha(i+1)$, we must only search for an independent set of size $\alpha(i) + 1$.
We can terminate the search once one is found.

We use a recursive, backtracking search where we attempt to construct a large independent set $A$ in $G(S)[n]$ in \emph{decreasing} order.
Initialize $A = \varnothing$ and $B = [n]$.
At every step, we are given sets $A$ and $B$, where $A$ is an independent set and $B$ consists of the vertices $b$ such that $b < \min A$ and $b$ is not adjacent to any vertex in $A$.
Thus, the vertices in $B$ are possible next choices for growing the independent set $A$.
If $|A| > \alpha(n-1)$, then we have determined $\alpha(n)$, we report the set $A$, and we terminate the algorithm.
If $|B| + |A| \leq \alpha(n-1)$, then there is no independent set $A' \supset A$ with size at least $\alpha(n-1)+1$, and we can backtrack.
If these simple termination conditions fail, we attempt to add a new element to $A$ from $B$, but use our previous calculations of $\alpha(i)$ to assist.
The standard use of $\alpha(i)$ given by the \textsl{cliquer} algorithm is to check if $\alpha(b) + |A|$ is at least the size of our goal independent set size.
We also use the structure of $G(S)$ to our advantage for an additional pruning mechanism.

Let $A$ be an independent set in $G(S)[n]$, and let $B$ be a set of vertices that are not adjacent to any vertices in $A$ with $\max B < \min A$.
Represent $B$ as disjoint intervals $[x_i,y_i]$ where $B = \union_{i=1}^t [x_i,y_i]$ and define $\beta(B) = \sum_{i=1}^t \alpha(y_i-x_i+1)$.
If there is an independent set $A'$ with $A \subseteq A' \subseteq A \cup B$, then $A' \cap [x_i,y_i]$ is also an independent set.
Further, $(A'\cap [x_i,y_i]) - x_i$ is an independent set in $G(S)[y_i-x_i+1]$ by the vertex transitivity of $G(S)$.
Thus, $|A'| \leq |A| + \sum_{i=1}^t \alpha(y_i-x_i+1) = |A| + \beta(B)$.
Therefore, if $|A| + \beta(B)$ is below our target size of an independent set, we can backtrack.

Algorithm~\ref{alg:cliquerinterval} defines the recursive algorithm FindIndependentSet($\alpha, n, S, A, B$) to find the largest size of an independent set $A'$ in $G(S)[n]$ with $A \subseteq A' \subseteq A \cup B \subseteq [n]$ and $|A'| > \alpha(n-1)$.
To compute $\alpha(n)$, call FindIndependentSet($\alpha, n, S, \varnothing, [n]$) to initialize the recursive algorithm.

\def\Null{\operatorname{Null}}
\begin{algorithm}[htp]
\caption{\label{alg:cliquerinterval} FindIndependentSet($\alpha, n, S, A, B$)}
\begin{algorithmic}
\IF{$|A| > \alpha(n-1)$}
	\STATE $\alpha(n) \leftarrow |A|$
	\RETURN $A$
\ELSIF{$|A|+\beta(B) \leq \alpha(n-1)$}
	\RETURN $\Null$
\ENDIF

\FOR{\textbf{all} $b \in B$ in decreasing order}
	\IF{$|A|+\alpha(b) \leq \alpha(n-1)$}
		\RETURN $\Null$
	\ENDIF
	\STATE $A' \leftarrow A \cup \{b\}$
	\STATE $B' \leftarrow (B \cap [b-1]) - N(b)$
	\STATE $A'' \leftarrow$ FindIndependentSet($\alpha, n, S, A', B'$)
	\IF{$A''\neq \Null$}
		\RETURN $A''$
	\ENDIF
\ENDFOR
\RETURN $\Null$
\end{algorithmic}
\end{algorithm}

Define $\alpha(n,i)$ to be the largest size of an independent set in the circulant graph $G(n,S)$ using only vertices in $\{1,\dots,i\}$.
We can define a similar algorithm, FindIndependentSet($\alpha, n, i, S, A, B$), that computes $\alpha(n,i)$ for $i \in [n]$.
In order to determine $\alpha(G(n,S))$, we compute all values $\alpha(n,i)$ for $i \in [n]$ in increasing order.

Note that for a fixed set $S$, it may be less work to compute $\alpha(n)$ than to compute $\alpha(n,n)$ as $n$ increases, since $\alpha(n')$ for $n'<n$ may be used in the computation of $\alpha(n)$, but $\alpha(n',i)$ is not helpful for computing $\alpha(n,i)$.
However, early computations suggested that the value of $n$ such that $\alpha(n,n)/n = \dalpha(S)$ is much smaller than the value $m$ such that $\dalpha(S) = \alpha(m)/m$.
Thus, we organized our computation as follows: for every $n \geq 1$, compute $\alpha(2n-1)$ and $\alpha(2n)$ and if $n > \max S$ then compute $\alpha(n,n)$.
We terminated our computation when the lower and upper bounds matched.

Our implementation and all computation data are available online\footnote{Code and data are available at \url{http://www.math.iastate.edu/dstolee/r/distance.htm} and \url{http://www.github.com/derrickstolee/DistanceGraphs}.}.
The computed values of $\dalpha(\{1,1+k, 1+k+i\})$ are given as a table in Appendix A.


\clearpage
\appendix
\begin{landscape}
\section{Table of Computed Values $\dalpha(\{1,1+k,1+k+i\})$}
\scalebox{0.5}{
\input{tables/table-appendix2.tex}
}
\begin{center}
\small
\textbf{Table.} Bounds on $\dalpha(S)$ for $S = \{ 1, 1+k, 1+k+i\}$.
Black values are exact values, while red values are lower bounds only.
\end{center}
\end{landscape}

\end{document}

%% file: tables/table-appendix2.tex
\begin{tabular}[h]{r|c@{\ }c@{\ }c@{\ }c@{\ }c@{\ }c@{\ }c@{\ }c@{\ }c@{\ }c@{\ }c@{\ }c@{\ }c@{\ }c@{\ }c@{\ }c@{\ }c@{\ }c@{\ }c@{\ }c@{\ }c@{\ }c@{\ }c@{\ }c@{\ }c@{\ }c@{\ }c@{\ }c@{\ }c@{\ }c@{\ }c@{\ }c@{\ }c@{\ }c@{\ }c@{\ }c@{\ }c@{\ }c@{\ }c@{\ }c@{\ }}
 $k \setminus i$ & $1$& $2$& $3$& $4$& $5$& $6$& $7$& $8$& $9$& $10$& $11$& $12$& $13$& $14$& $15$& $16$& $17$& $18$& $19$& $20$& $21$& $22$& $23$& $24$& $25$& $26$& $27$& $28$& $29$& $30$& $31$& $32$& $33$& $34$& $35$& $36$& $37$& $38$& $39$& $40$\\\hline
$1$ & { 1/ 4} & { 1/ 3} & { 1/ 3} & { 2/ 7} & { 1/ 3} & { 1/ 3} & { 3/10} & { 1/ 3} & { 1/ 3} & { 4/13} & { 1/ 3} & { 1/ 3} & { 5/16} & { 1/ 3} & { 1/ 3} & { 6/19} & { 1/ 3} & { 1/ 3} & { 7/22} & { 1/ 3} & { 1/ 3} & { 8/25} & { 1/ 3} & { 1/ 3} & { 9/28} & { 1/ 3} & { 1/ 3} & {10/31} & { 1/ 3} & { 1/ 3} & {11/34} & { 1/ 3} & { 1/ 3} & {12/37} & { 1/ 3} & { 1/ 3} & {13/40} & { 1/ 3} & { 1/ 3} & {14/43} \\\hline
$2$ & { 2/ 7} && { 1/ 3} && { 4/11} && { 5/13} && { 2/ 5} && { 7/17} && { 8/19} && { 3/ 7} && {10/23} && {11/25} && { 4/ 9} && {13/29} && {14/31} && { 5/11} && {16/35} && {17/37} && { 6/13} && {19/41} && {20/43} && { 7/15} &\\\hline
$3$ & { 1/ 3} & { 2/ 5} & { 3/ 8} & { 1/ 3} & { 2/ 5} & { 4/11} & { 2/ 5} & { 5/13} & { 5/14} & { 2/ 5} & { 3/ 8} & { 2/ 5} & { 7/18} & { 7/19} & { 2/ 5} & { 8/21} & { 2/ 5} & { 9/23} & { 3/ 8} & { 2/ 5} & { 5/13} & { 2/ 5} & {11/28} & {11/29} & { 2/ 5} & {12/31} & { 2/ 5} & {13/33} & {13/34} & { 2/ 5} & { 7/18} & { 2/ 5} & {15/38} & { 5/13} & { 2/ 5} & {16/41} & { 2/ 5} & {17/43} & {17/44} & { 2/ 5} \\\hline
$4$ & { 2/ 7} && { 1/ 3} && { 1/ 3} && { 6/17} && { 7/19} && { 8/21} && { 9/23} && { 2/ 5} && {11/27} && {12/29} && {13/31} && {14/33} && { 3/ 7} && {16/37} && {17/39} && {18/41} && {19/43} && { 4/ 9} && {21/47} && {22/49} &\\\hline
$5$ & { 4/13} & { 3/ 7} & { 2/ 5} & { 3/ 8} & { 5/12} & { 1/ 3} & { 3/ 7} & { 2/ 5} & { 3/ 7} & { 7/17} & { 9/23} & { 8/19} & { 2/ 5} & { 3/ 7} & { 9/22} & { 3/ 7} & { 5/12} & { 2/ 5} & {11/26} & {11/27} & { 3/ 7} & {12/29} & { 3/ 7} & {13/31} & {13/32} & {14/33} & { 7/17} & { 3/ 7} & { 5/12} & { 3/ 7} & { 8/19} & {16/39} & {17/40} & {17/41} & { 3/ 7} & {18/43} & { 3/ 7} & {19/45} & {\color{red} 19/46} & {20/47} \\\hline
$6$ & { 1/ 3} && { 4/11} && { 3/ 8} && { 1/ 3} && { 6/17} && { 7/19} && { 3/ 8} && {11/29} && {12/31} && {13/33} && { 2/ 5} && {15/37} && {16/39} && {17/41} && {18/43} && {19/45} && {20/47} && { 3/ 7} && {22/51} && {23/53} &\\\hline
$7$ & { 3/10} & { 4/ 9} & { 7/19} & { 2/ 5} & { 3/ 7} & { 4/11} & { 7/16} & { 1/ 3} & { 4/ 9} & { 5/13} & { 4/ 9} & { 3/ 7} & {12/29} & {10/23} & {12/31} & {11/25} & {13/33} & { 4/ 9} & { 3/ 7} & { 4/ 9} & {13/30} & { 8/19} & { 7/16} & { 2/ 5} & {15/34} & { 3/ 7} & { 4/ 9} & {16/37} & { 4/ 9} & {17/39} & {20/47} & {18/41} & { 3/ 7} & {19/43} & {19/44} & { 4/ 9} & {10/23} & { 4/ 9} & { 7/16} & { 3/ 7} \\\hline
$8$ & { 6/19} && { 5/13} && { 2/ 5} && { 2/ 5} && { 1/ 3} && { 8/21} && { 9/23} && { 2/ 5} && { 2/ 5} && {11/29} && {12/31} && {13/33} && { 2/ 5} && { 2/ 5} && {19/47} && {20/49} && { 7/17} && {22/53} && {23/55} && { 8/19} &\\\hline
$9$ & { 1/ 3} & { 5/11} & { 5/14} & { 5/12} & { 2/ 5} & { 5/13} & { 4/ 9} & { 4/11} & { 9/20} & { 1/ 3} & { 5/11} & { 3/ 8} & { 5/11} & { 7/17} & { 3/ 7} & { 4/ 9} & {15/37} & {13/29} & { 5/13} & {14/31} & {16/41} & { 5/11} & {18/43} & { 5/11} & { 4/ 9} & {10/23} & {17/38} & { 5/12} & { 9/20} & { 2/ 5} & {19/42} & {11/26} & { 5/11} & { 4/ 9} & { 5/11} & {21/47} & {25/57} & {22/49} & {25/59} & {23/51} \\\hline
$10$ & { 4/13} && { 2/ 5} && { 7/17} && { 5/12} && {12/31} && { 1/ 3} && { 2/ 5} && {11/27} && {12/29} && { 5/12} && {17/43} && { 2/ 5} && {15/37} && {16/39} && {17/41} && { 5/12} && { 2/ 5} && {\color{red} 19/47} && {\color{red} 20/49} && {\color{red}  7/17} &\\\hline
$11$ & { 8/25} & { 6/13} & { 3/ 8} & { 3/ 7} & { 5/13} & { 2/ 5} & {13/31} & { 8/21} & { 5/11} & { 5/13} & {11/24} & { 1/ 3} & { 6/13} & {10/27} & { 6/13} & { 2/ 5} & {18/41} & { 3/ 7} & {18/43} & { 5/11} & { 2/ 5} & {16/35} & { 5/13} & {17/37} & {19/49} & { 6/13} & { 7/17} & { 6/13} & {23/53} & { 4/ 9} & { 5/11} & { 3/ 7} & {21/46} & {12/29} & {11/24} & { 2/ 5} & {23/50} & {13/31} & { 6/13} & { 7/16} \\\hline
$12$ & { 1/ 3} && {11/29} && { 8/19} && { 3/ 7} && { 3/ 7} && {14/37} && { 1/ 3} && {16/41} && {13/31} && {14/33} && { 3/ 7} && { 3/ 7} && {20/51} && {21/53} && {18/43} && {19/45} && {20/47} && { 3/ 7} && { 3/ 7} && { 2/ 5} &\\\hline
$13$ & { 5/16} & { 7/15} & { 7/18} & { 7/16} & { 2/ 5} & { 7/17} & { 2/ 5} & { 9/23} & {16/37} & { 2/ 5} & { 6/13} & { 2/ 5} & {13/28} & { 1/ 3} & { 7/15} & {12/31} & { 7/15} & {13/33} & {21/47} & { 5/12} & { 3/ 7} & {11/25} & { 7/17} & { 6/13} & { 2/ 5} & {19/41} & { 2/ 5} & {20/43} & {\color{red} 17/44} & { 7/15} & {24/59} & { 7/15} & {26/61} & {14/31} & { 4/ 9} & { 7/16} & { 6/13} & {14/33} & {25/54} & { 7/17} \\\hline
$14$ & {10/31} && { 7/19} && { 3/ 7} && {10/23} && { 7/16} && {17/41} && { 3/ 8} && { 1/ 3} && {18/47} && { 3/ 7} && {16/37} && {17/39} && { 7/16} && { 8/19} && {23/59} && {24/61} && { 3/ 7} && {22/51} && {23/53} && {24/55} &\\\hline
$15$ & { 1/ 3} & { 8/17} & { 2/ 5} & { 4/ 9} & { 9/22} & { 8/19} & { 7/17} & { 2/ 5} & {17/41} & {11/27} & {19/43} & { 7/17} & { 7/15} & { 9/23} & {15/32} & { 1/ 3} & { 8/17} & { 2/ 5} & { 8/17} & {15/37} & {24/53} & {\color{red} 16/39} & {24/55} & { 3/ 7} & { 8/19} & {13/29} & { 9/22} & { 7/15} & { 7/17} & {22/47} & {25/63} & {23/49} & { 2/ 5} & { 8/17} & {\color{red} 21/52} & { 8/17} & {29/69} & {16/35} & {31/71} & { 4/ 9} \\\hline
$16$ & { 6/19} && { 8/21} && {16/39} && {11/25} && { 4/ 9} && { 4/ 9} && {19/47} && { 7/18} && { 1/ 3} && {\color{red} 14/37} && {23/55} && {18/41} && {19/43} && { 4/ 9} && { 4/ 9} && {27/65} && {\color{red}  7/18} && { 9/23} && {30/71} && {25/57} &\\\hline
$17$ & {12/37} & { 9/19} & { 5/13} & { 9/20} & { 5/12} & { 3/ 7} & { 8/19} & { 9/22} & { 2/ 5} & {12/29} & {20/47} & {13/31} & {22/49} & { 8/19} & { 8/17} & { 5/13} & {17/36} & { 1/ 3} & { 9/19} & {11/28} & { 9/19} & {17/41} & {27/59} & {18/43} & {27/61} & { 8/19} & { 3/ 7} & { 7/16} & {\color{red}  5/12} & { 5/11} & {21/50} & { 8/17} & { 8/19} & {25/53} & {28/71} & {26/55} & {29/73} & { 9/19} & {12/29} & { 9/19} \\\hline
$18$ & { 1/ 3} && { 9/23} && { 2/ 5} && { 4/ 9} && {13/29} && { 9/20} && {22/51} && { 2/ 5} && { 2/ 5} && { 1/ 3} && {16/41} && {25/61} && { 4/ 9} && {21/47} && {22/49} && { 9/20} && {31/71} && {30/73} && { 2/ 5} && {\color{red} 23/59} &\\\hline
$19$ & {\color{red}  7/22} & {10/21} & { 3/ 8} & { 5/11} & {11/26} & {10/23} & { 3/ 7} & { 5/12} & { 3/ 7} & { 2/ 5} & { 7/17} & {14/33} & {23/53} & { 3/ 7} & { 5/11} & { 3/ 7} & { 9/19} & { 8/21} & {19/40} & { 1/ 3} & {10/21} & {12/31} & {10/21} & {19/45} & { 6/13} & {20/47} & {30/67} & { 3/ 7} & {10/23} & { 3/ 7} & {\color{red} 11/26} & { 4/ 9} & {23/54} & {17/37} & { 3/ 7} & { 9/19} & { 3/ 7} & {28/59} & {31/79} & {29/61} \\\hline
$20$ & {14/43} && { 2/ 5} && {11/27} && { 3/ 7} && {14/31} && { 5/11} && { 5/11} && { 8/19} && { 9/22} && {24/61} && { 1/ 3} && { 2/ 5} && {\color{red} 19/47} && {10/23} && {23/51} && {24/53} && { 5/11} && { 5/11} && {34/79} && { 9/22} &\\\hline
$21$ & { 1/ 3} & {11/23} & { 5/13} & {11/24} & { 3/ 7} & {11/25} & {13/30} & {11/26} & {10/23} & {11/27} & { 2/ 5} & { 3/ 7} & { 8/19} & {16/37} & {26/59} & {10/23} & {28/61} & {13/31} & {10/21} & { 9/23} & {21/44} & { 1/ 3} & {11/23} & {\color{red} 18/47} & {11/23} & { 3/ 7} & {33/71} & {22/51} & {33/73} & {23/53} & {11/25} & {10/23} & { 3/ 7} & {17/39} & {25/58} & { 9/20} & {13/30} & {19/41} & {10/23} & {10/21} \\\hline
$22$ & {\color{red}  8/25} && {19/49} && {12/29} && { 5/12} && { 5/11} && {16/35} && {11/24} && {27/61} && {\color{red} 17/41} && { 5/12} && {26/67} && { 1/ 3} && {28/71} && { 7/17} && {32/75} && { 5/11} && {26/57} && {27/59} && {11/24} && {38/85} &\\\hline
$23$ & {16/49} & {12/25} & {11/28} & { 6/13} & {22/53} & { 4/ 9} & { 7/16} & { 3/ 7} & {11/25} & {12/29} & {25/59} & { 2/ 5} & {25/61} & {17/39} & { 3/ 7} & {18/41} & {29/65} & {11/25} & {31/67} & { 7/17} & {11/23} & { 2/ 5} & {23/48} & { 1/ 3} & {12/25} & {20/51} & {12/25} & { 8/19} & {36/77} & {24/55} & {36/79} & {25/57} & { 4/ 9} & {11/25} & {36/83} & { 3/ 7} & {27/62} & {19/43} & { 7/16} & { 5/11} \\\hline
$24$ & { 1/ 3} && {\color{red} 11/29} && {13/31} && {11/26} && {26/59} && {17/37} && { 6/13} && { 6/13} && {29/67} && {19/45} && {11/26} && {\color{red}  5/13} && { 1/ 3} && {30/77} && {23/55} && {\color{red}  8/19} && {37/83} && {28/61} && {29/63} && { 6/13} &\\\hline
$25$ & {\color{red}  9/28} & {13/27} & { 2/ 5} & {13/28} & {13/32} & {13/29} & {15/34} & {13/30} & { 4/ 9} & {13/31} & { 4/ 9} & {\color{red} 11/27} & { 2/ 5} & {14/33} & {28/67} & {19/43} & {10/23} & { 4/ 9} & {32/71} & { 4/ 9} & {34/73} & {\color{red} 11/27} & {12/25} & {15/38} & {25/52} & { 1/ 3} & {13/27} & { 2/ 5} & {13/27} & {17/41} & {39/83} & {26/59} & {39/85} & {27/61} & {13/29} & { 4/ 9} & {39/89} & { 4/ 9} & { 3/ 7} & {10/23} \\\hline
$26$ & {18/55} && {12/31} && {14/33} && { 3/ 7} && { 3/ 7} && { 6/13} && {19/41} && {13/28} && {32/71} && {\color{red} 20/47} && { 3/ 7} && { 3/ 7} && {11/28} && {\color{red} 22/67} && {\color{red} 22/57} && {25/59} && {26/61} && {39/89} && { 6/13} && {31/67} &\\\hline
$27$ & { 1/ 3} & {14/29} & {23/59} & { 7/15} & { 7/17} & {14/31} & { 4/ 9} & { 7/16} & {17/38} & {14/33} & {13/29} & {12/29} & {29/69} & { 2/ 5} & {\color{red}  9/22} & { 4/ 9} & {31/73} & {21/47} & {11/25} & {13/29} & { 5/11} & {17/39} & {37/79} & {12/29} & {13/27} & {16/41} & {27/56} & { 1/ 3} & {14/29} & {17/43} & {14/29} & {\color{red} 25/61} & {42/89} & { 4/ 9} & { 6/13} & {29/65} & {14/31} & {30/67} & {42/95} & {13/29} \\\hline
$28$ & {\color{red} 10/31} && {13/33} && { 3/ 7} && {16/37} && {13/30} && {31/69} && {20/43} && { 7/15} && { 7/15} && {34/77} && {22/51} && {13/30} && {35/83} && { 2/ 5} && { 1/ 3} && {24/61} && { 3/ 7} && {28/65} && {\color{red} 29/67} && {44/97} &\\\hline
$29$ & {20/61} & {15/31} & {\color{red} 13/34} & {15/32} & { 5/12} & { 5/11} & {29/67} & {15/34} & { 9/20} & { 3/ 7} & {14/31} & {\color{red} 18/43} & {32/73} & {13/31} & { 2/ 5} & { 8/19} & {\color{red}  5/12} & {22/49} & {34/79} & {23/51} & { 4/ 9} & {14/31} & {38/83} & { 3/ 7} & { 8/17} & {13/31} & {14/29} & {\color{red} 12/31} & {29/60} & {\color{red}  1/99} & {15/31} & { 9/23} & {15/31} & {27/65} & { 9/19} & { 7/16} & {45/97} & {31/69} & { 5/11} & {32/71} \\\hline
$30$ & { 1/ 3} && { 2/ 5} && {28/67} && {17/39} && { 7/16} && {32/73} && { 7/15} && {22/47} && {15/32} && {37/81} && {\color{red} 23/53} && {24/55} && { 7/16} && {37/89} && {36/91} && {\color{red}  1/99} && { 2/ 5} && {41/97} && {10/23} && {\color{red} 31/71} &\\\hline
$31$ & {\color{red} 11/34} & {16/33} & { 7/18} & { 8/17} & { 8/19} & {16/35} & {\color{red} 14/33} & { 4/ 9} & {19/42} & {16/37} & { 5/11} & {\color{red} 19/45} & { 5/11} & {14/33} & {33/79} & { 2/ 5} & {\color{red} 21/50} & {18/41} & {\color{red} 11/26} & {24/53} & {37/85} & { 5/11} & {13/29} & { 5/11} & {41/89} & {\color{red} 25/59} & {43/91} & {14/33} & {15/31} & {13/33} & {31/64} & {\color{red}  1/99} & {16/33} & {\color{red} 26/67} & {16/33} & {29/69} & {29/61} & {22/51} & {41/88} & {33/73} \\\hline
$32$ & {22/67} && { 9/23} && {\color{red} 16/39} && {18/41} && {15/34} && { 3/ 7} && {36/79} && {23/49} && { 8/17} && { 8/17} && {13/29} && {25/57} && {26/59} && {15/34} && {\color{red}  7/17} && {38/97} && {\color{red}  1/99} && {\color{red}  9/23} && {38/91} && {32/73} &\\\hline
$33$ & { 1/ 3} & {17/35} & {15/38} & {17/36} & {17/40} & {17/37} & { 3/ 7} & {17/38} & { 5/11} & {17/39} & {21/46} & {\color{red} 20/47} & {16/35} & { 3/ 7} & {36/83} & { 3/ 7} & { 2/ 5} & {18/43} & {\color{red} 23/54} & { 5/11} & { 3/ 7} & {26/57} & {40/91} & {16/35} & {14/31} & {21/47} & {44/95} & { 3/ 7} & {46/97} & { 3/ 7} & {16/33} & { 2/ 5} & {33/68} & {\color{red}  1/99} & {17/35} & {\color{red} 28/71} & {17/35} & {31/73} & {41/86} & {\color{red} 32/75} \\\hline
$34$ & {\color{red} 12/37} && {\color{red}  5/13} && {17/41} && {19/43} && { 4/ 9} && { 4/ 9} && {37/83} && { 8/17} && {25/53} && {17/36} && { 6/13} && {41/93} && {27/61} && { 4/ 9} && { 4/ 9} && { 5/12} && {\color{red}  7/18} && {\color{red}  1/99} && {\color{red}  7/18} && {\color{red} 31/75} &\\\hline
$35$ & {24/73} & {18/37} & { 2/ 5} & { 9/19} & { 3/ 7} & { 6/13} & {19/44} & { 9/20} & { 4/ 9} & {18/41} & {11/24} & { 3/ 7} & {17/37} & {22/51} & {13/29} & {16/37} & {37/89} & { 2/ 5} & { 3/ 7} & {10/23} & {\color{red} 25/58} & {27/59} & {\color{red} 16/37} & {28/61} & {43/97} & {17/37} & { 5/11} & {11/25} & {27/58} & {16/37} & {39/82} & {\color{red} 40/97} & {17/35} & {\color{red} 22/57} & {35/72} & {\color{red}  1/99} & {18/37} & { 2/ 5} & {18/37} & { 3/ 7} \\\hline
$36$ & { 1/ 3} && {16/41} && {18/43} && { 4/ 9} && {21/47} && {17/38} && {38/87} && {41/89} && {26/55} && { 9/19} && { 9/19} && {44/97} && { 4/ 9} && {29/65} && {17/38} && {39/89} && { 8/19} && {15/38} && {\color{red}  1/99} && {\color{red} 30/77} &\\\hline
$37$ & {\color{red} 13/40} & {19/39} & {31/79} & {19/40} & {34/81} & {19/41} & {10/23} & {19/42} & {\color{red} 17/39} & {19/43} & {23/50} & {19/44} & { 6/13} & {23/53} & { 6/13} & {17/39} & {40/93} & {20/47} & { 2/ 5} & { 5/12} & {13/30} & {22/49} & {\color{red} 27/62} & {29/63} & {\color{red} 17/39} & { 6/13} & {25/56} & { 6/13} & {16/35} & {\color{red} 10/23} & {43/92} & {17/39} & {31/65} & {\color{red} 16/39} & {18/37} & {\color{red}  5/13} & {37/76} & {\color{red}  1/99} & {19/39} & {\color{red}  1/99} \\\hline
$38$ & {26/79} && {17/43} && {19/45} && {37/85} && {22/49} && { 9/20} && { 3/ 7} && {14/31} && { 9/19} && {28/59} && {19/40} && {27/58} && {25/56} && {30/67} && {31/69} && { 9/20} && {\color{red}  3/ 7} && {17/40} && {\color{red} 25/77} && {\color{red}  1/99} &\\\hline
$39$ & { 1/ 3} & {20/41} & {\color{red} 17/44} & {10/21} & {\color{red} 19/46} & {20/43} & { 7/16} & { 5/11} & {18/41} & { 4/ 9} & { 6/13} & {10/23} & {25/54} & {24/55} & {19/41} & {25/57} & {43/97} & {18/41} & {\color{red} 17/41} & { 2/ 5} & {23/54} & {22/51} & { 7/16} & { 6/13} & {18/41} & {31/67} & {29/66} & {19/41} & {40/89} & { 5/11} & {17/37} & {32/73} & {38/81} & {18/41} & {11/23} & {\color{red} 17/41} & {19/39} & {\color{red}  1/99} & {39/80} & {\color{red}  1/99} \\\hline
$40$ & {\color{red} 14/43} && { 2/ 5} && {20/47} && { 3/ 7} && {23/51} && {19/42} && {42/95} && {43/97} && {46/99} && {29/61} && {10/21} && {10/21} && {38/83} && {37/84} && {32/71} && {33/73} && {19/42} && { 3/ 7} && { 3/ 7} && {\color{red}  1/99} &\\\hline
$41$ & {28/85} & {21/43} & { 9/23} & {21/44} & { 5/12} & { 7/15} & {11/25} & {21/46} & {19/43} & {21/47} & {43/95} & { 7/16} & {13/28} & { 3/ 7} & {20/43} & {26/59} & {41/90} & {19/43} & {41/96} & {11/26} & { 2/ 5} & {\color{red} 27/65} & {29/66} & { 4/ 9} & {15/34} & {32/69} & {19/43} & {33/71} & {27/61} & {20/43} & {33/73} & {\color{red} 11/25} & { 6/13} & {34/77} & {\color{red} 37/80} & {19/43} & {35/73} & {18/43} & {20/41} & {\color{red}  1/99} \\\hline
$42$ & { 1/ 3} && {35/89} && { 3/ 7} && {22/51} && {24/53} && { 5/11} && { 5/11} && {27/62} && {26/57} && {10/21} && {31/65} && {21/44} && {37/79} && {37/82} && {43/97} && {34/75} && { 5/11} && { 5/11} && {19/44} && {\color{red} 27/65} &\\\hline
$43$ & {\color{red} 15/46} & {22/45} & {19/48} & {11/23} & {21/50} & {22/47} & {23/52} & {11/24} & { 4/ 9} & {22/49} & { 4/ 9} & {11/25} & {27/58} & {22/51} & { 7/15} & {27/61} & { 7/15} & { 4/ 9} & {29/66} & { 4/ 9} & {19/45} & { 2/ 5} & {36/85} & { 3/ 7} & {31/70} & {26/57} & { 4/ 9} & {34/73} & { 4/ 9} & { 7/15} & {41/92} & { 7/15} & {\color{red}  9/20} & {\color{red} 35/79} & {19/41} & { 4/ 9} & {\color{red} 13/28} & { 4/ 9} & {\color{red} 10/21} & {19/45} \\\hline
$44$ & {30/91} && {\color{red} 19/49} && { 8/19} && {23/53} && { 5/11} && {26/57} && {21/46} && { 3/ 7} && {35/78} && {29/62} && {32/67} && {11/23} && {11/23} && {\color{red} 43/94} && {\color{red}  7/16} && { 5/11} && {36/79} && {21/46} && {\color{red} 10/23} && {10/23} &\\\hline
$45$ & { 1/ 3} & {23/47} & { 2/ 5} & {23/48} & {11/26} & {23/49} & { 4/ 9} & {23/50} & {25/56} & {23/51} & {21/47} & {23/52} & { 7/15} & {23/53} & {29/62} & { 4/ 9} & {22/47} & {29/65} & {41/91} & {21/47} & {\color{red} 20/47} & { 8/19} & { 2/ 5} & {\color{red} 30/71} & { 4/ 9} & {26/59} & {33/74} & { 7/15} & {21/47} & {36/77} & {18/41} & {22/47} & {43/96} & {\color{red} 21/47} & {\color{red} 19/42} & {37/83} & {20/43} & {21/47} & {\color{red} 41/88} & {\color{red} 28/65} \\\hline
$46$ & {\color{red} 16/49} && {20/51} && {\color{red} 22/53} && {24/55} && {25/56} && {27/59} && {11/24} && {37/84} && {34/77} && {29/63} && {11/23} && {34/71} && {23/48} && {41/87} && {\color{red} 22/49} && {11/25} && {37/81} && {38/83} && {11/24} && {\color{red} 37/84} &\\\hline
$47$ & {\color{red}  8/25} & {24/49} & {13/33} & {12/25} & {23/54} & { 8/17} & {38/87} & { 6/13} & {13/29} & {24/53} & {22/49} & { 4/ 9} & {39/85} & {24/55} & {15/32} & { 3/ 7} & {23/49} & {30/67} & {41/89} & {22/49} & {17/39} & {26/59} & { 3/ 7} & {\color{red}  1/99} & {\color{red} 23/56} & {\color{red} 32/75} & {17/38} & {\color{red} 34/77} & {35/78} & {37/79} & {22/49} & {38/81} & {19/43} & {23/49} & {\color{red} 39/88} & {\color{red} 38/85} & {\color{red}  5/11} & {13/29} & { 7/15} & {22/49} \\\hline
$48$ & { 1/ 3} && {21/53} && {23/55} && {25/57} && {\color{red} 11/25} && {28/61} && {23/50} && {37/82} && {10/23} && {29/64} && { 8/17} && {35/73} && {12/25} && {12/25} && {\color{red} 41/89} && {\color{red} 23/52} && {\color{red} 29/66} && {39/85} && {40/87} && {23/50} &\\\hline
$49$ & {\color{red} 17/52} & {25/51} & {\color{red}  7/18} & {25/52} & { 3/ 7} & {25/53} & {\color{red} 25/58} & {25/54} & { 9/20} & { 5/11} & {23/51} & {25/56} & {37/82} & {25/57} & {31/66} & {\color{red} 22/51} & { 8/17} & {31/69} & { 8/17} & {32/71} & {29/65} & {23/51} & {\color{red} 22/51} & {13/31} & {\color{red}  1/99} & { 3/ 7} & {41/93} & { 7/16} & { 9/20} & {\color{red} 37/81} & {23/51} & {39/83} & {43/97} & { 8/17} & { 4/ 9} & { 8/17} & {\color{red} 41/92} & {\color{red} 40/89} & {\color{red} 21/46} & {41/91} \\\hline
$50$ & {\color{red} 17/53} && { 2/ 5} && { 8/19} && {26/59} && {23/52} && {29/63} && { 6/13} && { 6/13} && {\color{red} 30/71} && {29/65} && {\color{red} 45/98} && {12/25} && {37/77} && {25/52} && {\color{red} 43/93} && {\color{red} 24/53} && { 4/ 9} && {\color{red} 13/29} && {41/89} && { 6/13} &\\\hline
\end{tabular}